\pdfoutput=1

\documentclass[letterpaper,10pt]{article}
\usepackage{hyperref}
\usepackage{amsmath,amssymb,amsthm,amsopn}
\usepackage{verbatim}
\usepackage{fullpage}
\usepackage{graphicx}
\usepackage{mathtools}

\usepackage{tikz}
\usetikzlibrary{topaths,calc,decorations.pathmorphing,quotes,positioning,fit,shapes}

\hypersetup{
    colorlinks=true,
    linkcolor=blue,
    filecolor=magenta,      
    urlcolor=cyan,      
    citecolor=blue,
    breaklinks=true,
    hyperindex=true,
    linktocpage=false,
}

\DeclareMathOperator{\Clo}{Clo}

\DeclareMathOperator{\Sg}{Sg}

\DeclareMathOperator{\Var}{Var}

\begin{document}

\makeatletter
\newtheorem*{rep@theorem}{\rep@title}
\newcommand{\newreptheorem}[2]{%
\newenvironment{rep#1}[1]{%
 \def\rep@title{#2 \ref{##1}}%
 \begin{rep@theorem}}%
 {\end{rep@theorem}}}
\makeatother

\newtheorem{thm}{Theorem}
\newreptheorem{thm}{Theorem}
\newtheorem{prop}{Proposition}
\newreptheorem{prop}{Proposition}
\newtheorem{cor}{Corollary}
\newtheorem{lem}{Lemma}
\newreptheorem{lem}{Lemma}

\theoremstyle{definition}
\newtheorem{defn}{Definition}
\newtheorem{conj}{Conjecture}

\theoremstyle{remark}
\newtheorem{rem}{Remark}

\newcommand{\Rho}{\mathrm{P}}
\newcommand{\cS}{\mathcal{S}}
\newcommand{\cM}{\mathcal{M}}
\newcommand{\cC}{\mathcal{C}}
\newcommand{\cF}{\mathcal{F}}
\newcommand{\cG}{\mathcal{G}}
\newcommand{\cP}{\mathcal{P}}
\newcommand{\cV}{\mathcal{V}}
\newcommand{\cD}{\mathcal{D}}
\newcommand{\RR}{\mathbb{R}}
\newcommand{\ZZ}{\mathbb{Z}}
\newcommand{\NN}{\mathbb{N}}
\newcommand{\bA}{\mathbb{A}}
\newcommand{\bB}{\mathbb{B}}
\newcommand{\bC}{\mathbb{C}}
\newcommand{\bD}{\mathbb{D}}
\newcommand{\bE}{\mathbb{E}}
\newcommand{\bF}{\mathbb{F}}
\newcommand{\bI}{\mathbb{I}}
\newcommand{\bS}{\mathbb{S}}
\newcommand{\fA}{\mathbf{A}}
\newcommand{\fB}{\mathbf{B}}
\newcommand{\gk}{\kappa}
\newcommand{\gS}{\Sigma}
\newcommand{\gl}{\lambda}
\newcommand{\gt}{\theta}

\newcommand{\dotcup}{\ensuremath{\mathaccent\cdot\cup}}

\title{Coarse classification of binary minimal clones}
\date{}
\author{Zarathustra Brady}
%\institute{Department of Mathematics, Massachusetts Institute of Technology, USA}
%\address{Department of Mathematics \\ Massachusetts Institute of Technology \\ 77 Massachusetts Avenue \\ Building 2, Room 350B\\ Cambridge, MA 02139-7307}
%\email{notzeb@gmail.com}

\maketitle

\begin{abstract} We classify binary minimal clones into seven categories: affine algebras, rectangular bands, $p$-cyclic groupoids, spirals, non-Taylor partial semilattices, melds, and dispersive algebras. Each category has nice enough properties to guarantee that any clone of one of these types contains a minimal clone of the same type.
\end{abstract}

%TODO: remove table of contents once finished
\tableofcontents

\section{Introduction}

The classification of minimal clones on a finite set is a very old problem, studied by many authors. For a survey of previous results, the reader is directed to the papers of Quackenbush \cite{minimal-clones-survey} and Cs\'ak\'any \cite{minimal-clones-minicourse}. First we review a few of the basic definitions and results. All algebras which occur throughout this paper will be assumed to be finite unless explicitly stated otherwise.

\begin{defn} A clone $\cC$ is \emph{minimal} if $f \in \cC$ nontrivial implies $\cC = \Clo(f)$. An algebra $\bA$ will be called \emph{clone-minimal} (equivalently: $\Clo(\bA)$ is a \emph{minimal clone}) if $\bA$ has no nontrivial proper reduct.
\end{defn}

\begin{prop} If $\Clo(f)$ is minimal and $g \in \Clo(f)$ nontrivial, then $f \in \Clo(g)$.
\end{prop}

\begin{defn}\label{set-defn} $\bA$ is called a \emph{projection algebra} if all of its operations are projections. Otherwise, we say that $\bA$ is \emph{nontrivial} (note that this is a nonstandard definition of what it means for an algebra to be nontrivial).
\end{defn}

\begin{prop} If $\Clo(\bA)$ is minimal and $\bB \in \Var(\bA)$ nontrivial, then $\Clo(\bB)$ is minimal.
\end{prop}

\begin{thm}[Rosenberg \cite{five-types}]\label{rosenberg-types} Suppose that $\bA = (A,f)$ is a finite clone-minimal algebra, and $f$ has minimal arity among nontrivial elements of $\Clo(\bA)$. Then one of the following is true:
\begin{enumerate}
\item $f$ is a unary operation which is either a permutation of prime order or satisfies $f(f(x)) \approx f(x)$,
\item $f$ is ternary, and $\bA$ is the idempotent reduct of a vector space over $\bF_2$,
\item $f$ is a ternary majority operation,
\item $f$ is a semiprojection of arity at least $3$,
\item $f$ is an idempotent binary operation.
\end{enumerate}
\end{thm}

Rosenberg's classification is not fully satisfactory: it can be very difficult to check whether a given majority operation, semiprojection, or idempotent binary operation generates a minimal clone. This paper is mainly concerned with the binary case. Previous authors have proven classifications of binary minimal clones under additional assumptions, such as the assumption that the binary operation is \emph{entropic} \cite{minimal-clones-entropic}, or the assumption that the associated algebra is \emph{weakly abelian} \cite{minimal-clones-weakly-abelian}, or the assumption that the number of binary operations in the clone is small \cite{minimal-clones-binary-few}. In this paper, we consider the general case of binary minimal clones.

\begin{defn} We say a property $\cP$ of functions $f$, or algebras $(A,f)$, is \emph{nice} if it satisfies the following conditions:
\begin{itemize}
\item Given $f$ as input, we can verify in polynomial time whether $f$ has property $\cP$,
\item If $f$ has property $\cP$ and $g \in \Clo(f)$ is nontrivial, then there is a nontrivial $f' \in \Clo(g)$ such that $f'$ has property $\cP$.
\item There exists a fixed nontrivial algebra $\bA_\cP$ such that for any nontrivial algebra $\bA = (A,f)$ where $f$ has the property $\cP$, the algebra $\bA_\cP$ is in the variety generated by $\bA$.
\item The collection of finite algebras $(A,f)$ such that $f$ has the property $\cP$ is closed under finite products, homomorphic images, and subalgebras, that is, it forms a pseudovariety.
\end{itemize}
\end{defn}

The first three cases in Rosenberg's classification are given by nice properties (with the caveat that in the case of unary permutations of prime order, we actually get an infinite family of nice properties indexed by the primes). As an example, we'll check that being a ternary majority operation is a nice property. The first, third, and fourth conditions are straightforward to verify. For the second condition, we use the following result from \cite{minimal-majority-few-terms}, which is used in that paper to simplify the study of minimal majority clones.

\begin{prop}[Waldhauser \cite{minimal-majority-few-terms}] If $f$ is a majority operation and $g \in \Clo(f)$ is nontrivial, then $g$ is a near-unanimity operation. In this case, $g$ has a majority term as an identification minor.
\end{prop}
\begin{proof} The proof is by induction on the construction of $g$ in terms of $f$. If $g = f(g_1,g_2,g_3)$, then by the inductive hypothesis each $g_i$ is either a near-unanimity term or a projection. If two of the $g_i$s are equal to the same projection, then so is $g$. Otherwise, if we identify all but one of the inputs to $g$ with $x$, then by the inductive hypothesis at least two of $g_1,g_2,g_3$ evaluate to $x$, so $g$ is a near-unanimity operation.

Since any near-unanimity operation is not a semiprojection, by {\'S}wierczkowski's Lemma from \cite{semiprojection-lemma} if $g$ is not already majority we can identify two of the inputs to $g$ to get a nontrivial $g' \in \Clo(g)$ of smaller arity. Thus $g$ has a majority term as an identification minor.
\end{proof}

The fact that being a majority operation is a nice property implies that in order to check whether a majority clone $\Clo(f)$ is minimal, one only needs to enumerate the ternary majority operations $g \in \Clo(f)$ and check that $f \in \Clo(g)$ for each one. While this may be difficult, we are at the very least assured that if $f$ is a majority operation, then $\Clo(f)$ contains \emph{some} minimal majority clone (in fact every minimal clone contained in $\Clo(f)$ will be a majority clone).

On the other hand, if one is given a binary idempotent operation $f$, then it can be difficult to rule out the possibility that $\Clo(f)$ might contain a semiprojection of very large arity. As a result, checking whether a binary operation generates a minimal clone could in principle be an enormous undertaking. The goal of the present paper is to provide a \emph{coarse classification} of binary minimal clones, that is, a list $\cP_1, \cP_2, ...$ of nice properties such that every binary minimal clone contains an operation satisfying exactly one of the nice properties $\cP_i$.

\begin{comment}
\[
t\left(\begin{bmatrix} x & ? & \cdots & ? \\ ? & x & \cdots & ? \\ \vdots & \vdots & \ddots & \vdots \\ ? & ? & \cdots & x \end{bmatrix}\right) \approx t\left(\begin{bmatrix} y & ? & \cdots & ? \\ ? & y & \cdots & ? \\ \vdots & \vdots & \ddots & \vdots \\ ? & ? & \cdots & y \end{bmatrix}\right),
\]
\end{comment}
We briefly recall a few standard results about Taylor algebras. A Taylor operation is an idempotent operation which satisfies a collection of height 1 identities which cannot be simultaneously satisfied by any particular projection operation. More concretely, an operation $t$ is Taylor if it satisfies a system of identities of the form
\begin{align*}
t(x, ?, ..., ?) &\approx t(y, ?, ..., ?),\\
t(?, x, ..., ?) &\approx t(?, y, ..., ?),\\
&\cdots\\
t(?, ?, ..., x) &\approx t(?, ?, ..., y),
\end{align*}
where the $?$s are filled in somehow with $x$s and $y$s.

\begin{prop}[Proposition 4.14 of \cite{bulatov-jeavons-varieties}, Proposition 2.1 of \cite{maltsev-conditions-hs}, or Corollary 4.2.1 of \cite{zhuk-strong}] If an algebra $\bA$ is finite and idempotent, then $\bA$ has a Taylor term operation if and only if there is no two-element projection algebra in $HS(\bA)$, where $HS(\bA)$ is the collection of homomorphic images of subalgebras of $\bA$.
\end{prop}

As a starting point, a previous paper by the present author has given a coarse classification of minimal clones which have a Taylor operation, into just three cases. Two of these cases are defined by nice properties, while the remaining case (vector spaces over a prime field) is given in terms of an infinite family of nice properties, one for each prime $p$.

\begin{thm}[Clone-minimal Taylor algebras \cite{brady-examples}]\label{thm-taylor-minimal} Suppose $\bA$ is a finite algebra which is both clone-minimal and Taylor. Then one of the following is true:
\begin{enumerate}
\item $\bA$ is the idempotent reduct of a vector space over $\bF_p$ for some prime $p$,
\item $\bA$ is a majority algebra,
\item $\bA$ is a spiral.
\end{enumerate}
\end{thm}

The last case above was one of the author's main motivations for introducing the concept of nice properties. Spirals are defined below. We use the following standard notation: for $S$ a subset of an algebra $\bA$, the \emph{subalgebra generated} by $S$, written $\Sg_\bA(S)$, is the smallest subset of $\bA$ which contains $S$ and is closed under the operations of $\bA$.

\begin{defn}\label{defn-spiral} An algebra $\bA = (A,f)$ is a spiral if $f$ is binary, idempotent, commutative, and for any $a,b \in \bA$ either $\{a, b\}$ is a subalgebra of $\bA$, or $\Sg_\bA\{a, b\}$ has a surjective homomorphism to the free semilattice on two generators.
\end{defn}

The reader may find it instructive to check that being a spiral is a nice property. While this follows from the results of \cite{brady-examples}, it is simple enough to give a direct argument. For now, we just note that it is possible to test whether a given algebra $\bA = (A,f)$ is a spiral in time polynomial in $|A|$: for any $a,b$ with $\{a,b\}$ not a subalgebra of $\bA$, if $\Sg_\bA\{a,b\} = \{a,b\} \cup S$ with $a,b \not\in S$, then any surjective homomorphism from $\Sg_\bA\{a,b\}$ to the free semilattice on two generators $x,y$ must map $a$ to one generator, say $x$, map $b$ to the other generator $y$, and map every element of $S$ to $xy$. Thus, the existence of such a surjective homomorphism is equivalent to $S$ being a subalgebra of $\bA$ such that $f(a,S) \subseteq S$ and $f(b,S) \subseteq S$.

The main classification result of this paper is the following coarse classification of the non-Taylor binary minimal algebras. To save space when writing out identities, sometimes we will abbreviate the basic binary operation $f$ of the algebras we study with the infix `$\cdot$', or with simple juxtaposition, as is common in the study of groupoids.

\begin{thm}\label{main-thm} Suppose that $\bA = (A,f)$ is a binary clone-minimal algebra which is not Taylor. Then, after possibly replacing $f(x,y)$ by $f(y,x)$, one of the following is true:
\begin{enumerate}
	\item $\bA$ is a rectangular band, i.e. an idempotent groupoid satisfying $(xy)(zw) \approx xw$,
	\item $\bA$ is a $p$-cyclic groupoid for some prime $p$,
	\item there is a nontrivial $s \in \Clo(f)$ which is a ``partial semilattice operation'':
	\[
	s(x,s(x,y)) \approx s(s(x,y),x) \approx s(x,y),
	\]
	\item $\bA$ is an idempotent groupoid satisfying $(xy)(zx) \approx xy$ (a ``meld''),
	\item $\bA$ is a ``dispersive algebra'' (see Definition \ref{defn-dispersive}).
\end{enumerate}
Furthermore, each of these cases is defined by a nice property, other than the case of $p$-cyclic groupoids, which consists of an infinite family of nice properties indexed by the primes.
\end{thm}

Most of the cases appearing in the classification have been described by previous authors. Rectangular bands are well-known in the theory of bands. The structure of $p$-cyclic groupoids was described by P{\l}onka \cite{plonka-p-cyclic} (see also \cite{representations-n-cyclic-groupoids}). Partial semilattice operations were isolated by Bulatov \cite{colored-graph} in his study of colored graphs attached to finite Taylor algebras, but they were not named there or studied for their own sake (in the more recent paper \cite{bulatov-local-structure-I}, Bulatov refers to the partial semilattice identities as the ``Semilattice Shift Condition''). Melds were described by L\'evai and P\'alfy in Theorem 5.2(e) of \cite{minimal-clones-binary-few}, which classifies binary minimal clones having exactly four binary operations, but melds were not given a name there, and no structure theory for them was given.

The last case of the above theorem - the case of dispersive algebras - has a definition which has a similar flavour to the definition of spirals, but which is somewhat harder to work with in practice. The author has struggled for some time to give a structure theory for this case, but was ultimately unsuccessful. Several difficult conjectures about the structure of dispersive algebras are given at the end of the next section.

To define the dispersive algebras, we first define the variety $\cD$ of idempotent groupoids satisfying
\begin{align}\tag{$\cD 1$}
x(yx) \approx (xy)x \approx (xy)y \approx (xy)(yx) \approx xy,
\end{align}
\begin{align}\tag{$\cD 2$}
\forall n \ge 0 \;\;\; x(...((xy_1)y_2)\cdots y_n) \approx x.
\end{align}
This variety of minimal clones appears in L\'evai and P\'alfy \cite{minimal-clones-binary-few}, and the notation $\cD$ for this variety is from Waldhauser's thesis \cite{minimal-clones-waldhauser}. We use the notation $\cF_\cV(x,y)$ to denote the \emph{free algebra} on two generators in any variety $\cV$.

\begin{prop}[L\'evai, P\'alfy \cite{minimal-clones-binary-few}] If $\bA \in \cD$, then $\Clo(\bA)$ is a minimal clone. Also, $\cF_\cD(x,y)$ has exactly four elements: $x,y,xy,yx$.
\end{prop}

\begin{defn}\label{defn-dispersive} An idempotent groupoid $\bA$ is \emph{dispersive} if it satisfies \eqref{D2} and if for all $a,b \in \bA$, either $\{a,b\}$ is a two element subalgebra of $\bA$ or there is a surjective homomorphism
\[
\Sg_{\bA^2}\Big\{\begin{bmatrix} a\\ b\end{bmatrix}, \begin{bmatrix} b\\ a\end{bmatrix}\Big\} \twoheadrightarrow \cF_\cD(x,y).
\]
\end{defn}

The name ``dispersive'' for such algebras was chosen in order to reflect the fact that they satisfy very few \emph{absorption identities}. Absorption identities are a crucial tool in the study of minimal clones. (There is an unfortunate naming collision here: absorption identities have nothing to do with the theory of absorbing subalgebras from \cite{cyclic} which recently found many applications in the study of Taylor algebras.)

\begin{defn} An \emph{absorption identity} is an identity of the form
\[
t(x_1, ..., x_n) \approx x_i.
\]
\end{defn}

\begin{prop}[Kearnes \cite{minimal-clones-abelian}, L\'evai, P\'alfy \cite{minimal-clones-binary-few}]\label{absorb-prop} If $\bA$ is clone-minimal and $\bB \in \Var(\bA)$ is nontrivial, then any absorption identity which holds in $\bB$ must also hold in $\bA$.
\end{prop}
\begin{proof} Suppose that the absorption identity $t(x_1, ..., x_n) \approx x_i$ holds in $\bB$ but not in $\bA$. Then $t$ generates a nontrivial proper subclone of $\Clo(\bA)$: the reduct of $\bB$ with basic operation $t$ is a projection algebra, so $\Clo(t)$ must be proper, and $\Clo(t)$ is nontrivial on $\bA$ since (by assumption) $t$ cannot act as any projection on $\bA$.
\end{proof}

As a consequence of the third condition for being a nice property (i.e. that the fixed algebra $\bA_\cP$ can be found in $\Var(\bA)$), if $\cP$ is a nice property and $f$ generates a minimal clone and has property $\cP$, then the set of absorption identities that hold in $\bA = (A,f)$ is the same as the set of absorption identities that hold on $\bA_\cP$, which depends only on $\cP$. Thus the coarse classification of binary minimal clones given in this paper is also a classification of the possible collections of absorption identities which can be satisfied in a binary minimal clone.

In the case of partial semilattice operations, Corollary \ref{cor-absorption-semilattice} shows that no absorption identities hold at all (other than those following from idempotence), for the simple reason that every nontrivial partial semilattice contains a two-element semilattice subalgebra by Proposition \ref{prop-partial-semilattice}. In the case of melds, every absorption identity follows from idempotence and the identity
\[
x((yx)z) \approx x,
\]
and in fact Proposition \ref{prop-meld-absorption} shows that this identity is equivalent to the defining identity $(xy)(zx) \approx xy$ of melds (modulo idempotence). In the case of dispersive algebras, Proposition \ref{dispersive-absorption} shows that all absorption identities follow from idempotence and the absorption identity \eqref{D2}.

\bigskip

In light of the coarse classification of binary minimal clones provided in this paper, it is natural to ask whether a coarse classification of semiprojections can be found. Such a classification would, in a certain sense, complete the classification of minimal clones which was started by Rosenberg \cite{five-types}. The author has not made any attempt at classifying semiprojections, but the main result of \cite{minimal-clones-conservative} which classifies \emph{conservative} semiprojections indicates that the problem is not hopeless.

\bigskip

This paper is organized as follows. In Section \ref{s-concepts} we go over several basic techniques that are widely applicable throughout the paper, as well as several ways of visualizing small groupoids which the author has found to be more helpful than simply writing out operation tables. Many of the results in Section \ref{s-concepts} encapsulate various ways to take advantage of the finiteness of the relevant algebras. In Section \ref{s-descriptions} we go over each of the five types of non-Taylor binary minimal clone in more detail, proving basic results and bit of structure theory about each of them. We also verify that each of these types is described by a nice property in Section \ref{s-descriptions}. In Section \ref{s-classification}, which is the meat of the paper, we prove Theorem \ref{main-thm}.

Appendices \ref{a-partial-semi} and \ref{a-dispersive} prove some more technical results about clone-minimal partial semilattices and dispersive algebras which are mentioned in Section \ref{s-descriptions}, but which would have disrupted the flow of the paper due to their length. These results are not necessary for the proof of our main result (Theorem \ref{main-thm}), but they may be useful to any reader who wishes to use the results of this paper to assist in classifying binary minimal clones on a small domain.

\section{Concepts and techniques used throughout the paper}\label{s-concepts}

There are a few basic ideas which we will use repeatedly. One fact which we will frequently exploit is that in any algebra $\bA$, there is a bijection between the set $\Clo_n(\bA)$ of $n$-ary term operations of $\bA$ and the free algebra on $n$ generators $\cF_{\cV(A)}(x_1, ..., x_n)$ in the variety $\cV(\bA) = \Var(\bA)$, which sends the $i$th projection $\pi_i$ to $x_i$. We will often implicitly identify $\Clo_2(\bA)$ with the free algebra $\cF_{\cV(\bA)}(x,y)$ throughout the paper, identifying $\pi_1$ with $x$ and $\pi_2$ with $y$.

The following well-known iteration argument is one of the main ways in which we exploit finiteness.

\begin{prop}\label{unary-iteration} If $f$ is a unary operation on a finite set $A$, then for every $a \in A$ there is some $1 \le m_a \le |A|$ such that
\[
f^{km_a}(a) = f^{m_a}(a)
\]
for all $k \ge 1$. In particular, the sequence
\[
f^{1!}(a), f^{2!}(a), ..., f^{n!}(a), ...
\]
is eventually constant with eventual value $f^{|A|!}(a) = f^{m_a}(a)$, and if we define a unary operation $f^\infty$ by
\[
f^\infty(a) = \lim_{n \rightarrow \infty} f^{n!}(a) = f^{m_a}(a),
\]
where the limit is taken in the discrete topology on $A$, then $f^\infty = f^{|A|!} \in \Clo(f)$, the operation $f^\infty$ satisfies the identity
\[
f^\infty(f^\infty(x)) \approx f^\infty(x),
\]
and $f^\infty(A) \subseteq f(A)$, so $f^\infty$ is nontrivial if $f$ is not a permutation of $A$.
\end{prop}
\begin{proof} By the Pigeonhole Principle, for any given $a$ there are $1 \le i < j \le |A|+1$ such that $f^i(a) = f^j(a)$. We then have $f^{i+k}(a) = f^{j+k}(a)$ for all $k \ge 0$, so
\[
f^{k + j-i}(a) = f^k(a)
\]
as long as $k \ge i$. Taking $m_a$ to be the multiple of $j-i$ which is in the half-open interval $[i,j)$, we see that $1 \le m_a \le |A|$ and
\[
f^{km_a + m_a}(a) = f^{km_a}(a)
\]
for all $k \ge 1$, since $km_a \ge m_a \ge i$.

That the sequence $f^{n!}(a)$ is eventually equal to $f^{m_a}(a)$ follows from the fact that $n!$ is a multiple of $m_a$ for $n \ge m_a$. To check the identity $f^\infty(f^\infty(x)) \approx f^\infty(x)$, we note that for any particular $a \in A$ we have
\[
f^\infty(f^\infty(a)) = f^{|A|!}(f^{|A|!}(a)) = f^{2|A|!}(a) = f^{m_a}(a) = f^\infty(a).
\]

Finally, if $f$ is not a permutation of $A$, then by the finiteness of $A$ the image $f(A)$ must be a strict subset of $A$, and each $f^{n!}$ has image $f^{n!}(A) = f(f^{n!-1}(A)) \subseteq f(A)$, so $f^\infty = \lim_n f^{n!}$ has $f^\infty(A) \subseteq f(A) \ne A$.
\end{proof}

The same argument shows that if $f$ is an element of a finite semigroup, then the sequence $f^{1!}, f^{2!}, ...$ is eventually constant with eventual value equal to some idempotent element $e$, and in a monoid this $e$ will not be the identity iff $f$ is not invertible. In particular, any finite monoid which is not a group contains some non-identity idempotent elements. We will apply this fact to monoids built out of the free algebra $\cF_{\cV(\bA)}(x,y)$, which is finite as long as $\bA$ is finite.

There are three natural ways to compose binary functions in an associative way. The first way is based on viewing a binary function $f(x,y)$ as a family of unary functions $f_y(x)$ indexed by a parameter $y$. This leads us to the associative operation $*_1$ on $\cF_{\cV(\bA)}(x,y)$ given by
\begin{equation}
(f *_1 g)(x,y) = f(g(x,y), y).\label{star-1}\tag{$*_1$}
\end{equation}
Similarly, if we view $f$ as a family of unary functions indexed by a parameter $x$, we are led to define the associative operation $*_2$ on $\cF_{\cV(\bA)}(x,y)$ given by
\begin{equation}
(f *_2 g)(x,y) = f(x, g(x,y)).\label{star-2}\tag{$*_2$}
\end{equation}
The third associative operation comes from viewing $f$ as a function $\phi$ on pairs of elements, which takes the pair $(a,b)$ to the pair $(f(a,b), f(b,a))$. Note that any map on pairs $\phi : A^2 \rightarrow A^2$ which commutes with reversing the order of the pairs, i.e. such that
\[
\phi(a,b) = (c,d) \;\;\; \iff \;\;\; \phi(b,a) = (d,c),
\]
has the form
\[
\phi(x,y) = (f(x,y), f(y,x))
\]
for $f = \pi_1 \circ \phi$, and the collection of such maps $\phi$ is closed under composition. This leads us to the \emph{circular composition} operation $*_c$ on $\cF_{\cV(\bA)}(x,y)$ which is given by
\begin{equation}
(f *_c g)(x,y) = f(g(x,y), g(y,x)).\label{star-c}\tag{$*_c$}
\end{equation}

\begin{defn}\label{defn-iteration} For any binary operation $f$ on a finite set, we define the binary operations $f^{\infty_1}, f^{\infty_2}, f^{\infty_c} \in \Clo(f)$ by
\[
f^{\infty_1}(a,b) = \lim_{n \rightarrow \infty} \big(x \mapsto f(x,b)\big)^{n!}(a) = \lim_{n \rightarrow \infty} f^{*_1 n!}(a,b),
\]
similarly
\[
f^{\infty_2}(a,b) = \lim_{n \rightarrow \infty} \big(y \mapsto f(a,y)\big)^{n!}(b) = \lim_{n \rightarrow \infty} f^{*_2 n!}(a,b),
\]
and finally
\[
f^{\infty_c}(a,b) = \lim_{n \rightarrow \infty} \pi_1\Big(\big((x,y) \mapsto (f(x,y), f(y,x))\big)^{n!}(a,b)\Big) = \lim_{n \rightarrow \infty} f^{*_c n!}(a,b).
\]
\end{defn}

\begin{cor}\label{binary-iteration} For any binary operation $f$ on a finite set, the operations $f^{\infty_1}, f^{\infty_2}, f^{\infty_c}$ are well-defined elements of $\Clo(f)$ which satisfy the following identities:
\begin{align*}
f^{\infty_1}(f^{\infty_1}(x,y), y) &\approx f^{\infty_1}(x,y),\\
f^{\infty_2}(x, f^{\infty_2}(x,y)) &\approx f^{\infty_2}(x,y),\\
f^{\infty_c}(f^{\infty_c}(x,y), f^{\infty_c}(y,x)) &\approx f^{\infty_c}(x,y).
\end{align*}
Furthermore, the operation $f^{\infty_1}$ is first projection iff the map $x \mapsto f(x,y)$ is a permutation for all $y$, the operation $f^{\infty_2}$ is second projection iff the map $y \mapsto f(x,y)$ is a permutation for all $x$, and $f^{\infty_c}$ is first projection iff the map $(x,y) \mapsto (f(x,y), f(y,x))$ is a permutation on the collection of ordered pairs.
\end{cor}

Since we mainly study algebras $\bA = (A,f)$ which are not Taylor, it will be useful to know how $f$ acts on projection algebras $\bB$ in the variety generated by $\bA$. We will generally focus on functions $f$ which restrict to the first projection on such an algebra $\bB$.

\begin{defn} If $f$ is a binary operation, then we define $\Clo_2^{\pi_1}(f)$ to be the smallest set of binary operations $g$ which contains $\pi_1$ and contains $f(g(x,y), h(x,y))$ for all $g \in \Clo_2^{\pi_1}(f)$ and $h \in \Clo_2(f)$ (note that we do not require $h \in \Clo_2^{\pi_1}(f)$).
\end{defn}

An alternative characterization of $\Clo_2^{\pi_1}(f)$ is the set of operations $g(x,y)$ which can be written as a term in $f$ whose tree representation has an $x$ as its leftmost leaf. Note that the operations $f^{\infty_1}, f^{\infty_2}, f^{\infty_c}$ all belong to $\Clo_2^{\pi_1}(f)$, and that $\Clo_2^{\pi_1}(f)$ is closed under $*_1, *_2, *_c$.

\begin{prop}\label{prop-pi1} If a term operation $f$ of $\bA$ acts as first projection on some algebra $\bB \in \Var(\bA)$ of size at least two, then $\Clo_2^{\pi_1}(f)$ is exactly the set of binary operations $g \in \Clo_2(f)$ which also act as first projection on $\bB$.
\end{prop}
\begin{proof} We prove by induction on the definition of a term $g \in \Clo_2^{\pi_1}(f)$ that if $f$ is first projection on $\bB$ then so is $g$. The base case is $g = \pi_1$, which certainly acts as first projection on $\bB$. For the inductive step, if we have already proved that some $g \in \Clo_2^{\pi_1}(f)$ acts as first projection on $\bB$, then for any $h \in \Clo_2(f)$ we have
\[
f(g(x,y),h(x,y)) \approx g(x,y) \approx x
\]
on $\bB$.

Since $\pi_2(y,x) = \pi_1 \in \Clo_2^{\pi_1}(f)$, an inductive argument shows that for any $g \in \Clo_2(f)$ we either have $g(x,y) \in \Clo_2^{\pi_1}(f)$ or $g(y,x) \in \Clo_2^{\pi_1}(f)$. If $g(x,y) \not\in \Clo_2^{\pi_1}(f)$, then from $g(y,x) \in \Clo_2^{\pi_1}(f)$ we see that $g(y,x)$ acts as first projection on $\bB$, so $g(x,y)$ acts as second projection on $\bB$, and first and second projection are not the same on $\bB$ as long as $|\bB| \ge 2$.
\end{proof}

\begin{defn}\label{defn-clo2-pi1} For any algebra $\bA$, let $\Clo_2^{\pi_1}(\bA)$ be the set of binary terms of $\bA$ which restrict to the first projection on some algebra $\bB \in \Var(\bA)$ of size at least $2$.
\end{defn}

%Proposition \ref{prop-pi1} can be used to clarify the relationship between $\Clo_2^{\pi_1}(\bA)$ and $\Clo_2^{\pi_1}(f)$.

\begin{prop} If $f \in \Clo_2^{\pi_1}(\bA)$ and $g \in \Clo_2^{\pi_1}(f)$, then $g \in \Clo_2^{\pi_1}(\bA)$.
\end{prop}
\begin{proof} If $f$ acts as first projection on some $\bB \in \Var(\bA)$, then by Proposition \ref{prop-pi1} so does $g$.
\end{proof}

For operations $f \in \Clo_2^{\pi_1}(\bA)$, the following definition is convenient.

\begin{defn}\label{defn-right-orbit} If $\bA = (A,f)$ is an algebra with a binary operation $f$ and $a \in \bA$, then the \emph{right orbit} of $a$ (under $f$) is the smallest set $O(a)$ such that
\begin{itemize}
\item $a \in O(a)$, and

\item for $b \in O(a)$ and $c \in \bA$, we have $f(b,c) \in O(a)$.
\end{itemize}
Equivalently, we have $b \in O(a)$ iff there is a sequence of elements $c_1, ..., c_k \in \bA$ such that
\[
b = f(\cdots f(f(a,c_1),c_2)\cdots, c_k).
\]
We say that a set $O$ is a right orbit of $\bA$ if $O = O(a)$ for some $a \in \bA$.
\end{defn}

\begin{prop}\label{prop-right-pi1} If $\bA = (A,f)$ is generated by two elements $a, b \in \bA$, then the right orbit $O(a)$ is exactly the set of elements that can be written as $g(a,b)$ for some $g \in \Clo_2^{\pi_1}(f)$.
\end{prop}
\begin{proof} If $d \in O(a)$, then by the definition of $O(a)$ there are $c_1, ..., c_k \in \bA$ such that
\[
d = f(\cdots f(f(a,c_1),c_2)\cdots, c_k).
\]
Since $\bA$ is generated by $a$ and $b$, for each $c_i$ there is some binary term $t_i \in \Clo_2(f)$ such that $t_i(a,b) = c_i$. Defining $g$ by
\[
g(x,y) = f(\cdots f(f(x,t_1(x,y)),t_2(x,y))\cdots, t_k(x,y)),
\]
we see that $g \in \Clo_2^{\pi_1}(f)$ and $g(a,b) = a$. The converse follows directly from the definition of a right orbit.
\end{proof}

\begin{cor} If $\bA = (A,f)$, then $\Clo_2^{\pi_1}(f)$ can be canonically identified with the right orbit of $x$ in the free algebra on two generators $\cF_{\cV(\bA)}(x,y)$.
\end{cor}

\begin{prop}\label{prop-right-gen} If $\bA = (A,f)$ is generated by $a_1, ..., a_n$, then $A = \bigcup_i O(a_i)$.
\end{prop}
\begin{proof} Let $S = \bigcup_i O(a_i)$. Since each $a_i \in O(a_i) \subseteq S$, we just need to check that $S$ is closed under $f$. If $b,c \in S$, then $b \in O(a_i)$ for some $i$, so by the definition of a right orbit we have $f(b,c) \in O(a_i) \subseteq S$ as well.
\end{proof}

There is a useful way to visualize binary operations $f$ which makes it easy to keep track of the right orbits.

\begin{defn}\label{defn-digraph} If $\mathbb{A} = (A,f)$, we define the digraph $D_\mathbb{A}$ to have a directed edge from $a$ to $f(a,b)$ for all pairs $a,b \in \mathbb{A}$ with $f(a,b) \ne a$. We think of this as a labeled digraph, where the label of an edge $a \rightarrow c$ is the set of all elements $b$ such that $f(a,b) = c$.% Note that the right orbit of $a$ is exactly the set of elements which are reachable from $a$ in $D_\bA$.

We define the preorder $\preceq$ on $\bA$ by $a \preceq b$ iff $b \in O(a)$, and we consider $\preceq$ to be a partial order on the strongly connected components of $D_\bA$ in the usual way. (This is a special case of what is usually called the \emph{reachability} partial order on the strongly connected components of a digraph.)% We say that a subalgebra $\bB \subseteq \bA$ is an \emph{orbital subalgebra} of $\bA$ if for all $b \in \bB, a \in \bA,$ we have $f(b,a) \in \bB$.
\end{defn}

\begin{comment}
\begin{prop} Every right orbit of $\bA$ is an orbital subalgebra of $\bA$.
\end{prop}
\end{comment}

%, and the collection of orbital subalgebras of $\bA$
Note that while the labeled digraph $D_\bA$ is \emph{not} a clone-invariant of $\bA$ (that is, the labeled digraph depends on the choice of a particular binary operation $f$ rather than the set of all binary terms of $\bA$), the collection of right orbits $O(a)$ and the partial order $\preceq$ \emph{are} clone-invariants if $\bA$ is clone-minimal, at least if we stick to operations in $\Clo_2^{\pi_1}(f)$.

\begin{prop}\label{digraph-clone-inv} If $\bA = (A,f)$ is a clone-minimal algebra such that $f \in \Clo_2^{\pi_1}(\bA)$, then for any nontrivial $g \in \Clo_2^{\pi_1}(f)$, the digraphs $D_{(A,f)}$ and $D_{(A,g)}$ have the same collection of strongly connected components and the same partial ordering $\preceq$.% Equivalently, any orbital subalgebra of $(A,f)$ is also an orbital subalgebra of $(A,g)$.
\end{prop}
\begin{proof} First we check that for every nontrivial $g \in \Clo_2^{\pi_1}(f)$, we have $f \in \Clo_2^{\pi_1}(g)$. Since $\bA$ is clone-minimal, we automatically have $f \in \Clo_2(g)$. By the definition of $\Clo_2^{\pi_1}(\bA)$, $f$ acts as first projection on some $\bB \in \Var(\bA)$ of size at least two. By Proposition \ref{prop-pi1} we see that $g$ also acts as first projection on $\bB$, and then by Proposition \ref{prop-pi1} applied to $g$ we see that in fact we have $f \in \Clo_2^{\pi_1}(g)$.

Now if $b = f(a,c)$, then Proposition \ref{prop-right-pi1} applied to $(A,g)$ implies that $b$ is in the right orbit of $a$ with respect to $g$. Applying this repeatedly, we see that right orbits with respect to $f$ are the same as right orbits with respect to $g$.
\end{proof}

The next result is often useful when classifying minimal binary clones on a small domain.

\begin{prop}\label{nice-terms} For any finite idempotent algebra $\bA = (A,f)$, there are binary operations $g_1, g_2, g_3 \in \Clo_2^{\pi_1}(f)$ such that as elements of the free algebra $\cF_{\cV(\bA)}(x,y)$ we have $f(x,y) \preceq g_i(x,y)$ and each $g_i(x,y)$ is contained in a maximal strongly connected component of the digraph $D_{\cF_{\cV(\bA)}(x,y)}$, and such that the $g_i$ satisfy the identities
\begin{align*}
g_1(g_1(x,y),y) &\approx g_1(x,y),\\
g_2(g_2(x,y),g_2(y,x)) &\approx g_2(x,y),\\
g_3(g_3(x,y),x) &\approx g_3(x,y).
\end{align*}
\end{prop}
\begin{proof} First pick any $\preceq$-maximal $g_0(x,y) \in \cF_{\cV(\bA)}(x,y)$ with $f(x,y) \preceq g_0(x,y)$ (that such a $g_0$ exists follows from the finiteness of $\cF_{\cV(\bA)}(x,y)$). We will construct $g_1, g_2, g_3$ such that $g_i(x,y) \in O(g_0(x,y))$, which will guarantee that each $g_i(x,y)$ is also $\preceq$-maximal and satisfies $f(x,y) \preceq g_i(x,y)$.

The construction of $g_1$ and $g_2$ from $g_0$ is straightforward: we just take $g_1 = g_0^{\infty_1}$ and $g_2 = g_0^{\infty_c}$ and note that this construction ensures $g_0 \preceq g_1, g_2$ by Proposition \ref{prop-right-pi1} and the identities
\[
g_1(x,y) \approx g_0^{*_1(|A|!-1)}(g_0(x,y),y)
\]
and
\[
g_2(x,y) \approx g_0^{*_c(|A^2|!-1)}(g_0(x,y),g_0(y,x)).
\]
Note, however, that we do not generally have $g_0 \preceq g_0^{\infty_2}$.

In order to construct $g_3$, we set
\[
g_3(x,y) \coloneqq g_1(g_1(x,y),x).
\]
Then we have $g_1 \preceq g_3$ by Proposition \ref{prop-right-pi1}, and
\begin{align*}
g_3(g_3(x,y),x) &\approx g_1(g_1(g_3(x,y),x),g_3(x,y))\\
&\approx g_1(g_1(g_1(g_1(x,y),x),x),g_3(x,y))\\
&\approx g_1(g_1(g_1(x,y),x),g_3(x,y))\\
&\approx g_1(g_3(x,y),g_3(x,y))\\
&\approx g_3(x,y).\qedhere
\end{align*}
\end{proof}

In some cases it will be convenient to work with a different, undirected graph which we attach to the algebra $\bA$.

\begin{defn}\label{defn-graph} For any algebra $\bA$, we let $\cG_\bA$ be the graph of two-element subalgebras of $\bA$.
\end{defn}

Note that every projection subalgebra of $\bA$ will become a clique in $\cG_\bA$. In many of the cases we will consider, every two-element subalgebra of $\bA$ will automatically be a projection subalgebra.% A second way we will make use of finiteness will let us focus on cases where our algebras \emph{only} have projection subalgebras.

\begin{prop}\label{minimal-nontrivial} If $\bA$ is a finite nontrivial algebra, then there is some subquotient $\bB \in HS(\bA)$ which is also nontrivial, such that every proper subalgebra or quotient of $\bB$ is a projection algebra. If $\bA$ is clone-minimal, then the set of absorption identities which hold in $\bA$ is the same as the set of absorption identities which hold in $\bB$.
\end{prop}
\begin{proof} Take $\bB$ to be a nontrivial subquotient of $\bA$ of minimal cardinality, and note that any subquotient of $\bB$ is also a subquotient of $\bA$. The second sentence follows immediately from Proposition \ref{absorb-prop}.
\end{proof}

This lets us reduce most questions about absorption identities in minimal clones to the case where every proper subalgebra or quotient of $\bA$ is a projection algebra. In particular, if $f$ is the basic binary operation of such an algebra $\bA$, then for any $a,b \in \bA$ such that $(\{a,b\}, f)$ does not form a projection subalgebra of $\bA$, we must have $\Sg_\bA\{a,b\} = \bA$. We will exploit this when studying binary relations on $\bA$.

\begin{defn} If $\RR \le_{sd} \bA \times \bB$ is a subdirect product, and if $\pi_1,\pi_2$ are the projections from $\RR$ to $\bA$ and $\bB$, then the \emph{linking congruence} of $\RR$ can refer to any of the following three congruences:
\begin{itemize}
\item the congruence $\rho = \ker \pi_1 \vee \ker \pi_2$ on $\RR$,
\item the congruence $\alpha = \pi_1(\ker \pi_1 \vee \ker \pi_2)$ on $\bA$, or
\item the congruence $\beta = \pi_2(\ker \pi_1 \vee \ker \pi_2)$ on $\bB$.
\end{itemize}
Concretely, the linking congruence of $\RR$ relates two elements of $\RR, \bA$, or $\bB$, respectively, if they are in the same connected component of $\RR$ considered as a bipartite graph on $\bA \sqcup \bB$.

We say that $\RR$ is \emph{linked} if the linking congruence of $\RR$ is the full congruence.
\end{defn}

The following standard result is useful when the linking congruence is not full.

\begin{prop} If $\RR \le_{sd} \bA \times \bB$ is subdirect, and if $\rho, \alpha, \beta$ are the linking congruence of $\RR$ as congruences on $\RR, \bA, \bB$ respectively, then $\RR/\rho$ is the graph of an isomorphism between $\bA/\alpha$ and $\bB/\beta$.
\end{prop}

To handle more complex arguments involving binary relations, the following ``additive'' notation is very convenient.

\begin{defn} If $\RR \le \bA \times \bB$ and $\bC \le \bA$, then we define the subalgebra $\bC + \RR$ of $\bB$ by
\[
\bC + \RR = \{b \in \bB \mid \exists c\in\bC \text{ s.t. } (c,b) \in \RR\} = \pi_2(\pi_1^{-1}(\bC) \cap \RR),
\]
and for $\bD \le \bB$ we define the subalgebra $\bD - \RR$ of $\bA$ by
\[
\bD - \RR = \{a \in \bA \mid \exists d\in\bD \text{ s.t. } (a,d) \in \RR\} = \pi_1(\pi_2^{-1}(\bD) \cap \RR).
\]
\end{defn}

The main trick we use when we find a linked subdirect relation is the following result.

\begin{prop} If $\RR \le_{sd} \bA\times\bB$ is a linked subdirect product of finite algebras such that at least one of $\bA,\bB$ has a proper subalgebra, then there is either a proper subalgebra $\bC < \bA$ such that $\bC + \RR = \bB$, or there is a proper subalgebra $\bD < \bB$ such that $\bD - \RR = \bA$.
\end{prop}
\begin{proof} Suppose that $\bC_0$ is a proper subalgebra of $\bA$, and inductively define the sequence of subalgebras $\bC_i \le \bA$ and $\bD_i \le \bB$ by
\[
\bD_i = \bC_i + \RR, \;\;\; \bC_{i+1} = \bD_i - \RR.
\]
Then each $\bC_{i+1}$ contains $\bC_i$ and similarly $\bD_{i+1}$ contains $\bD_i$, and $\bC_i \cup \bD_i$ can be concretely understood as the set of elements of $\bA$ that have distance at most $2i+1$ from $\bC_0$ in $\RR$ considered as a bipartite graph on $\bA \sqcup \bB$. Since $\RR$ is linked, for sufficiently large $i$ we must have $\bC_i = \bA$ and $\bD_i = \bB$, so we just need to find the last place in the sequence $\bC_0, \bD_0, \bC_1, ...$ where we had a proper subalgebra of $\bA$ or $\bB$ to get the required $\bC < \bA$ or $\bD < \bB$.
\end{proof}

\begin{cor}\label{cor-linked} If $\bS \le_{sd} \bA \times \bA$ is a symmetric and linked subdirect relation on a finite algebra $\bA$, and if $\bA$ has a proper subalgebra, then $\bA$ has a proper subalgebra $\bC < \bA$ such that $\bC + \bS = \bA$.
\end{cor}

The last trick we will use allows us to exploit the fact that a binary minimal clone may not contain any semiprojections of arity at least $3$. Recall that a function $f$ is a \emph{semiprojection} if it is not a projection, but there is some index $i$ such that $f(x_1, ..., x_n) = x_i$ whenever two of the inputs to $f$ are equal.

\begin{prop}\label{prop-no-semiprojection} If $\Clo(\bA)$ does not contain any nontrivial semiprojections of arity at least $3$, then for any term $t$ of $\bA$, any absorption equation
\[
t(x_1, ..., x_n) = x_i
\]
which holds whenever at most two distinct values occur among $x_1, ..., x_n$ must hold for all tuples $x_1, ..., x_n$.
\end{prop}
\begin{proof} Suppose for contradiction that $t(x_1, ..., x_n) \not\approx x_i$ in $\bA$, and suppose that the arity $n$ of $t$ is minimal among counterexamples to the proposition. If $n = 3$, then $t$ is a semiprojection, contradicting our assumption on $\Clo(\bA)$. If $n \ge 4$ and $t$ is not a semiprojection, then {\'S}wierczkowski's Lemma from \cite{semiprojection-lemma} shows that there is some pair of inputs $x_j,x_k$ such that the term $t'$ which we obtain by identifying $x_k$ with $x_j$ in $t$ is not a projection. Then $t'$ is a counterexample of arity $n-1$, contradicting our assumption that $n$ was minimal among counterexamples.
\end{proof}

\section{Descriptions of the different types of binary minimal clones}\label{s-descriptions}

Before diving into the proof of the classification of non-Taylor binary minimal clones in the next section, we will discuss each of the five cases in more detail, verifying that each of them really is a nice property. We'll start with a very quick discussion of rectangular bands. The following result is standard.

\begin{prop}\label{rect-band-structure} A rectangular band $\bA = (A,f)$ can always be written as a direct product $\bA \cong \bA_1 \times \bA_2$ such that the basic binary operation $f$ acts as first projection on $\bA_1$ and acts as second projection on $\bA_2$.
\end{prop}
%TODO: prove, if time

The free rectangular band on two generators is a four element algebra, with the multiplication table, labeled digraph, and undirected graph of two-element subalgebras displayed below.
\begin{center}
\begin{tabular}{ccccc}
\begin{tabular}{c | c c c c} $\cdot$ & $x$ & $y$ & $xy$ & $yx$\\ \hline $x$ & $x$ & $xy$ & $xy$ & $x$ \\ $y$ & $yx$ & $y$ & $y$ & $yx$ \\ $xy$ & $x$ & $xy$ & $xy$ & $x$\\ $yx$ & $yx$ & $y$ & $y$ & $yx$\end{tabular}
& \;\;\; &
\begin{tikzpicture}[scale=1.5,baseline=0.75cm]
  \node (x) at (0,0) {$x$};
  \node (yx) at (1.7,0) {$yx$};
  \node (xy) at (0,1) {$xy$};
  \node (y) at (1.7,1) {$y$};
  \draw [->] (x) edge [bend left, "{$y,xy$}"] (xy);
  \draw [->] (xy) edge [bend left, "{$x,yx$}"] (x);
  \draw [->] (y) edge [bend left, "{$x,yx$}"] (yx);
  \draw [->] (yx) edge [bend left, "{$y,xy$}"] (y);
\end{tikzpicture}
& \;\;\; &
\begin{tikzpicture}[scale=1.5,baseline=0.75cm]
  \node (x) at (0,0) {$x$};
  \node (yx) at (1,0) {$yx$};
  \node (xy) at (0,1) {$xy$};
  \node (y) at (1,1) {$y$};
  \draw (x) -- (xy) -- (y) -- (yx) -- (x);
\end{tikzpicture}
\end{tabular}
\end{center}

We also have the following elementary characterization of rectangular bands in terms of absorption identities.%, which we will prove in Proposition \ref{prop-rectangular-projections}.%possibly: (This result may be well-known in the theory of bands - the author is far from being an expert on that topic.)

\begin{prop}\label{rect-band-absorb} An idempotent binary operation $\cdot$ defines a rectangular band iff it satisfies the absorption identity
\begin{equation}
((ux)y)(z(wu)) \approx u.\label{eq-rect-band}
\end{equation}
\end{prop}
\begin{proof} First suppose that $\cdot$ is a rectangular band. Then we have
\[
((ux)y)(z(wu)) \approx (ux)(wu) \approx uu \approx u,
\]
which is \eqref{eq-rect-band}.

Now suppose that $\cdot$ satisfies \eqref{eq-rect-band}. Substituting $u = x = w = z$ in \eqref{eq-rect-band}, we have
\[
(xy)x \approx ((xx)y)(x(xx)) \approx x
\]
and similarly we have
\[
x(yx) \approx ((xx)x)(y(xx)) \approx x,
\]
so by substituting $u = xw$ in \eqref{eq-rect-band} we have
\[
(xy)(zw) \approx (((xw)x)y)(z(w(xw))) \approx xw,
\]
so $\cdot$ is a rectangular band.
\end{proof}

\begin{thm}\label{rect-band-nice} Being a rectangular band is a nice property, and every rectangular band is clone-minimal.
\end{thm}
\begin{proof} Since rectangular bands form a variety defined by finitely many identities, it is easy to check that a given operation $f$ is a rectangular band. If $\bA = (A,f)$ is a nontrivial rectangular band, then by Proposition \ref{rect-band-structure} there is some pair $a,b \in \bA$ such that $\{a,b\}$ is not a projection subalgebra of $\bA$, and the subalgebra generated by $a$ and $b$ is isomorphic to the four-element free rectangular band on two generators. Finally, it's easy to check that every nontrivial operation $t \in \Clo(f)$ which depends on all of its inputs is binary, and is in fact either $f(x,y)$ or $f(y,x)$.
\end{proof}

\subsection{$p$-cyclic groupoids}

\begin{defn}\label{p-cyclic-defn} We say that an idempotent binary operation $f$ is a \emph{$p$-cyclic groupoid} if it satisfies the identities
\begin{align*}
f(x,f(y,z)) &\approx f(x,y),\\
f(f(x,y),z) &\approx f(f(x,z),y),
\end{align*}
and
\[
f(\cdots f(f(x,y),y)\cdots, y) \approx x,
\]
where there are exactly $p$ copies of the variable $y$ in the last identity.
\end{defn}

The free $2$-cyclic groupoid on two generators is a four element algebra, with the multiplication table, labeled digraph, and undirected graph of two-element subalgebras displayed below.
\begin{center}
\begin{tabular}{ccccc}
\begin{tabular}{c | c c c c} $\cdot$ & $x$ & $y$ & $xy$ & $yx$\\ \hline $x$ & $x$ & $xy$ & $x$ & $xy$ \\ $y$ & $yx$ & $y$ & $yx$ & $y$ \\ $xy$ & $xy$ & $x$ & $xy$ & $x$\\ $yx$ & $y$ & $yx$ & $y$ & $yx$\end{tabular}
& \;\;\; &
\begin{tikzpicture}[scale=1.5,baseline=0.75cm]
  \node (x) at (0,0) {$x$};
  \node (y) at (1.7,0) {$y$};
  \node (xy) at (0,1) {$xy$};
  \node (yx) at (1.7,1) {$yx$};
  \draw [->] (x) edge [bend left, "{$y,yx$}"] (xy);
  \draw [->] (xy) edge [bend left, "{$y,yx$}"] (x);
  \draw [->] (y) edge [bend left, "{$x,xy$}"] (yx);
  \draw [->] (yx) edge [bend left, "{$x,xy$}"] (y);
\end{tikzpicture}
& \;\;\; &
\begin{tikzpicture}[scale=1.5,baseline=0.75cm]
  \node (x) at (0,0) {$x$};
  \node (y) at (1,0) {$y$};
  \node (xy) at (0,1) {$xy$};
  \node (yx) at (1,1) {$yx$};
  \draw (x) -- (xy);
  \draw (yx) -- (y);
\end{tikzpicture}
\end{tabular}
\end{center}

The structure of $p$-cyclic groupoids was described by P{\l}onka \cite{plonka-p-cyclic} (see the discussion right after Example 2 of \cite{plonka-p-cyclic}) and elaborated on in Section 2 of \cite{representations-n-cyclic-groupoids}. The structure theorem is most conveniently stated in terms of right orbits.%In order to state the structure theorem, it is convenient to use the notion of a right orbit (Definition \ref{defn-right-orbit}).

\begin{thm}[P{\l}onka \cite{plonka-p-cyclic}, Romanowska, Roszkowska \cite{representations-n-cyclic-groupoids}]\label{p-cyclic-structure} Suppose $\bA = (A,f)$ is a $p$-cyclic groupoid. Then the right orbits $O_1, ..., O_k$ of elements of $\bA$ form a partition of $\bA$, and there exist auxiliary affine structures on the $O_i$ such that the following hold:
\begin{itemize}
\item if we define an equivalence relation $\Theta$ on $\bA$ with equivalence classes given by the $O_i$, then $\Theta$ is a congruence of $\bA$ such that $f$ acts as first projection on $\bA/\Theta$,

\item the restriction of $f$ to any $O_i$ is the first projection,

\item the auxiliary affine structure on $O_i$ is an affine space over $\ZZ/p$, so the number of elements of $O_i$ is a power of $p$,

\item for each $i\ne j$ there exists a vector $v_{ij}$ in the vector space associated to $O_i$ such that for all $a \in O_i, b \in O_j$, we have
\[
a \in O_i, b \in O_j \;\;\; \implies \;\;\; f(a,b) = a + v_{ij},
\]

\item for each $O_i$, the associated vector space is generated by the set of vectors $v_{ij}$ with $j \ne i$.
\end{itemize}
Conversely, given a collection of affine spaces $O_i$ and vectors $v_{ij}$ over $\ZZ/p$ satisfying the above conditions, there is a $p$-cyclic groupoid $\bA$ with $f(a,b) = a+v_{ij}$ for $a \in O_i, b \in O_j$.
\end{thm}
\begin{proof} To see that the right orbits form a partition of $\bA$, we just note that the identity
\[
f(\cdots f(f(x,y),y)\cdots, y) \approx x,
\]
implies that for any $a,c \in \bA$, we have $a \in O(f(a,c))$, so $c \in O(a) \iff a \in O(c)$. By the definition of a right orbit, we have
\[
f(O(a), O(b)) \subseteq O(a),
\]
so the first bullet point follows automatically from the fact that the right orbits partition $\bA$.

Repeated application of the identity
\[
f(x,f(y,z)) \approx f(x,y)
\]
shows that
\[
c \in O(b) \;\;\; \implies \;\;\; f(a,c) = f(a,b).
\]
In particular, since $f$ is idempotent we have
\[
b \in O(a) \;\;\; \implies \;\;\; f(a,b) = f(a,a) = a,
\]
so the restriction of $f$ to any right orbit $O_i$ is the first projection. For $i \ne j$, we let $\alpha_{ij} : O_i \rightarrow O_i$ be the map
\[
\alpha_{ij} : a \in O_i \mapsto f(a,b), \;\;\; b \in O_j.
\]
The identity
\[
f(f(x,y),z) \approx f(f(x,z),y)
\]
shows that $\alpha_{ij}$ commutes with $\alpha_{ik}$ for all $j,k$, and the identity
\[
f(\cdots f(f(x,y),y)\cdots, y) \approx x
\]
shows that $\alpha_{ij}^p$ is the identity for all $j$. Thus for any fixed $i$ the maps $\alpha_{ij}$ generate a vector space over $\ZZ/p$ which acts transitively on $O_i$, granting each $O_i$ the structure of an affine space over $\ZZ/p$ in which each $\alpha_{ij}$ corresponds to a translation by some vector $v_{ij}$.
\end{proof}

\begin{cor}[Corollary 1 of P{\l}onka \cite{plonka-p-cyclic}] The free $p$-cyclic groupoid on $n$ generators has size $np^{n-1}$.
\end{cor}

\begin{rem} The fact that the size of the free $p$-cyclic groupoid on $n$ generators looks like a formal derivative of the size of the free vector space over $\bF_p$ on $n$ generators is quite striking. This may be morally explained by the fact that the free $p$-cyclic groupoid on $n$ generators is isomorphic to the subalgebra of the quasiaffine algebra
\[
\big(\ZZ/p^2, (1-p)x + py\big)^n
\]
which is generated by the basis vectors $e_i = (0, ..., 0, 1, 0, ..., 0)$: the operation
\[
(1-p)x + py = x + p(y-x)
\]
can be thought of as an infinitesimal deformation of first projection, since $p^2 \equiv 0$ in $\ZZ/p^2$. This representation was essentially studied in P{\l}onka's paper \cite{plonka-p-cyclic-minimal} which first introduced $p$-cyclic groupoids, and it appears again in Lemma 3.5 of \cite{minimal-clones-abelian}.

Although the free $p$-cyclic groupoids are quasiaffine, most $p$-cyclic groupoids are not even abelian (in the sense of universal algebra), so free $p$-cyclic groupoids are natural examples of abelian algebras with non-abelian quotients.
\end{rem}

The fact that $p$-cyclic groupoids form minimal clones is attributed to P{\l}onka \cite{plonka-p-cyclic-minimal} - we include a proof below.

\begin{thm}\label{p-cyclic-nice} For any fixed prime $p$, being a $p$-cyclic groupoid is a nice property, and every nontrivial $p$-cyclic groupoid is clone-minimal.
\end{thm}
\begin{proof} Since $p$-cyclic groupoids form a variety defined by finitely many identities, it is easy to check that a given operation $f$ is a $p$-cyclic groupoid. If $\bA = (A,f)$ is a nontrivial $p$-cyclic groupoid, and if $a,b \in \bA$ have $f(a,b) \ne a$, then the subalgebra of $\bA$ generated by $a$ and $b$ is a $p$-cyclic groupoid with two right orbits, and either has size $p+1$ or $2p$. If the subalgebra generated by $a$ and $b$ has size $p+1$, then the algebra $\Sg_{\bA^2}\{(a,b),(b,a)\}$ will have size $2p$. Either way, we see that the free $p$-cyclic groupoid on two generators is in the variety generated by $\bA$. (In fact, any nontrivial $p$-cyclic groupoid generates the full variety of all $p$-cyclic groupoids.)

If $f$ defines a $p$-cyclic groupoid, then we can use Theorem \ref{p-cyclic-structure} to see that each of the $p-1$ operations $f_1, ..., f_{p-1}$ defined by $f_1 = f$ and
\[
f_{k+1}(x,y) = f(f_k(x,y), y)
\]
also defines a $p$-cyclic groupoid operation with the same decomposition into right orbits $O_i$, with the operation $f_k$ given by
\[
a \in O_i, b \in O_j \;\;\; \implies \;\;\; f_k(a,b) = a + kv_{ij}.
\]
So for every nontrivial binary operation $g \in \Clo(f)$, either $g(x,y)$ or $g(y,x)$ defines a $p$-cyclic groupoid as well, and $f \in \Clo(g)$.

More generally, suppose $g$ is any nontrivial term in $\Clo(f)$. We may assume without loss of generality that $g$ is in the same right orbit of the free $p$-cyclic groupoid on $n$-generators as the first projection, and that $g$ depends on its last input. Then using the explicit description of the free $p$-cyclic groupoid on $n$ generators as a disjoint union of $n$ affine spaces of size $p^{n-1}$, we see that there are $k_2, ..., k_n \in \ZZ/p$ such that
\[
a \in O_i, b_m \in O_{j_m} \;\;\; \implies \;\;\; g(a, b_2, ..., b_n) = a + k_2v_{ij_2} + \cdots + k_nv_{ij_n},
\]
and that $k_n \not\equiv 0 \pmod{p}$ if $g$ depends on its last input. Thus the binary operation
\[
g(x,...,x,y)
\]
is the same as the (nontrivial) binary operation $f_{k_n}(x,y)$, and so $f \in \Clo(g)$.
\end{proof}

\subsection{Partial semilattice operations}

\begin{defn} An idempotent binary operation $s$ is called a \emph{partial semilattice operation} if it satisfies the identities
\[
s(x,s(x,y)) \approx s(s(x,y),x) \approx s(x,y).
\]
\end{defn}

\begin{prop}\label{prop-partial-semilattice} The algebra $\bA = (A,s)$ is a partial semilattice if and only if $\{a,s(a,b)\}$ is a semilattice subalgebra of $\bA$ with absorbing element $s(a,b)$ for every pair of elements $a,b \in \bA$.

In particular, if $s$ is a partial semilattice operation which is not identically equal to first projection, then $\bA$ contains at least one nontrivial semilattice subalgebra.
\end{prop}

\begin{cor}\label{cor-absorption-semilattice} A nontrivial partial semilattice operation does not satisfy any absorption identities other than those following from idempotence.
\end{cor}
\begin{proof} This follows from the fact that a semilattice operation on a two element set does not satisfy any absorption identities other than those following from idempotence.
\end{proof}

It is often helpful to visualize partial semilattice operations via the directed graph of two element semilattice subalgebras, with $a \rightarrow b$ when $s(a,b) = b$. The next result shows that this is the same as the digraph $D_\bA$ from Definition \ref{defn-digraph}.

\begin{prop}\label{semilattice-structure} If $\bA = (A,s)$ is a partial semilattice which defines a minimal clone, then for any $a,b \in \bA$ the following are equivalent:
\begin{itemize}
\item[(a)] $\{a,b\}$ is a semilattice subalgebra of $\bA$ with $s(a,b) = s(b,a) = b$,

\item[(b)] $s(a,b) = b$,

\item[(c)] there exists $c \in \bA$ such that $s(a,c) = b$,

\item[(d)] $(b,b) \in \Sg_{\bA^2}\{(a,b),(b,a)\}$.
\end{itemize}
If $\{a,b\}$ is not a subalgebra of $\bA$, then there is a surjective homomorphism from $\bS = \Sg_{\bA^2}\{(a,b),(b,a)\}$ to the free semilattice on two generators.
\end{prop}
\begin{proof} That (a), (b), (c) are equivalent follows directly from the definition of a partial semilattice. That (a) implies (d) is clear. To see that (d) implies (a), let $t$ be any term with $t(a,b) = t(b,a) = b$. Then $t$ is not a projection, and in order for $\bA$ to be a minimal clone, we must have $s \in \Clo(t)$. Since $\{a,b\}$ forms a semilattice subalgebra under $t$, it must also form a semilattice subalgebra under $s$.

For the last statement, suppose $t$ is any nontrivial binary term of $\bA$. Then if $\{a,b\}$ is not a subalgebra of $\bA$, it must not be closed under $t$ (as $s \in \Clo(t)$), so we must have at least one of $t(a,b),t(b,a) \not\in \{a,b\}$. Thus, for any nontrivial binary term $t$ of $\bA$ we have
\[
t\Big(\begin{bmatrix} a\\ b\end{bmatrix}, \begin{bmatrix} b\\ a\end{bmatrix}\Big) \not\in \Big\{\begin{bmatrix} a\\ b\end{bmatrix},\begin{bmatrix} b\\ a\end{bmatrix}\Big\},
\]
so the equivalence relation on $\bS = \Sg_{\bA^2}\{(a,b),(b,a)\}$ with equivalence classes $\{(a,b)\}, \{(b,a)\}, \bS\setminus \{(a,b),(b,a)\}$ defines a congruence with quotient isomorphic to the free semilattice on two generators.
\end{proof}

There are two possible types of two-element subalgebras of a partial semilattice: semilattice subalgebras, and projection subalgebras. If $\bA$ is a nontrivial partial semilattice which is not Taylor, then both types of two element subalgebra must appear in the variety generated by $\bA$. The free algebra in the subvariety of $\Var(\bA)$ which is generated by these two types of two-element partial semilattices has four elements, and its multiplication table, labeled digraph, and undirected graph of two-element subalgebras are displayed below.
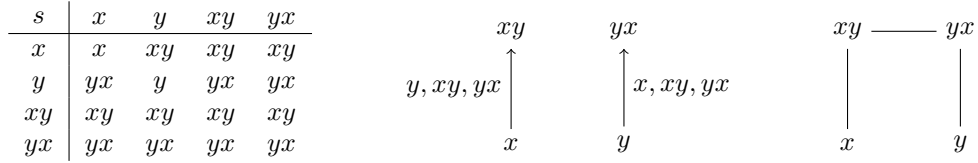
\begin{center}
\begin{tabular}{ccccc}
\begin{tabular}{c | c c c c} $s$ & $x$ & $y$ & $xy$ & $yx$\\ \hline $x$ & $x$ & $xy$ & $xy$ & $xy$ \\ $y$ & $yx$ & $y$ & $yx$ & $yx$ \\ $xy$ & $xy$ & $xy$ & $xy$ & $xy$\\ $yx$ & $yx$ & $yx$ & $yx$ & $yx$\end{tabular}
& \;\;\; &
\begin{tikzpicture}[scale=1.5,baseline=0.75cm]
  \node (x) at (0,0) {$x$};
  \node (y) at (1,0) {$y$};
  \node (xy) at (0,1) {$xy$};
  \node (yx) at (1,1) {$yx$};
  \draw [->] (x) edge ["{$y,xy,yx$}"] (xy);
  \draw [->] (y) edge ["{$x,xy,yx$}"'] (yx);
\end{tikzpicture}
& \;\;\; &
\begin{tikzpicture}[scale=1.5,baseline=0.75cm]
  \node (x) at (0,0) {$x$};
  \node (y) at (1,0) {$y$};
  \node (xy) at (0,1) {$xy$};
  \node (yx) at (1,1) {$yx$};
  \draw (x) -- (xy) -- (yx) -- (y);
\end{tikzpicture}
\end{tabular}
\end{center}

In order to check that being a partial semilattice is a nice property, we need a way to construct a nontrivial partial semilattice operation $s$ out of a binary operation $t$ which acts as a semilattice on some two-element set $\{a,b\}$. That this is possible is originally a result of Bulatov \cite{colored-graph}, and a more algorithmic construction based on Bulatov's technique was given in the ``Semilattice Iteration Lemma'' of \cite{brady-examples}. We include the proof here since it is short and the proof technique is useful elsewhere. The first step is to apply Corollary \ref{binary-iteration} to construct $f = t^{\infty_2}$ satisfying the identity $f(x,f(x,y)) \approx f(x,y)$.

\begin{prop}\label{t-to-u} If $f$ is an idempotent binary operation which satisfies the identity
\[
f(x,f(x,y)) \approx f(x,y),
\]
and if we define a binary operation $u$ by
\[
u(x,y) \coloneqq f(x,f(y,x)),
\]
then $u$ satisfies the identity
\[
u(u(x,y),x) \approx u(x,y).
\]
Similarly, if an idempotent operation $g$ satisfies the identity $g(g(x,y),y) \approx g(x,y)$, then $g(g(x,y),x)$ is a binary operation satisfying the above identity.
\end{prop}
\begin{proof}
We have
\begin{align*}
u(u(x,y),x) &\approx f(u(x,y),f(x,u(x,y)))\\
&\approx f(u(x,y),f(x,f(x,f(y,x))))\\
&\approx f(u(x,y),f(x,f(y,x)))\\
&\approx f(u(x,y),u(x,y))\\
&\approx u(x,y).
\end{align*}
For the last claim, take $f(x,y) = g(y,x)$.
\end{proof}

\begin{prop}\label{u-infinity} If $u$ is a binary operation on a finite set which satisfies the identity
\[
u(u(x,y),x) \approx u(x,y),
\]
then $s = u^{\infty_2}$ satisfies the identities
\[
s(x,s(x,y)) \approx s(s(x,y),x) \approx s(x,y).
\]
In particular, if $u$ is idempotent then $s$ is a partial semilattice operation.
\end{prop}
\begin{proof} First, note that even if $u$ is not idempotent we have
\begin{align*}
u(u(x,y),u(x,y)) &\approx u(u(u(x,y),x),u(x,y))\\
&\approx u(u(x,y),x)\\
&\approx u(x,y).
\end{align*}
Define $u^{*_2n}$ as a shorthand for $u *_2 \cdots *_2 u$, as in the definition of $u^{\infty_2}$. Then since the unary operation
\[
z \mapsto u(u(x,y), z)
\]
takes $x$ to $u(x,y)$ and takes $u(x,y)$ to $u(x,y)$, we see that for any $m$ we have
\[
u^{*_2m}(u(x,y),x) \approx u(x,y).
\]
Replacing $y$ by $u^{*_2(m-1)}(x,y)$, we get
\[
u^{*_2m}(u^{*_2m}(x,y),x) \approx u^{*_2m}(x,y)
\]
for all $m \ge 1$.

Since $s = u^{\infty_2}$ is a pointwise limit of a subsequence of the $u^{*_2m}$s, we have
\[
s(s(x,y),x) \approx s(x,y),
\]
and the identity $s(x,s(x,y)) \approx s(x,y)$ follows from Corollary \ref{binary-iteration}.
\end{proof}

\begin{lem}[Proposition 10 of Bulatov \cite{colored-graph}, Semilattice Iteration Lemma of \cite{brady-examples}]\label{lem-semilattice-iteration} If an idempotent algebra $\bA$ has $\bB \in \Var(\bA)$ and $a \ne b \in \bB$ with
\[
\begin{bmatrix} b\\ b\end{bmatrix} \in \Sg_{\bB^2}\Big\{\begin{bmatrix} a\\ b\end{bmatrix},\begin{bmatrix} b\\ a\end{bmatrix}\Big\}
%(b,b) \in \Sg_{\bB^2}\{(a,b),(b,a)\}
\]
then $\bA$ has a nontrivial partial semilattice term operation $s$ satisfying $s(a,b) = s(b,a) = b$ in $\bB$.% Furthermore, if $s$ is a partial semilattice operation, then for any nontrivial term $t \in \Clo(s)$ there is a nontrivial partial semilattice operation $s' \in \Clo(t)$.
\end{lem}
\begin{proof} Choose $t \in \Clo(\bA)$ such that $t(a,b) = t(b,a) = b$ in $\bB$. Let $u(x,y) = t^{\infty_2}(x,t^{\infty_2}(y,x))$ and let $s = u^{\infty_2}$. Then $s$ is a partial semilattice by Corollary \ref{binary-iteration} and Propositions \ref{t-to-u}, \ref{u-infinity}. Explicit computation shows that we have $s(a,b) = s(b,a) = b$, so in particular $s$ is nontrivial.
\begin{comment}
For the second part, if $s$ is a nontrivial partial semilattice operation, then there are $a,b \in \bA$ with $s(a,b) \ne a$, so $\{a, s(a,b)\}$ is a proper subalgebra of $\bA$, on which $s$ acts as a semilattice operation. Then any nontrivial $t \in \Clo(s)$ which depends on all its inputs also acts as a semilattice operation on $\{a, s(a,b)\}$, so if we take $g(x,y) = t(x,y,...,y)$ then we have $g(a,s(a,b)) = g(s(a,b),a) = s(a,b)$, so by the first part of the argument there is a nontrivial partial semilattice operation $s' \in \Clo(g) \subseteq \Clo(t)$.
\end{comment}
\end{proof}

\begin{thm}\label{partial-semi-nice} Being a partial semilattice operation is a nice property.
\end{thm}
\begin{proof} Since partial semilattices form a variety defined by finitely many identities, it is easy to check that a given operation $f$ is a partial semilattice. By Proposition \ref{semilattice-structure}, if $\bA = (A,s)$ is a nontrivial partial semilattice, then there is at least one two-element semilattice subalgebra $\{a, s(a,b)\}$ contained in $\bA$.

Now suppose that $g \in \Clo(s)$ is a nontrivial $n$-ary operation. We may assume without loss of generality that $g$ depends on all of its arguments, in which case $g$ acts on the semilattice subalgebra $\{a,s(a,b)\}$ as an $n$-ary semilattice operation. In particular, we see that the binary operation $t(x,y) = g(x,y,...,y)$ acts as a semilattice operation on $\{a,s(a,b)\}$. To finish the proof we apply Lemma \ref{lem-semilattice-iteration} to see that there is a partial semilattice operation $s' \in \Clo(t)$ which also acts as a semilattice operation on $\{a,s(a,b)\}$.
\end{proof}

It is hard to say much more about the structure of partial semilattices which define minimal clones, but the following general result about partial semilattices can be used to rule out cases where the digraph of semilattice subalgebras of $\Sg_\bA\{a,b\}$ has very few strongly connected components. We say that a strongly connected component $U$ of a digraph is \emph{maximal} (with respect to the reachability order $\preceq$ on strongly connected components) if $a \in U$ and $a \rightarrow b$ implies $b \in U$.

\begin{thm}[Theorem 2(c) from \cite{brady-examples}, generalizing a result from \cite{colored-graph}]\label{strong-binary} Fix a partial semilattice term $s$ on $\bA\times\bB$. Suppose $\RR \le_{sd} \bA \times \bB$ is subdirect and $U,V$ are maximal strongly connected components of the labeled digraphs of $\bA, \bB$, respectively, with respect to $s$. If $U$ is contained in a linked component of $\RR$% (that is, a connected component of $\RR$ considered as a bipartite graph on $\bA \sqcup \bB$)
, $(U\times V) \cap \RR \ne \emptyset$, and $U,V$ are finite, then $U \times V \subseteq \RR$.
\end{thm}

\begin{comment}
In order to apply this result, we need the concept of the linking congruence.

\begin{defn} If $\RR \le_{sd} \bA \times \bB$ is a subdirect product, then the \emph{linking congruence} of $\RR$ can refer to any of the following three congruences: the congruence $\ker \pi_1 \vee \ker \pi_2$ on $\RR$, the congruence $\pi_1(\ker \pi_1 \vee \ker \pi_2)$ on $\bA$, or the congruence $\pi_2(\ker \pi_1 \vee \ker \pi_2)$ on $\bB$. Concretely, the linking congruence of $\RR$ relates two elements of $\RR, \bA$, or $\bB$, respectively, if they are in the same connected component of $\RR$ considered as a bipartite graph on $\bA \sqcup \bB$.
\end{defn}
\end{comment}

\begin{prop}\label{semi-not-strongly-connected} If $\bA = (A,s)$ is a clone-minimal partial semilattice, and if $\bA$ is generated by two elements $a,b$, then the directed graph $D_\bA$ of semilattice subalgebras of $\bA$ is not strongly connected.
\end{prop}
\begin{proof} Suppose for contradiction that $\bA$ is a counterexample. Since any quotient of a strongly connected algebra is strongly connected, we may assume without loss of generality that $\bA$ is simple. Consider the algebra
\[
\bS = \Sg_{\bA^2}\Big\{\begin{bmatrix} a\\ b \end{bmatrix}, \begin{bmatrix} b\\ a \end{bmatrix}\Big\}.
\]
Since $\bA$ is simple, the linking congruence of $\bS$ on $\bA$ is either trivial or full, so either $\bS$ is the graph of an isomorphism or $\bS$ is linked. If $\bS$ is the graph of an isomorphism, then $\bS \cong \bA$, so Proposition \ref{semilattice-structure} implies that $\bA$ has a surjective homomorphism to the free semilattice on two generators, which contradicts the assumption that $\bA$ is strongly connected, and also contradicts the assumption that $\bA$ is simple.

Thus $\bS$ must be linked. By Theorem \ref{strong-binary}, this implies that $\bS = \bA\times \bA$, so in particular we have $(b,b) \in \bS$. But then Proposition \ref{semilattice-structure} implies that $\{a,b\}$ is a semilattice subalgebra of $\bA$, which is again a contradiction.
\end{proof}

More results along these lines can be found in Appendix \ref{a-partial-semi}.

\begin{comment}
 The following conjecture seems natural.

\begin{conj} If $\bA = (A,s)$ is a (finite) clone-minimal algebra with $s$ a partial semilattice operation, such that $\bA$ is generated by two elements $a,b$, then $a$ and $b$ are never in the same strongly connected component of the digraph $D_\bA$ of semilattice subalgebras of $\bA$.
\end{conj}
\end{comment}

\begin{rem} The definition of a spiral implies that every nontrivial spiral $\bA$ has a two-element semilattice subquotient. By Lemma \ref{lem-semilattice-iteration}, this implies that every nontrivial spiral has a nontrivial partial semilattice term operation. However, it seems helpful to keep the non-Taylor case separate from the Taylor case, since clone-minimal spirals are significantly easier to classify than non-Taylor clone-minimal partial semilattices. Part of the reason for this is that commutativity significantly reduces the amount of casework, and part of it is that the definition of a spiral can be used to directly prove a recursive structure theorem for spirals, in the sense that every spiral can be thought of as being built out of an overlapping arrangement of strictly smaller spirals.
\end{rem}

\begin{rem} The phrase ``partial semilattice operation'' is a bit of a mouthful, which the author has unfortunately committed to. Perhaps a better name would have been ``arrow operation'', as a reference to the fact that such an operation produces ``arrows'' $a \rightarrow s(a,b)$.
\end{rem}

\subsection{Melds}

Now we consider the case of melds.

\begin{defn} An idempotent groupoid $\bA = (A, \cdot)$ is called a \emph{meld} if it satisfies the identity
\[
(xy)(zx) \approx xy.
\]
\end{defn}

\begin{prop}\label{prop-meld-absorption} An idempotent groupoid is a meld if and only if it satisfies the absorption identity
\begin{align}
x((yx)z) \approx x.\tag{*}\label{meld-star}
\end{align}
\end{prop}
\begin{proof} First we show that any idempotent groupoid which satisfies $x((yx)z) \approx x$ is a meld. Note first that if we make the substitution $y \mapsto x$ in \eqref{meld-star} (and apply idempotence), we get
\begin{align}
x(xz) \approx x((xx)z) \approx x.\label{meld-x(xz)}
\end{align}
Similarly, if we first apply idempotence to see that $yx \approx (yx)(yx)$ and then substitute $z \mapsto yx$ in \eqref{meld-star}, then we get
\begin{align}
x(yx) \approx x((yx)(yx)) \approx x.\label{meld-x(yx)}
\end{align}
Applying \eqref{meld-x(xz)} with $z \mapsto y$ and \eqref{meld-x(yx)} with $x \mapsto xy, y \mapsto x$, we see that
\begin{align}
(xy)x \approx (xy)(x(xy)) \approx xy.\label{meld-(xy)x}
\end{align}
Substituting $x$ with $xy$ in $x \approx x((yx)z)$ and applying \eqref{meld-x(yx)} with $x$ and $y$ reversed, we see that
\begin{align}
xy \approx (xy)((y(xy))z) \approx (xy)(yz),\label{meld-(xy)(yz)}
\end{align}
and similarly if we substitute $x$ with $xy$ and $y$ with $zx$ in \eqref{meld-star} and apply \eqref{meld-(xy)(yz)} and \eqref{meld-(xy)x} with $x,y,z$ permuted cyclically then we get
\begin{align}
xy \approx (xy)(((zx)(xy))z) \approx (xy)((zx)z) \approx (xy)(zx).\label{meld-(xy)(zx)}
\end{align}
Thus, the identity $x((yx)z) \approx x$ implies the identity $(xy)(zx) \approx xy$.

\begin{comment}
All in one string of equalities:
\begin{align*}
(xy)(zx) &\approx (xy)((zx)((z(zx))(z(zx))))\\
&\approx (xy)((zx)(z(zx)))\\
&\approx (xy)((zx)(z((zz)x)))\\
&\approx (xy)((zx)z)\\
&\approx (xy)(((zx)((x(zx))y))z)\\
&\approx (xy)(((zx)((x((zx)(zx)))y))z)\\
&\approx (xy)(((zx)(xy))z)\\
&\approx xy.
\end{align*}
\end{comment}

For the converse direction, suppose $\bA = (A,\cdot)$ is a meld. Substituting $y \mapsto x$ in $(xy)(zx) \approx xy$, we get
\[
x(zx) \approx (xx)(zx) \approx xx \approx x,
\]
so \eqref{meld-x(yx)} holds in every meld. Applying \eqref{meld-x(yx)}, followed by substituting $x \mapsto yx, y \mapsto z, z \mapsto x$ into the meld equation $(xy)(zx) \approx xy$, we get
\begin{align}
((yx)z)x \approx ((yx)z)(x(yx)) \approx (yx)z,\label{meld-((yx)z)x}
\end{align}
and if we first apply \eqref{meld-((yx)z)x} and then apply \eqref{meld-x(yx)} with $y$ replaced by $(yx)z$, then we get
\[
x((yx)z) \approx x(((yx)z)x) \approx x,
\]
which is \eqref{meld-star}.
\begin{comment}
All in one string of equalities:
\begin{align*}
x((yx)z) &\approx x(((yx)z)(x(yx)))\\
&\approx x(((yx)z)((xx)(yx)))\\
&\approx x(((yx)z)(xx))\\
&\approx x(((yx)z)x)\\
&\approx (xx)(((yx)z)x)\\
&\approx xx\\
&\approx x.
\end{align*}
\end{comment}
\end{proof}

\begin{prop}\label{meld-symmetry} Suppose $\bA = (A,\cdot)$ is a meld. Whenever we have $ab = a$, we also have $ba = b$.
\end{prop}
\begin{proof} Supposing that $ab = a$, we have
\[
%ba = b(ab) = (bb)(ab) = bb = b.\qedhere
ba = b(ab) = b,
\]
where the second equality follows from equation \eqref{meld-x(yx)} from the proof of Proposition \ref{prop-meld-absorption}.
\end{proof}

\begin{cor} If $\bA = (A,\cdot)$ is a meld, then the undirected graph $\cG_\bA$ from Definition \ref{defn-graph} has an edge connecting $a$ to $b$ iff $ab = a$ iff $ba = b$. In particular, $\cdot$ acts as first projection on every edge of $\cG_\bA$.
\end{cor}

\begin{comment}
\begin{defn} If $\bA = (A,\cdot)$ is a meld, define the graph $\cG_\bA$ to have an edge connecting $a$ to $b$ whenever $ab = a$, or equivalently whenever $ba = b$.
\end{defn}
\end{comment}

\begin{prop}\label{meld-graph} If $\bA = (A,\cdot)$ is a meld, then for any $x,y \in \bA$, the element $xy$ is adjacent to $x$, to $y$, and to every neighbor of $x$ in $\cG_\bA$.
\end{prop}
\begin{proof} By equation \eqref{meld-(xy)x} from the proof of Proposition \ref{prop-meld-absorption}, we have $(xy)x = xy$, so $xy$ is adjacent to $x$. By swapping $x$ and $y$ in equation \eqref{meld-x(yx)}, we have $y(xy) = y$, so $xy$ is adjacent to $y$. Finally, for any $z$ with $zx = z$ we have
\[
(xy)z = (xy)(zx) = xy,
\]
so $xy$ is adjacent to $z$.
\end{proof}

\begin{cor} If $\bA$ is a finite meld, then the graph $\cG_\bA$ has a vertex which connects to all other vertices.
\end{cor}

Conversely, we have the following result.

\begin{prop}\label{meld-construction} Suppose $\cG$ is a graph with vertex set $A$ such that some vertex of $\cG$ is connected to all other vertices. Then for any idempotent binary function $\cdot : A^2 \rightarrow A$ which restricts to the first projection on every edge of $\cG$, and which has the property that $xy$ connects to $x$, $y$, and every neighbor of $x$, the algebra $\bA = (A,\cdot)$ will be a meld with $\cG_\bA = \cG$.% $\cdot$ will satisfy the identity $(xy)(zx) \approx xy$.
\end{prop}
\begin{proof} We need to verify the meld equation $(xy)(zx) \approx xy$. By our assumption on $\cdot$, $zx$ is adjacent to $x$, and therefore $xy$ is adjacent to $zx$ as well, so $(xy)(zx) = xy$ since $\cdot$ restricts to first projection on the edge from $xy$ to $zx$.
\end{proof}

\begin{prop}\label{free-meld} If $\bA$ is a nontrivial meld, then the free algebra on two generators $\cF_{\cV(\bA)}(x,y)$ has exactly four elements: $x,y,xy,yx$. The undirected graph associated to this free algebra has only one non-edge (between $x$ and $y$).
\end{prop}
\begin{proof} Since $yx$ is connected to $x$, we see that $xy$ is connected to all three of $x,y,yx$, so $\{x,y,xy,yx\}$ is closed under $\cdot$. To see that $xy \not\in \{y,yx\}$ in the free algebra, note that every meld of size at least two has at least one edge, so there is a projection algebra in $\Var(\bA)$. Since $\cdot$ must act like first projection on any projection algebra, we see that the right orbits $O(x), O(y)$ must be disjoint in $\cF_{\cV(\bA)}(x,y)$ by Proposition \ref{prop-pi1} and Proposition \ref{prop-right-pi1}. That $xy \not\approx x$ follows from the assumption that $\bA$ is nontrivial. Thus the elements $x,y,xy,yx$ of $\cF_{\cV(\bA)}(x,y)$ are pairwise distinct, which completes the proof.
\end{proof}

The multiplication table, labeled digraph, and undirected graph of two-element subalgebras of the free meld on two generators are displayed below.
\begin{center}
\begin{tabular}{ccccc}
\begin{tabular}{c | c c c c} $\cdot$ & $x$ & $y$ & $xy$ & $yx$\\ \hline $x$ & $x$ & $xy$ & $x$ & $x$ \\ $y$ & $yx$ & $y$ & $y$ & $y$ \\ $xy$ & $xy$ & $xy$ & $xy$ & $xy$\\ $yx$ & $yx$ & $yx$ & $yx$ & $yx$\end{tabular}
& \;\;\; &
\begin{tikzpicture}[scale=1.5,baseline=0.75cm]
  \node (x) at (0,0) {$x$};
  \node (y) at (1,0) {$y$};
  \node (xy) at (0,1) {$xy$};
  \node (yx) at (1,1) {$yx$};
  \draw [->] (x) edge ["{$y$}"] (xy);
  \draw [->] (y) edge ["{$x$}"'] (yx);
\end{tikzpicture}
& \;\;\; &
\begin{tikzpicture}[scale=1.5,baseline=0.75cm]
  \node (x) at (0,0) {$x$};
  \node (y) at (1,0) {$y$};
  \node (xy) at (0,1) {$xy$};
  \node (yx) at (1,1) {$yx$};
  \draw (yx) -- (x) -- (xy) -- (yx) -- (y) -- (xy);
\end{tikzpicture}
\end{tabular}
\end{center}

Proposition \ref{free-meld} gets us most of the way to verifying that being a meld is a nice property. In fact, it turns out that every nontrivial meld is clone-minimal. The only tricky part of verifying this is checking that a meld has no semiprojections of arity at least $3$ in its clone.

\begin{lem}\label{meld-semiprojection} Suppose $\bA = (A, f)$ is a nontrivial meld. If a term $t \in \Clo(f)$ satisfies
\[
t(x,y,...,y) \approx x,
\]
then $t$ is first projection, and if
\[
t(x,y,...,y) \not\approx y
\]
then $t(x,y,z,...)$ must be adjacent to $x$ in the graph $\cG_\bF$ associated to the free algebra $\bF = \cF_{\cV(\bA)}(x,y,z,...)$.
\end{lem}
\begin{proof} We induct on the construction of $t$ as a term in $\Clo(f)$. Suppose that $t = f(u,v)$, where $u,v$ are simpler terms. From the explicit description of $\cF_{\cV(\bA)}(x,y)$ in Proposition \ref{free-meld}, we see that
\[
t(x,y,...,y) \not\approx y
\]
is equivalent to
\[
u(x,y,...,y) \not\approx y \;\;\; \text{ or } \;\;\; v(x,y,...,y) \approx x.
\]
In the first case, the induction hypothesis shows that $u$ is adjacent to $x$, so $t$ is also adjacent to $x$ by Proposition \ref{meld-graph}. In the second case, the induction hypothesis shows that $v$ is first projection, so $t$ is adjacent to $v \approx x$ by Proposition \ref{meld-graph}.

Now suppose that
\[
t(x,y,...,y) \approx x.
\]
Then from the explicit description of $\cF_{\cV(\bA)}(x,y)$ in Proposition \ref{free-meld}, we see that
\[
u(x,y,...,y) \approx x \;\;\; \text{ and } \;\;\; v(x,y,...,y) \not\approx y.
\]
By the induction hypothesis, in this case $u$ is first projection and $v(x,y,z,...)$ is adjacent to $x$, so $f(u,v) \approx f(x,v) \approx x$.
\end{proof}

\begin{cor} If an absorption identity holds in some nontrivial meld $\bA$, then it is a consequence of the absorption identity $x((yx)z) \approx x$.
\end{cor}

\begin{thm}\label{meld-nice} Being a meld is a nice property, and every nontrivial meld is clone-minimal.
\end{thm}
\begin{proof} Since melds form a variety defined by finitely many identities, it is easy to check that a given operation $f$ defines a meld. By Proposition \ref{free-meld}, for any nontrivial meld $\bA$ the four-element algebra $\cF_{\cV(\bA)}(x,y)$ is in the variety generated by $\bA$ and is isomorphic to the free meld on two generators. The only thing left to check is that every nontrivial meld $\bA = (A,f)$ is clone-minimal.

Suppose that $g$ is any nontrivial term in $\Clo(f)$. By Rosenberg's Theorem \ref{rosenberg-types}, there is some nontrivial $t \in \Clo(g)$ which is either a unary operation, a majority operation, an affine operation, a binary operation, or a semiprojection. It's easy to see that $t$ can't be unary (since $f$ is idempotent) and that $t$ can't be majority or affine (since there is a two-element projection algebra in the variety generated by $\bA$). If $t$ is binary, then by the explicit description of $\cF_{\cV(\bA)}(x,y)$ we see that $t$ is either $f(x,y)$ or $f(y,x)$, in which case $f \in \Clo(g)$. The only remaining case is the case where $t$ is a semiprojection of arity at least $3$, but in this case Lemma \ref{meld-semiprojection} shows that $t$ must in fact be a projection, contradicting the assumption that $t$ is nontrivial.
\end{proof}

Although the free meld on two generators was easy to describe, free melds on more than two generators are significantly more complex.

\begin{prop} The free meld on three generators has infinitely many elements.
\end{prop}
\begin{proof} We will prove this by constructing an infinite meld $\bA$ which is generated by three elements, using the approach of Proposition \ref{meld-construction}. The underlying set of $\bA$ will be
\[
A = \{x_i, y_i, z_i \mid i \ge 0\} \cup \{x_iy_i, y_iz_i, z_ix_i \mid i \ge 0\} \cup \{t\}.
\]
We will arrange the construction so that $\bA$ has an automorphism $\sigma$ of order three which cyclically permutes $x_i, y_i, z_i$ and cyclically permutes $x_iy_i, y_iz_i, z_ix_i$ for all $i$, and which fixes $t$. To ensure that $\bA$ is generated by the three elements $x_0, y_0, z_0$, we will arrange our construction to satisfy the equations
\[
x_{i+1} = x_i(y_iz_i), \;\;\; y_{i+1} = y_i(z_ix_i), \;\;\; z_{i+1} = z_i(x_iy_i).
\]
We start by defining a graph $\cG$ on the underlying set of $\bA$ which is compatible with the equations above.

First, we connect $t$ by an edge to every other vertex in $\cG$. We take $x_i$ adjacent to $x_j$ and $x_jy_j$ for all $j$, while $x_i$ is never adjacent to $y_j$ or $z_j$ for any $j$ (and similarly for $y_i, z_i$, by cyclically permuting the variables via $\sigma$). Additionally, every pair of elements contained in the three sequences $x_iy_i, y_jz_j, z_kx_k$ will be adjacent. For the remaining pairs, we take $x_i$ adjacent to $y_jz_j$ iff $i > j$, while $x_i$ is adjacent to $z_jx_j$ iff $i \ge j$.

To finish the construction we need to define a multiplication $\cdot$ on $A$ which is compatible with $\cG$ in the sense of Proposition \ref{meld-construction}. On the edges of $\cG$, $\cdot$ must restrict to first projection, so we just need to describe the operation $\cdot$ on the non-edges of $\cG$.

For the non-edges of $\cG$ involving the vertex $x_i$, the operation $\cdot$ is defined as follows. We set $y_j\cdot x_i = t$ for all $i,j$, $x_i\cdot y_j = t$ for $i \ne j$, and $x_i\cdot y_i = x_iy_i$. For $i < j$, we set $x_i\cdot (y_jz_j) = (y_jz_j)\cdot x_i = t$ and $x_i\cdot (z_jx_j) = (z_jx_j)\cdot x_i = t$. Finally, we set $x_i\cdot (y_iz_i) = x_{i+1}$ and $(y_iz_i)\cdot x_i = t$. (The values of $\cdot$ on the remaining non-edges of $\cG$ are determined by our desire for the permutation $\sigma$ to be an automorphism of $\bA$.)
\end{proof}

The previous result shows that the variety of melds is not locally finite, but the next result shows that it is nevertheless quite easy to answer questions about finitely presented melds.

\begin{thm} The word problem can be solved efficiently in the variety of melds.

More concretely, suppose we are given a finite list of generators $x_i$, and a finite list of relations $u_i = v_i$, where $u_i, v_i$ are terms built out of the generators $x_i$. Then we can determine whether the relations $u_i = v_i$ (together with the meld identities) imply that $r = s$ for any given terms $r,s$, in time polynomial in the total length of the descriptions of $u_i, v_i, r, s$.
\end{thm}
\begin{proof} Let $\mathcal{W}$ be the set of all subterms of any of the $u_i,v_i$ or of $r,s$. The main idea is to keep track of all of the equations which can be easily proved between various elements of $\mathcal{W}$, and to use Proposition \ref{meld-construction} at the end to see that no other equations between elements of $\mathcal{W}$ can be proved. To this end, we construct the minimal equivalence relation $\sim$ on $\mathcal{W}$ and graph $\cG/\!\sim$ on $\mathcal{W}/\!\sim$ (which we think of as a graph $\cG$ on $\mathcal{W}$ which is compatible with the equivalence relation $\sim$) which have the following closure properties:
\begin{itemize}
\item the relation $\sim$ is an equivalence relation on $\mathcal{W}$,
\item for each given relation $u_i = v_i$, we have $u_i \sim v_i$,
\item the graph $\cG$ on $\mathcal{W}$ is compatible with $\sim$, in the sense that if $(a,b)$ is an edge of $\cG$ and $a \sim a'$, $b \sim b'$, then $(a',b')$ is also an edge of $\cG$,
\item for every vertex $w \in \mathcal{W}$, the loop $(w,w)$ is an edge of $\cG$,
\item the graph $\cG$ is undirected, that is, for every edge $(a,b)$ of $\cG$, the reversed edge $(b,a)$ is also an edge of $\cG$,
\item for any $a,b \in \mathcal{W}$ such that $ab$ is also in $\mathcal{W}$, the edges $(ab, a)$ and $(ab, b)$ are in $\cG$, and additionally for any $c$ such that $(a,c)$ is an edge of $\cG$, $(ab, c)$ is also an edge of $\cG$,
\item for any $a,b \in \mathcal{W}$ such that $ab$ is also in $\mathcal{W}$, if $(a,b)$ is an edge of $\cG$, then we have $a \sim ab$,
\item for any $a,b,a',b' \in \mathcal{W}$ such that $ab$ and $a'b'$ are also in $\mathcal{W}$, if $a \sim a'$ and $b \sim b'$ then $ab \sim a'b'$.
\end{itemize}
By Propositions \ref{meld-symmetry} and \ref{meld-graph}, if we have $r \sim s$ then the equation $r = s$ follows from the relations $u_i = v_i$.

To finish the proof, suppose that $r \not\sim s$. Then we will construct a meld $\bA$ in which the relations $u_i = v_i$ hold, but in which $r \ne s$. The underlying set of $\bA$ will be $\mathcal{W}/\!\sim \cup \{t\}$, where $t$ is a new symbol. We extend the graph $\cG/\!\sim$ to a graph $\cG_\bA$ on $\mathcal{W}/\!\sim \cup \{t\}$ by connecting the vertex $t$ to all other vertices.

We define the multiplication on $\bA$ as follows. For $a,b \in \mathcal{W}$ such that $ab \in \mathcal{W}$, we set $(a/\!\sim)\cdot (b/\!\sim) = ab/\!\sim$ (to see this is well-defined, we use the last bullet point above). For every pair of elements $a,b \in \bA$ connected by an edge of $\cG_\bA$, we set $a \cdot b = a$ and $b \cdot a = b$. Finally, for any pair $a,b \in \bA$ which are not connected by an edge of $\cG_\bA$ and which have not had a value assigned to $a\cdot b$ already, we set $a\cdot b = t$. Then by the closure propertes of $\sim$ and $\cG$, Proposition \ref{meld-construction} shows that $\bA$ is a meld, and by construction each relation $u_i/\!\sim\ = v_i/\!\sim$ holds in $\bA$ but $r/\!\sim\ \ne s/\!\sim$.
\end{proof}

\subsection{Dispersive algebras}

Now we come to the most difficult case, which I call the \emph{dispersive} case. We repeat the definition of the variety $\cD$ from \cite{minimal-clones-waldhauser} below.

\begin{defn} We define the variety $\cD$ of idempotent groupoids with basic operation $f$ satisfying
\begin{align}\label{D1}\tag{$\cD 1$}
f(x,f(y,x)) \approx f(f(x,y),x) \approx f(f(x,y),y) \approx f(f(x,y),f(y,x)) \approx f(x,y)
\end{align}
and
\begin{align}\label{D2}\tag{$\cD 2$}
\forall n \ge 0 \;\;\; f(x,f(\cdots f(f(x,y_1),y_2)\cdots,y_n)) \approx x.
\end{align}
\end{defn}

\begin{comment}
Note that the free algebra $\cF_\cD(x,y)$ is a four element algebra, with the following multiplication table.
\[
\begin{array}{c | c c c c} \cF_\cD(x,y) & x & y & xy & yx\\ \hline x & x & xy & x & xy \\ y & yx & y & yx & y \\ xy & xy & xy & xy & xy\\ yx & yx & yx & yx & yx\end{array}
\]
\end{comment}

The free algebra $\cF_\cD(x,y)$ is a four element algebra, with the multiplication table, labeled digraph, and undirected graph of two-element subalgebras displayed below.
\begin{center}
\begin{tabular}{ccccc}
\begin{tabular}{c | c c c c} $\cF_\cD(x,y)$ & $x$ & $y$ & $xy$ & $yx$\\ \hline $x$ & $x$ & $xy$ & $x$ & $xy$ \\ $y$ & $yx$ & $y$ & $yx$ & $y$ \\ $xy$ & $xy$ & $xy$ & $xy$ & $xy$\\ $yx$ & $yx$ & $yx$ & $yx$ & $yx$\end{tabular}
& \;\;\; &
\begin{tikzpicture}[scale=1.5,baseline=0.75cm]
  \node (x) at (0,0) {$x$};
  \node (y) at (1,0) {$y$};
  \node (xy) at (0,1) {$xy$};
  \node (yx) at (1,1) {$yx$};
  \draw [->] (x) edge ["{$y,yx$}"] (xy);
  \draw [->] (y) edge ["{$x,xy$}"'] (yx);
\end{tikzpicture}
& \;\;\; &
\begin{tikzpicture}[scale=1.5,baseline=0.75cm]
  \node (x) at (0,0) {$x$};
  \node (y) at (1,0) {$y$};
  \node (xy) at (0,1) {$xy$};
  \node (yx) at (1,1) {$yx$};
  \draw (x) -- (xy) -- (yx) -- (y);
\end{tikzpicture}
\end{tabular}
\end{center}

The equation \eqref{D2} can be restated in terms of right orbits:
\[
b \in O(a) \implies a\cdot b = a.
\]
More succinctly, this can be written as
\[
a \cdot O(a) = \{a\},
\]
and this can be expressed in terms of the labeled digraph $D_\bA$ as requiring that any edge which leaves $a$ must not be labeled by any element of $O(a)$.

\begin{prop}\label{clone-d2} If $f$ satisfies \eqref{D2}, then any $g \in \Clo_2^{\pi_1}(f)$ also satisfies \eqref{D2}.
\end{prop}
\begin{proof} Let $O_f(a)$ and $O_g(a)$ be the right orbits of $a$ with respect to $f$ and $g$, respectively. By Proposition \ref{prop-right-pi1}, we have $O_g(a) \subseteq O_f(a)$ for any $g \in \Clo_2^{\pi_1}(f)$. We will prove by induction on the definition of $g$ in terms of $f$ that $g(a,b) = a$ for all $b \in O_f(a)$. The base case $g = \pi_1$ is trivial. For the inductive step, suppose that we have already verified this for $g$ and that $h$ is any binary operation in $\Clo_2(f)$. Then for $b \in O_f(a)$ we have
\[
h(a,b) \in O_f(a) \cup O_f(b) \subseteq O_f(a),
\]
so
\[
f(g(a,b), h(a,b)) \in f(a, O_f(a)) = \{a\}
\]
by \eqref{D2}.
\end{proof}

Checking whether a two-generated algebra which satisfies \eqref{D2} has a surjective homomorphism to $\cF_\cD(x,y)$ is surprisingly natural in terms of the associated directed graph $D_\bA$.% from Definition \ref{defn-digraph}.

\begin{prop}\label{dispersive-hom-check} If $\bA = (A,f)$ and $f$ is a binary operation satisfying \eqref{D2}, and if $\bA$ is generated by two elements $a,b$, then any surjective homomorphism $\alpha : \bA \twoheadrightarrow \cF_\cD(x,y)$ must send $a$ and $b$ to $x$ and $y$, in some order. If such an $\alpha$ exists, then
\begin{itemize}
\item the right orbits $O(a)$ and $O(b)$ are disjoint,

\item for any $c \in O(a)$ and $d \in O(b)$ we have $f(a,d) \ne a$ and $f(b,c) \ne b$, and

\item there is no directed edge from $O(a) \setminus \{a\}$ to $a$ or from $O(b) \setminus \{b\}$ to $b$ in the directed graph $D_\bA$.
\end{itemize}
Conversely, if $\bA$ satisfies the three bullet points above, then there is a unique surjective homomorphism from $\bA$ to $\cF_\cD(x,y)$ with $\alpha(a) = x$ and $\alpha(b) = y$. This $\alpha$ satisfies $\alpha(c) = xy$ iff $c \in O(a) \setminus \{a\}$ and $\alpha(d) = yx$ iff $d \in O(b) \setminus \{b\}$.
\end{prop}
\begin{proof} First suppose that such an $\alpha$ exists. Since $x$ is not generated by $\cF_\cD(x,y)\setminus\{x\}$ and similarly for $y$, $\alpha$ must send one of the generators $a,b$ to $x$ and one to $y$. Suppose without loss of generality that $\alpha(a) = x$ and $\alpha(b) = y$. By Proposition \ref{prop-right-pi1}, for any $c \in O(a)$ there is some $g \in \Clo_2^{\pi_1}(f)$ such that $g(a,b) = c$, so
\[
\alpha(c) = \alpha(g(a,b)) = g(x,y) \in O(x),
\]
and similarly $d \in O(b)$ implies that $\alpha(d) \in O(y)$. Since $O(x)$ and $O(y)$ are disjoint in $\cF_\cD(x,y)$, we see that $O(a)$ and $O(b)$ must be disjoint in $\bA$.

For the second bullet point, just note that for any $d \in O(b)$ we have
\[
\alpha(f(a,d)) \in f(x,O(y)) = \{xy\},
\]
so $f(a,d) \ne a$, and similarly for any $c \in O(a)$ we have $f(b,c) \ne b$.

Next we check that for any $c \in O(a) \setminus \{a\}$, we have $\alpha(c) = xy$. We will prove this by induction on the length of the shortest directed path from $a$ to $c$ in $D_\bA$. For the base case, if $c = f(a,d)$ for some $d \in \bA$, then since $c \ne a$ we must have $d \not\in O(a)$ by \eqref{D2}, so $d \in O(b)$ and we have
\[
\alpha(c) = \alpha(f(a,d)) \in f(x, O(y)) = \{xy\}.
\]
For the inductive step, if we already know that $\alpha(c) = xy$ for some $c$, then for any $d \in \bA$ we have
\[
\alpha(f(c,d)) \in f(xy, \cF_\cD(x,y)) = \{xy\}.
\]
The proof of the inductive step also shows that there is no directed edge from $O(a) \setminus \{a\}$ to $a$, since $\alpha(a) \ne xy$. Symmetric arguments show that for any $d \in O(b) \setminus \{b\}$ we have $\alpha(d) = yx$, and that there is no directed edge from $O(b) \setminus \{b\}$ to $b$.

Now suppose that $\bA$ satisfies all three bullet points, and define $\alpha$ as in the theorem statement (that $\alpha$ is well-defined follows from the first bullet point in the theorem statement). We just need to verify that $\alpha$ is a homomorphism - the uniqueness will follow from the fact that any homomorphism is determined by where it sends the generators $a,b$. We just need to check that
\begin{itemize}
\item for $\alpha(c) \in O(x)$ we have $\alpha(f(a,c)) \in f(x,O(x)) = \{x\}$,

\item for $\alpha(d) \in O(y)$ we have $\alpha(f(a,d)) \in f(x,O(y)) = \{xy\}$, and

\item for $\alpha(c) = xy$ and any $d$ we have $\alpha(f(c,d)) \in f(xy, \cF_\cD(x,y)) = \{xy\}$,
\end{itemize}
as well as the corresponding statements with $x$ and $y$ reversed. The first of these conditions follows directly from \eqref{D2}, and by the definition of $\alpha$ the other two conditions correspond exactly to the last two bullet points from the theorem statement.
\end{proof}

\begin{prop}\label{dispersive-absorption} If $t$ is any term which satisfies
\[
t(x, y, ..., y) = x
\]
when interpreted in $\cF_\cD(x,y)$, then the identity $t(x_1, ..., x_n) \approx x_1$ follows from \eqref{D2}. In particular, every absorption identity which holds in $\cF_\cD(x,y)$ is a consequence of \eqref{D2}.
\end{prop}
\begin{proof} We will prove this by induction on the definition of $t$. We see that either $t$ is the first projection, or we can write $t = s \cdot u$ for terms $s$ and $u$ which satisfy
\[
%s(a_1, ..., a_n) = a_1, \;\;\; u(a_1, ..., a_n) \in \{a_1, b_1\}.
s(x, y, ..., y) = x, \;\;\; u(x, y, ..., y) \in \{x, xy\}
\]
in $\cF_\cD(x,y)$. Since $x,xy$ are not contained in the right orbit $O(y)$, any term $u$ such that
\[
%u(a_1, ..., a_n) \in \{a_1, b_1\}
u(x,y,...,y) \in \{x,xy\}
\]
must have the form
\[
u(x_1, ..., x_n) = (...((x_1\cdot v_1)\cdot v_2)... \cdot v_k)
\]
for some terms $v_1, ..., v_k$. Inductively, we see that the identity $s(x_1, ..., x_n) \approx x_1$ follows from \eqref{D2}, and since the identity
\[
x_1 \cdot u(x_1, ..., x_n) \approx x_1
\]
is an instance of \eqref{D2}, we see that $t(x_1, ..., x_n) \approx x_1$ follows from \eqref{D2} as well.
\end{proof}

\begin{prop}\label{dispersive-hom} If $\bA = (A,f)$ and $f$ is a binary operation satisfying \eqref{D2}, then for $a,b \in \bA$, the following are equivalent:
\begin{itemize}
\item there exists a surjective homomorphism $\alpha : \Sg_{\bA^2}\{(a,b),(b,a)\} \twoheadrightarrow \cF_\cD(x,y)$,

\item for every binary term $g$ of $\bA$ such that $g(a,b) = a$ and $g(b,a) = b$, the identity $g(x,y) \approx x$ follows from \eqref{D2}.
\end{itemize}
\end{prop}
\begin{proof} To make things easier to read, we write $a_2 = (a,b)$ and $b_2 = (b,a)$. Suppose first that such an $\alpha$ exists. By Proposition \ref{dispersive-hom-check}, $\alpha$ must send $a_2$ and $b_2$ to $x$ and $y$ in some order, say $\alpha(a_2) = x, \alpha(b_2) = y$. Then if $g(a,b) = a$ and $g(b,a) = b$, that is if $g(a_2,b_2) = a_2$, then we must have
\[
g(x,y) = \alpha(g(a_2,b_2)) = \alpha(a_2) = x
\]
in $\cF_\cD(x,y)$. Since all absorption identities which hold in $\cF_\cD(x,y)$ are consequences of \eqref{D2}, we see that $g(x,y) \approx x$ follows from \eqref{D2}.

Now suppose that the second bullet point holds. We just need to check that the algebra $\bA_2 = \Sg_{\bA^2}\{a_2,b_2\}$ satisfies the three conditions of Proposition \ref{dispersive-hom-check}. We will frequently use the fact that any identity which does not hold in $\cF_\cD(x,y)$ also can't be a consequence of \eqref{D2}.

To see that the right orbits of $a_2$ and $b_2$ in $\bA_2$ are disjoint, note that by Proposition \ref{prop-right-pi1} $c \in O(a_2)$ if and only if there is some $g \in \Clo_2^{\pi_1}(f)$ such that $g(a_2,b_2) = c$, and since the identity $f(y,g(x,y)) \approx y$ does not hold in $\cF_\cD(x,y)$ for $g \in \Clo_2^{\pi_1}(f)$, we have $f(b_2, g(a_2,b_2)) \ne b_2$, so by \eqref{D2} we have $c = g(a_2,b_2) \not\in O(b_2)$.

Next we need to check that $f(a_2,d) \ne a_2$ for any $d \in O(b_2)$. If $d \in O(b_2)$, then there is some $g \in \Clo_2^{\pi_1}(f)$ such that $d = g(b_2,a_2)$, and since the identity $f(x,g(y,x)) \approx x$ does not hold in $\cF_\cD(x,y)$ for $g \in \Clo_2^{\pi_1}(f)$ we see that $f(a_2,g(b_2,a_2)) \ne a_2$.%, so $f(a_2,d) \in O(a_2) \setminus\{a_2\}$.

Finally, we check that for any $c \in O(a_2) \setminus \{a_2\}$ and any $d \in \bA_2$ we have $f(c,d) \ne a_2$. Suppose otherwise, for the sake of contradiction - then there would be $c \in O(a_2)\setminus\{a_2\}$ and $d \in \Sg\{a_2,b_2\}$ such that $a_2 = f(c,d)$. Then by Proposition \ref{prop-right-pi1} there are $g \in \Clo_2^{\pi_1}(f)$ and $h \in \Clo_2(f)$ such that $g(a_2,b_2) = c$ and $h(a_2,b_2) = d$, so from
\[
a_2 = f(g(a_2,b_2), h(a_2,b_2)),
\]
we see that $f(g(x,y), h(x,y)) \approx x$ follows from \eqref{D2}. This can only occur if $g(x,y) = x$ in $\cF_\cD(x,y)$, in which case we have $c = g(a_2,b_2) = a_2$, which contradicts the assumption $c \in O(a_2)\setminus\{a_2\}$.
\end{proof}

Proposition \ref{dispersive-hom} is a big part of the reason that dispersive algebras are defined the way they are in Definition \ref{defn-dispersive}: the existence of a surjective homomorphism $\Sg_{\bA^2}\{(a,b),(b,a)\} \twoheadrightarrow \cF_\cD(x,y)$ ensures that there is no nontrivial binary reduct of $\bA$ such that $\{a,b\}$ becomes a projection subalgebra of the reduct.% (Recall that an algebra $\mathbb{A} = (A,f)$ is dispersive if $f$ satisfies \eqref{D2}, and if for any $a, b \in \bA$ which generate a nontrivial subalgebra of $\bA$ there is a surjective homomorphism $\Sg_{\bA^2}\{(a,b),(b,a)\} \twoheadrightarrow \cF_\cD(x,y)$.)

\begin{prop}\label{dispersive-reduct} If $f$ is dispersive, then any binary operation $g \in \Clo_2^{\pi_1}(f)$ is also dispersive.
\end{prop}
\begin{proof} By Proposition \ref{clone-d2}, any such $g$ automatically satisfies \eqref{D2}. Now suppose that $\{a,b\}$ does not form a subalgebra of $(A,g)$, in which case $g$ must be nontrivial. Then since $f$ is dispersive, there is a surjective homomorphism
\[
\alpha : \Sg_{\bA^2}\Big\{\begin{bmatrix} a\\ b\end{bmatrix}, \begin{bmatrix} b\\ a\end{bmatrix}\Big\} \twoheadrightarrow \cF_\cD(x,y)
\]
which maps $(a,b)$ to $x$. Let $\alpha_g$ be the restriction of $\alpha$ to the subreduct $\Sg_{(A,g)^2}\{(a,b),(b,a)\}$, which defines a homomorphism to the reduct $(\{x,y,xy,yx\}, g)$ of $\cF_\cD(x,y)$. Then by Proposition \ref{dispersive-hom} and the fact that $g$ is nontrivial we see that $g((a,b),(b,a)) \ne (a,b)$, so after applying $\alpha_g$ we see that $g(x,y) \ne x$ in $\cF_\cD(x,y)$. Since $g(x,y) \in O(x)$ by Proposition \ref{prop-right-pi1}, we see that $g(x,y) = xy$ in $\cF_\cD(x,y)$, so $\alpha_g$ is surjective and the reduct $(\{x,y,xy,yx\},g)$ is isomorphic to $\cF_\cD(x,y)$.
\end{proof}

\begin{prop}\label{dispersive-nice-terms} If $\bA = (A,f)$ is a nontrivial dispersive algebra, then each of the three operations $g_1,g_2,g_3 \in \Clo_2^{\pi_1}(f)$ from Proposition \ref{nice-terms} is necessarily nontrivial.
\end{prop}
\begin{proof} We just need to check that we do not have $f(x,y) \preceq x$ in the free algebra $\cF_{\cV(\bA)}(x,y)$. By the definition of a dispersive algebra and the assumption that $\bA$ is nontrivial, there is some surjective homomorphism $\alpha : \cF_{\cV(\bA)}(x,y) \twoheadrightarrow \cF_{\cD}(x,y)$. Therefore any hypothetical directed path from $f(x,y)$ to $x$ in $\cF_{\cV(\bA)}(x,y)$ would be mapped by $\alpha$ to a directed path from $xy$ to $x$ in $\cF_{\cD}(x,y)$, which is impossible.
\end{proof}

Note that the definition of a dispersive algebra makes sense even for infinite algebras. The class of dispersive algebras is very close to forming a variety.

\begin{thm}\label{dispersive-pseudovariety} The class of (possibly infinite) dispersive algebras is closed under the following constructions.
\begin{itemize}
\item Any subalgebra of a dispersive algebra is dispersive.

\item Any product $\prod_{i \in I} \bA_i$ of dispersive algebras $\bA_i$ is dispersive.

\item If $\bA$ is a dispersive algebra and $\theta$ is a congruence on $\bA$ such that the intersection of any finitely generated subalgebra of $\bA$ with any congruence class of $\theta$ is finite, then $\bA/\theta$ is a dispersive algebra.
\end{itemize}
In particular, the finite dispersive algebras form a pseudovariety.
\end{thm}
\begin{proof} The claim about subalgebras follows directly from the definition of a dispersive algebra. For the claim about products, suppose that $\bA = \prod_{i \in I} \bA_i$ and that $a,b \in \bA$ generate a nontrivial subalgebra of $\bA$. Since $f$ acts as first projection on every two-element dispersive algebra (by the identity $f(x,f(x,y)) \approx x$, which follows from \eqref{D2}), there must be at least one $i \in I$ such that $\{a_i,b_i\}$ is not a subalgebra of $\bA_i$. Then $\pi_{(i,i)}$ defines a surjective homomorphism from $\Sg_{\bA^2}\{(a,b),(b,a)\}$ to $\Sg_{\bA_i^2}\{(a_i,b_i),(b_i,a_i)\}$, and since $\bA_i$ is assumed to be dispersive this has a surjective homomorphism to $\cF_\cD(x,y)$.

Now we handle the claim about quotients. So suppose $\bA = (A,f), \theta$ are as in the last bullet point, and let $C,D$ be a pair of congruence classes of $\theta$. By Proposition \ref{dispersive-hom}, checking the existence of a surjective homomorphism
\[
\Sg_{\bA/\theta}\{(C,D),(D,C)\} \twoheadrightarrow \cF_\cD(x,y)
\]
is equivalent to checking that for every binary term $g$ such that $g(C,D) \subseteq C$ and $g(D,C) \subseteq D$, the identity $g(x,y) \approx x$ follows from \eqref{D2}.
Suppose for contradiction that $\bA$ has a nontrivial term $g$ such that $g(C,D) \subseteq C$ and $g(D,C) \subseteq D$, but that $\Sg_{\bA/\theta}\{C,D\} \ne \{C,D\}$. Pick $a \in C$ and $b \in D$ such that
\[
\Sg_{\bA^2}\Big\{\begin{bmatrix} a\\ b\end{bmatrix},\begin{bmatrix} b\\ a\end{bmatrix}\Big\} \cap \begin{bmatrix} C\\ D\end{bmatrix}
\]
has minimal cardinality (using the finiteness assumption on intersections of finitely generated subalgebras and congruence classes of $\theta$).

Since $g$ is nontrivial and $\bA$ is dispersive, either $\{a,b\}$ is a subalgebra of $\bA$, or the pairs $(g(a,b), g(b,a))$ and $(g(b,a), g(a,b))$ generate a proper subalgebra of $\Sg_{\bA^2}\{(a,b),(b,a)\}$, which is contained in the preimage of the subalgebra $\{xy,yx\} \le \cF_\cD(x,y)$ under the surjective homomorphism $\Sg_{\bA^2}\{(a,b),(b,a)\} \twoheadrightarrow \cF_\cD(x,y)$. Since $a, g(a,b) \in C$ and $b, g(b,a) \in D$, the second case contradicts the choice of $a,b$ (otherwise $(a,b) \in C \times D$ is in the subalgebra generated by $(g(a,b),g(b,a))$ and $(g(b,a),g(a,b))$, in which case $(b,a)$ is as well). Therefore $\{a,b\}$ must be a subalgebra of $\bA$. Since $\theta$ is a congruence and $a \in C, b \in D$, this implies that $\{C,D\}$ is a subalgebra of $\bA/\theta$, which is a contradiction.
\end{proof}

\begin{rem} Not every quotient of an infinite dispersive algebra is dispersive. In fact, there is an infinite dispersive algebra $\bA_\infty$ which is a subproduct of a sequence of finite dispersive algebras $\bA_n$, such that $\bA_\infty$ has a four-element quotient which is not dispersive (and is instead a meld). The algebra $\bA_n$ has underlying set $\{a_0, ..., a_n, b_0, ..., b_n\}$, with the basic operation given by
\[
\begin{array}{c | c c c c} \bA_n & a_j & a_n & b_j & b_n\\ \hline a_i & a_i & a_i & a_n & a_{i+1}\\ a_n & a_n & a_n & a_n & a_n\\ b_i & b_n & b_{i+1} & b_i & b_i\\ b_n & b_n & b_n & b_n & b_n\end{array}
\]
for any $i,j < n$. To check that $\bA_n$ is dispersive, note that
\[
O(a_i) = \{a_j \mid j \ge i\},
\]
so $\bA_n$ satisfies \eqref{D2} since $a_i \cdot a_j = a_i$ for all $i,j$. Furthermore, for $i < n$ we have
\[
a_i \cdot b_j \in O(a_{i+1})
\]
for all $j \le n$, so as long as either $i < n$ or $j < n$ we have a map from $\Sg_{\bA_n^2}\{(a_i,b_j),(b_j,a_i)\}$ to $\cF_\cD(x,y)$ which sends $(a_i,b_j)$ to $x$, sends $(b_j,a_i)$ to $y$, and sends all other elements to $x \cdot y$ or $y \cdot x$ based on whether the first coordinate is in $O(a_i)$ or $O(b_j)$. (The reader may check that for $n \ge 2$, the clone of the dispersive algebra $\bA_n$ is \emph{not} a minimal clone, by applying Proposition \ref{nice-terms}.)

We can now define the infinite algebra $\bA_\infty$ to be the subalgebra of the product $\prod_n \bA_n$ which is generated by the elements $(a_0, a_0, ...)$ and $(b_0, b_0, ...)$. First define a congruence $\sim$ on $\bA_\infty$ which identifies sequences $c$ and $d$ whenever $c_i = d_i$ for all sufficiently large $i$. Then each element of $\bA_\infty/\!\!\sim$ is either of the form $a^i$ or $b^i$ for some $i \in \mathbb{N} \cup \{\infty\}$, where
\[
a^i = (a_{\max(i,1)}, a_{\max(i,2)}, ...)/\!\!\sim
\]
and $b^i$ is defined similarly. Then the equivalence relation $\theta$ which has equivalence classes $\bigcup_{i < \infty} a^i, \bigcup_{i < \infty} b^i, a^\infty, b^\infty$ is easily checked to be a congruence of $\bA_\infty$, and $\bA_\infty/\theta$ is isomorphic to the free meld on two generators.
\end{rem}

\begin{thm}\label{dispersive-nice} Being dispersive is a nice property.
\end{thm}
\begin{proof} First we verify that we can efficiently check whether or not a given algebra $\bA = (A,f)$ is dispersive. Although \eqref{D2} is an infinite family of identities, we can check whether it holds by computing the right orbit $O(a)$ and verifying that $f(a,O(a)) = \{a\}$ for each $a \in \bA$. Once we have checked that \eqref{D2} holds in $\bA$, we examine each pair $a,b \in \bA$ such that $\{a,b\}$ is not a two-element subalgebra of $\bA$, compute the labeled digraph associated to the algebra $\Sg_{\bA^2}\{(a,b),(b,a)\}$, and check whether or not there is a surjective homomorphism from $\Sg_{\bA^2}\{(a,b),(b,a)\}$ to $\cF_\cD(x,y)$ using Proposition \ref{dispersive-hom-check}.

The next step is to check that $\cF_\cD(x,y)$ appears in the variety generated by any nontrivial dispersive algebra $\bA$. To see this, note that by the definition of a dispersive algebra, if $\cF_\cD(x,y)$ is not a subquotient of $\bA^2$, then every pair of elements $a,b \in \bA$ forms a two-element subalgebra of $\bA$, and $f$ must act as first projection on each of these subalgebras as a consequence of \eqref{D2}, in which case $\bA$ is a projection algebra.

That the finite dispersive algebras form a pseudovariety was proved in Theorem \ref{dispersive-pseudovariety}. All that is left is to check that if $\bA$ is nontrivial, then for any nontrivial term $g \in \Clo(f)$, there is some nontrivial $f' \in \Clo(g)$ such that $f'$ is also dispersive.

By Rosenberg's Theorem \ref{rosenberg-types}, there is some nontrivial $t \in \Clo(g)$ which is either a unary operation, a majority operation, an affine operation, a binary operation, or a semiprojection. It's easy to see that $t$ can't be unary (since $f$ is idempotent) and that $t$ can't be majority or affine (since there is a two-element projection algebra in the variety generated by $\bA$). Proposition \ref{dispersive-absorption} shows that $t$ can't be a semiprojection of arity at least $3$. Thus $t$ is binary, and Proposition \ref{dispersive-reduct} shows that either $t(x,y)$ or $t(y,x)$ is dispersive.
\begin{comment}
To finish the proof, we need to check that any nontrivial binary operation $t \in \Clo_2^{\pi_1}(f)$ is dispersive. By Proposition \ref{clone-d2}, any such $t$ automatically satisfies \eqref{D2}. Now suppose that $\{a,b\}$ does not form a subalgebra of $(A,t)$. Then since $f$ is dispersive, there is a surjective homomorphism
\[
\alpha : \Sg_{\bA^2}\{(a,b),(b,a)\} \twoheadrightarrow \cF_\cD(x,y)
\]
which maps $(a,b)$ to $x$. Let $\alpha_t$ be the restriction of $\alpha$ to the subreduct $\Sg_{(A,t)^2}\{(a,b),(b,a)\}$, which defines a homomorphism to the reduct $(\{x,y,xy,yx\}, t)$ of $\cF_\cD(x,y)$. Then by Proposition \ref{dispersive-hom} and the fact that $t$ is nontrivial we see that $t((a,b),(b,a)) \ne (a,b)$, so after applying $\alpha_t$ we see that $t(x,y) \ne x$ in $\cF_\cD(x,y)$. Since $t(x,y) \in O(x)$ by Proposition \ref{prop-right-pi1}, we see that $t(x,y) = xy$ in $\cF_\cD(x,y)$, so $\alpha_t$ is surjective and the reduct $(\{x,y,xy,yx\},t)$ is isomorphic to $\cF_\cD(x,y)$.
\end{comment}
\end{proof}

Now we will investigate the structure of the digraphs $D_\bA$ attached to clone-minimal dispersive algebras.

\begin{prop}\label{dispersive-strong-inert} If $\bA$ is a dispersive algebra, then every strongly connected component of $D_\bA$ is a projection subalgebra of $\bA$.
\end{prop}
\begin{proof} If $a,b$ are in the same strongly connected component of $D_\bA$, then $b \in O(a)$ and $a \in O(b)$, so \eqref{D2} implies that $ab = a$ and $ba = b$.
\end{proof}

\begin{prop} If $\bA = (A,f)$ is a dispersive algebra, then for any $a \ne b \in \bA$ and any $c \in \bA$, we have $(c,c) \not\in \Sg_{\bA^2}\{(a,b),(b,a)\}$. In particular, we have $f(a,b) \ne f(b,a)$.
\end{prop}
\begin{proof} Suppose for contradiction that $\bA$ has a term $g \in \Clo_2^{\pi_1}(f)$ with $g(a,b) = g(b,a) = c$. %Then since $\bA$ satisfies \eqref{D2} and $\cF_\cD(x,y) \in \Var(\bA)$, exactly one of the identities $f(x,g(x,y)) \approx x, f(x,g(y,x)) \approx x$ holds, depending on whether the description of $g$ in terms of $f$ has an $x$ or a $y$ as its leftmost variable, and we may assume without loss of generality that $f(x,g(x,y)) \approx x$.
Then the binary term $h(x,y) = f(x,g(y,x))$ is nontrivial in $\cF_\cD(x,y)$, but $h(a,b) = f(a,c) = a$ and $h(b,a) = f(b,c) = b$ since $c \in O(a) \cap O(b)$, which is a contradiction to Proposition \ref{dispersive-hom}.
\end{proof}

\begin{prop} If $\bA = (A,f) = \Sg\{a,b\}$ is a dispersive algebra with $a \ne b$, then $b \not\in O(a)$, that is, $a \not\preceq b$. In particular, $f(a,b) \ne b$.
\end{prop}
\begin{proof} If $b \in O(a)$, then $\bA = O(a) \cup O(b) = O(a)$. But then $f(a,c) = a$ for all $c$ in $\bA$, so $O(a) = \{a\}$, a contradiction.
\end{proof}

\begin{defn} For every $n \ge 1$, we define the clone-minimal dispersive \emph{hoop} algebras $\mathbb{H}_n, \mathbb{H}_n'$ on the set
\[
H_n = \{a,b,c_0, ..., c_{n-1},d_0, ..., d_{n-1}\}
\]
by the multiplication tables
\[
\begin{array}{c c c}
\begin{array}{c | c c c c} \mathbb{H}_n & a & b & c_j & d_j\\ \hline a & a & c_0 & a & a \\ b & b & b & b & b \\ c_i & d_i & c_i & c_i & c_i\\ d_i & d_i & c_{i+1} & d_i & d_i\end{array} & &
\begin{array}{c | c c c c} \mathbb{H}_n' & a & b & c_j & d_j\\ \hline a & a & d_0 & a & a \\ b & b & b & b & b \\ c_i & d_i & c_i & c_i & c_i\\ d_i & d_i & c_{i+1} & d_i & d_i\end{array}
\end{array}
\]
where the indices are considered cyclically modulo $n$.
\end{defn}

For $n = 2$, the labeled digraphs associated to the hoop algebras $\mathbb{H}_2, \mathbb{H}_2'$ are pictured side by side below.
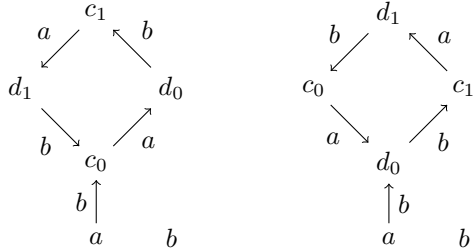
\begin{center}
\begin{tabular}{ccc}
\begin{tikzpicture}[scale=1]
  \node (a) at (0,0) {$a$};
  \node (b) at (1,0) {$b$};
  \node (c0) at (0,1) {$c_0$};
  \node (c1) at (0,3) {$c_1$};
  \node (d0) at (1,2) {$d_0$};
  \node (d1) at (-1,2) {$d_1$};
  \draw [->] (a) edge ["{$b$}"] (c0) (c0) edge ["{$a$}"'] (d0) (d0) edge ["$b$"'] (c1) (c1) edge ["{$a$}"'] (d1) (d1) edge ["$b$"'] (c0);
\end{tikzpicture}
& \;\;\;\;\; &
\begin{tikzpicture}[scale=1]
  \node (a) at (0,0) {$a$};
  \node (b) at (1,0) {$b$};
  \node (c0) at (-1,2) {$c_0$};
  \node (c1) at (1,2) {$c_1$};
  \node (d0) at (0,1) {$d_0$};
  \node (d1) at (0,3) {$d_1$};
  \draw [->] (a) edge ["{$b$}"'] (d0) (c0) edge ["{$a$}"'] (d0) (d0) edge ["$b$"'] (c1) (c1) edge ["{$a$}"'] (d1) (d1) edge ["$b$"'] (c0);
\end{tikzpicture}
\end{tabular}
\end{center}
Visual inspection shows that the hoop algebras satisfy \eqref{D2}. To check that they are dispersive, first note that every pair of elements other than the pair $a,b$ generates a proper subalgebra of size two or three which is isomorphic to a subalgebra of the dispersive algebra $\cF_\cD(x,y)$. So we only need to check that $\Sg\{(a,b),(b,a)\}$ has a surjective homomorphism to $\cF_\cD(x,y)$, which follows from Propositions \ref{dispersive-hom-check}, \ref{dispersive-hom} as a consequence of $O(b) = \{b\}$, $ab \ne a$, and the fact that there is no directed path from $ab$ back to $a$.

Note that the algebras $\mathbb{H}_n, \mathbb{H}_n'$ are term-equivalent: if $f$ is the basic operation of $\mathbb{H}_n$, then the basic operation $f'$ of $\mathbb{H}_n'$ is given by
\[
f'(x,y) = f(f(x,y),x).
\]
The basic operation of $\mathbb{H}_n$ satisfies the identities $f(f(x,y),y) \approx f(x,y) \approx f(f(x,y),f(y,x))$, while the basic operation of $\mathbb{H}_n'$ satisfies the identity $f'(f'(x,y),x) \approx f'(x,y)$, but the next result shows that no nontrivial term of either algebra satisfies all of these identities simultaneously.

\begin{prop}\label{hoop-reduct} If $\mathbb{H}_n = (H_n, f)$ is one of the hoop algebras, then for any nontrivial $g \in \Clo_2^{\pi_1}(f)$ the reduct $(H_n, g)$ is either isomorphic to $\mathbb{H}_n$ or to $\mathbb{H}_n'$.

In particular, $\Clo(g)$ has the same size as $\Clo(f)$ for any nontrivial $g \in \Clo_2^{\pi_1}(f)$, so $\Clo(f)$ is a minimal clone.
\end{prop}
\begin{proof} Since every pair of elements other than $a,b$ generates a subalgebra isomorphic to a subalgebra of $\cF_\cD(x,y)$, $g$ agrees with $f$ on any such pair. Since $g(b,a) \in O(b) = \{b\}$ by Proposition \ref{prop-right-pi1}, the operation $g$ is completely determined by the value of $g(a,b) \in O(a) \setminus \{a\}$. If $g(a,b) = c_i$ for some $i$, then $(H_n, g)$ is isomorphic to $\mathbb{H}_n$ by cyclically permuting the $c$s and $d$s, and similarly if $g(a,b) = d_i$ for some $i$ then $(H_n,g)$ is isomorphic to $\mathbb{H}_n'$.
\end{proof}

\begin{prop} Suppose that $\bA = (A,f)$ is a dispersive clone-minimal algebra, with $\bA$ generated by $a,b$ such that $O(b) = \{b\}$. Then either $|\bA| \le 3$ or $\bA$ is isomorphic to one of the hoop algebras $\mathbb{H}_n, \mathbb{H}_n'$ for some $n \ge 1$.
\end{prop}
\begin{proof} By Propositions \ref{nice-terms} and \ref{dispersive-nice-terms} we can find a nontrivial $g \in \Clo_2^{\pi_1}(f)$ such that $g(g(x,y),y) \approx g(x,y)$ and such that $g(a,b)$ is contained in a maximal strongly connected component of $D_\bA$. Let $c = g(a,b)$, and note that $a \preceq c$ since $g \in \Clo_2^{\pi_1}(f)$. Then since $a \preceq c$, we have $g(a,x) = a$ for all $x \in O(c)$. Additionally, from $g \in \Clo_2^{\pi_1}(f)$ and the assumption $O(b) = \{b\}$, we see that $g(b,x) = b$ for all $x$.

Thus $O(c) \cup \{a,b\}$ is closed under $g$, and must therefore be equal to $\bA$ since $f \in \Clo(g)$. Furthermore, since $O(c)$ is strongly connected it must be a projection subalgebra by Proposition \ref{dispersive-strong-inert}. If we consider the digraph $D_{(A,g)}$, we see that the edge labels are subsets of $\{a,b\}$, and the identity $g(g(x,y),y) \approx g(x,y)$ implies that no two consecutive edges can share a label.

Assume now that $|\bA| > 3$, so $|O(c)| \ge 2$. Since every element of $O(c)$ has indegree and outdegree at least one, we see that no edge in $O(c)$ is labelled $\{a,b\}$, and that in fact each element of $O(c)$ must have outdegree exactly one. Thus $O(c)$ is a directed cycle, with edge labels alternating between $\{a\}$ and $\{b\}$, so $(A,g)$ is isomorphic to the hoop algebra $\mathbb{H}_n$ for $n = |O(c)|/2$. To finish the argument, note that $f \in \Clo_2^{\pi_1}(g)$ and Proposition \ref{hoop-reduct} imply that $\bA = (A,f)$ is also isomorphic to a hoop algebra.
\end{proof}

\begin{cor} If $\bA$ is a clone-minimal dispersive algebra and $b \preceq a$ in $\bA$, then $\Sg_\bA\{a,b\}$ either has size at most $3$, or is isomorphic to one of the hoop algebras $\mathbb{H}_n, \mathbb{H}_n'$ for some $n \ge 1$.
\end{cor}

The previous result puts a constraint on the ways an element $a$ can interact with its right orbit $O(a)$. In some cases we can put an even stronger constraint on how the right orbits can behave.

\begin{defn} We say that a dispersive algebra $\bA$ is \emph{inert} if $O(a)$ is a projection subalgebra for each $a \in \bA$.
\end{defn}

\begin{prop} If $\bA$ is a dispersive algebra, then the union of all of the maximal strongly connected components of $D_\bA$ forms an inert subalgebra of $\bA$.
\end{prop}

The next result will be proved in Appendix \ref{a-dispersive}.

\begin{thm}\label{inert-small} If $\bA$ is a clone-minimal inert dispersive algebra which is generated by two elements $a,b$, and if $|O(b)| \le 2$, then $|O(a)| \le 2$ and $\bA$ is isomorphic to a subalgebra of $\cF_\cD(x,y)$.
\end{thm}

As soon as both right orbits have size at least $3$, it is possible to find interesting examples of minimal inert dispersive algebras which are generated by two elements.

\begin{defn} For every $n \ge 1$ and every $0 \le k < n$, we define the minimal inert dispersive \emph{wheel} algebras $\mathbb{W}_{n,k}^a, \mathbb{W}_n^s$ on the set $\{a,b,c_0, ..., c_{n-1},d_0, ..., d_{n-1}\}$ by the multiplication tables
\[
\begin{array}{c c c}
\begin{array}{c | c c c c} \mathbb{W}_{n,k}^a & a & b & c_j & d_j\\ \hline a & a & c_0 & a & c_{j+1} \\ b & d_k & b & d_j & b \\ c_i & c_i & c_i & c_i & c_i\\ d_i & d_i & d_i & d_i & d_i\end{array} & &
\begin{array}{c | c c c c} \mathbb{W}_n^s & a & b & c_j & d_j\\ \hline a & a & c_0 & a & c_{j+1} \\ b & d_0 & b & d_{j+1} & b \\ c_i & c_i & c_i & c_i & c_i\\ d_i & d_i & d_i & d_i & d_i\end{array}
\end{array}
\]
where the indices are considered cyclically modulo $n$. We define $\mathbb{W}_n^a$ to be $\mathbb{W}_{n,0}^a$.
\end{defn}

For $n = 2$, the labeled digraphs of the wheel algebras $\mathbb{W}_2^a, \mathbb{W}_2^s$ are pictured side by side below (to understand why they are named ``wheel'' algebras, the reader should try to draw $\mathbb{W}_6^s$ as symmetrically as possible).
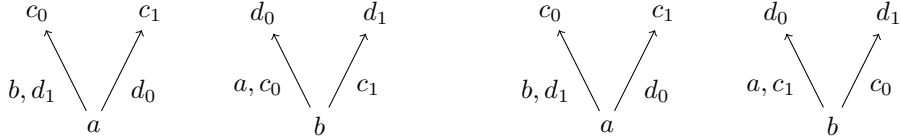
\begin{center}
\begin{tabular}{ccc}
\begin{tikzpicture}[scale=1.5]
  \node (a) at (0,0) {$a$};
  \node (b) at (2,0) {$b$};
  \node (c0) at (-0.5,1) {$c_0$};
  \node (c1) at (0.5,1) {$c_1$};
  \node (d0) at (1.5,1) {$d_0$};
  \node (d1) at (2.5,1) {$d_1$};
  \draw [->] (a) edge ["{$b,d_1$}"] (c0) (a) edge ["{$d_0$}"'] (c1);
  \draw [->] (b) edge ["{$a,c_0$}"] (d0) (b) edge ["{$c_1$}"'] (d1);
\end{tikzpicture}
& \;\;\;\;\;\;\; &
\begin{tikzpicture}[scale=1.5]
  \node (a) at (0,0) {$a$};
  \node (b) at (2,0) {$b$};
  \node (c0) at (-0.5,1) {$c_0$};
  \node (c1) at (0.5,1) {$c_1$};
  \node (d0) at (1.5,1) {$d_0$};
  \node (d1) at (2.5,1) {$d_1$};
  \draw [->] (a) edge ["{$b,d_1$}"] (c0) (a) edge ["{$d_0$}"'] (c1);
  \draw [->] (b) edge ["{$a,c_1$}"] (d0) (b) edge ["{$c_0$}"'] (d1);
\end{tikzpicture}
\end{tabular}
\end{center}
The argument which shows that the wheel algebras are dispersive is almost identical to the argument for the hoop algebras: \eqref{D2} follows by visual inspection, and every pair of elements other than $a,b$ generates a subalgebra of size at most $3$ which is isomorphic to a subalgebra of the dispersive algebra $\cF_\cD(x,y)$. That $\Sg\{(a,b),(b,a)\}$ maps surjectively onto $\cF_\cD(x,y)$ follows from Propositions \ref{dispersive-hom-check}, \ref{dispersive-hom} as a consequence of $a \not\in a \cdot O(b)$ and the fact that there is no edge from any element to $a$.

To check that the wheel algebras are inert, we just need to check that the right orbits $O(a) = \{a, c_0, ..., c_{n-1}\}$ and $O(b) = \{b, d_0, ..., d_{n-1}\}$ are projection subalgebras, since every right orbit is contained in one of $O(a), O(b)$. The next result implies that the wheel algebras are clone-minimal.

\begin{prop} The wheel algebra $\mathbb{W}_{n,k}^a$ is isomorphic to $\mathbb{W}_{n,n-1-k}^a$. If $n$ is odd, then $\mathbb{W}_n^s$ is isomorphic to $\mathbb{W}_{n,\frac{n-1}{2}}^a$. For every $n \ge 2$ and every $k \ne \frac{n-1}{2}$, $\Sg_{(\mathbb{W}_{n,k}^a)^2}\{(a,b),(b,a)\}$ is isomorphic to $\mathbb{W}_{2n}^s$.

Finally, if $\mathbb{A} = (A,f) \in \{\mathbb{W}_{n,k}^a, \mathbb{W}_n^s\}$, then any nontrivial term $g \in \Clo_2^{\pi_1}(f)$ defines an algebra $(A,g)$ which is isomorphic to $\bA$.
\end{prop}
\begin{proof} Note that since the wheel algebras are inert and since the elements $c_i,d_i$ have no outgoing edges, the only values of the multiplication which are not determined by the digraph structure (ignoring the labels) are those of the form $a\cdot d$ with $d \in O(b)$ and $b\cdot c$ with $c \in O(a)$. So it's enough to check a candidate isomorphism between wheel algebras by checking that it preserves the digraph structure and that it is compatible with the multiplication on pairs in $\{a\}\times O(b) \cup \{b\} \times O(a)$.

An isomorphism $\alpha$ from $\mathbb{W}_{n,k}^a$ to $\mathbb{W}_{n,n-1-k}^a$ is given by taking
\[
\alpha(a) = b, \;\; \alpha(b) = a, \;\; \alpha(c_i) = d_{i-k-1}, \;\; \alpha(d_i) = c_{i-k},
\]
where the indices are taken modulo $n$. To check that $\alpha$ is compatible with the multiplications $f_{n,k}^a, f_{n,n-k-1}^a$ of $\mathbb{W}_{n,k}^a$ and $\mathbb{W}_{n,n-1-k}^a$, note that we have
\begin{align*}
\alpha(f_{n,k}^a(a,b)) &= \alpha(c_0) = d_{n-k-1} = f_{n,n-k-1}^a(b,a),\\% = f_{n,n-k-1}^a(\alpha(a),\alpha(b)),\\
\alpha(f_{n,k}^a(b,a)) &= \alpha(d_k) = c_0 = f_{n,n-k-1}^a(a,b),\\% = f_{n,n-k-1}^a(\alpha(b),\alpha(a)),\\
\alpha(f_{n,k}^a(a,d_i)) &= \alpha(c_{i+1}) = d_{i-k} = f_{n,n-k-1}^a(b,c_{i-k}),\\% = f_{n,n-k-1}^a(\alpha(a),\alpha(d_i)),\\
\alpha(f_{n,k}^a(b,c_i)) &= \alpha(d_i) = c_{i-k} = f_{n,n-k-1}^a(a,d_{i-k-1}).% = f_{n,n-k-1}^a(\alpha(b),\alpha(c_i)).
\end{align*}

If $n$ is odd, then $2$ is invertible in $\ZZ/n$, and we can define an isomorphism $\beta$ from $\mathbb{W}_n^s$ to $\mathbb{W}_{n,\frac{n-1}{2}}^a$ by taking
\[
\beta(a) = a, \;\; \beta(b) = b, \;\; \beta(c_i) = c_{\frac{i}{2}}, \;\; \beta(d_i) = d_{\frac{i-1}{2}}.
\]
To check that $\beta$ is compatible with the multiplications $f_n^s, f_{n,\frac{n-1}{2}}^a$ of $\mathbb{W}_n^s$ and $\mathbb{W}_{n,\frac{n-1}{2}}^a$, note that we have
\begin{align*}
\beta(f_n^s(a,b)) &= \beta(c_0) = c_0 = f_{n,\frac{n-1}{2}}^a(a,b),\\% = f_{n,\frac{n-1}{2}}^a(\beta(a),\beta(b)),\\
\beta(f_n^s(b,a)) &= \beta(d_0) = d_{\frac{n-1}{2}} = f_{n,\frac{n-1}{2}}^a(b,a),\\% = f_{n,\frac{n-1}{2}}^a(\beta(b),\beta(a)),\\
\beta(f_n^s(a,d_i)) &= \beta(c_{i+1}) = c_{\frac{i+1}{2}} = f_{n,\frac{n-1}{2}}^a(a,d_{\frac{i-1}{2}}),\\% = f_{n,\frac{n-1}{2}}^a(\beta(a),\beta(d_i)),\\
\beta(f_n^s(b,c_i)) &= \beta(d_{i+1}) = d_{\frac{i}{2}} = f_{n,\frac{n-1}{2}}^a(b,c_{\frac{i}{2}}).% = f_{n,\frac{n-1}{2}}^a(\beta(b),\beta(c_i)).
\end{align*}

For the third isomorphism, for any $n,k$ with $k \ne \frac{n-1}{2}$ we define $\gamma$ from $\mathbb{W}_{2n}^s$ to $\Sg_{(\mathbb{W}_{n,k}^a)^2}\{(a,b),(b,a)\}$ by
\begin{align*}
\gamma(a) &= (a,b), \;\; \gamma(c_{2i}) = (c_i,d_{i+k}), \;\; \gamma(c_{2i+1}) = (c_{i+k+1},d_i),\\
\gamma(b) &= (b,a), \;\; \gamma(d_{2i}) = (d_{i+k},c_i), \;\; \gamma(d_{2i+1}) = (d_i, c_{i+k+1}).
\end{align*}
We use the assumption $k \ne \frac{n-1}{2}$ in order to ensure that $\gamma$ is injective: if $\gamma(c_{2i}) = \gamma(c_{2j+1})$ for some $i,j$, then $c_i = c_{j+k+1}$ and $d_{i+k} = d_j$, which can only happen if $i-j \equiv k+1 \equiv -k \pmod{n}$.

To check that $\gamma$ is compatible with the multiplications $f_{2n}^s, f_{n,k}^a$ of $\mathbb{W}_{2n}^s$ and $\mathbb{W}_{n,k}^a$, note that we have
\begin{align*}
\gamma(f_{2n}^s(a,b)) &= \gamma(c_0) = (c_0,d_k) = f_{n,k}^a((a,b), (b,a)),\\% = f_{n,k}^a(\gamma(a),\gamma(b)),\\
\gamma(f_{2n}^s(a,d_{2i})) &= \gamma(c_{2i+1}) = (c_{i+k+1},d_i) = f_{n,k}^a((a,b), (d_{i+k},c_i)),\\% = f_{n,k}^a(\gamma(a),\gamma(d_{2i})),\\
\gamma(f_{2n}^s(a,d_{2i+1})) &= \gamma(c_{2i+2}) = (c_{i+1},d_{i+k+1}) = f_{n,k}^a((a,b), (d_i,c_{i+k+1})).% = f_{n,k}^a(\gamma(a),\gamma(d_{2i+1})).
\end{align*}

Now we turn to the last claim. Since every pair of elements of $\bA$ other than $a,b$ generates a subalgebra isomorphic to a subalgebra of $\cF_\cD(x,y)$, any nontrivial $g \in \Clo_2^{\pi_1}(f)$ agrees with $f$ on any such pair. Thus any such $g$ is completely determined by the values of $g(a,b) \in O(a)$ and $g(b,a) \in O(b)$. As a consequence, we see that the natural map from $\Clo_2(f) = \cF_{\cV(\bA)}(x,y)$ to the algebra $\Sg_{\bA^2}\{(a,b),(b,a)\}$ is an isomorphism.

If $\bA$ has an automorphism that swaps $a$ and $b$ - that is, if $\bA$ is $\mathbb{W}_n^s$ or $\mathbb{W}_{n,\frac{n-1}{2}}^a$ - then $\Sg_{\bA^2}\{(a,b),(b,a)\}$ is isomorphic to $\bA$ and every nontrivial $g \in \Clo_2^{\pi_1}(f)$ is determined by the $j$ such that $g(a,b) = c_j$. In this case, $(A,f)$ is isomorphic to $(A,g)$ by the map which sends $a,b$ to themselves, and increases the index in each $c_i,d_i$ by $j$ modulo $n$.

Now suppose that $\bA$ is $\mathbb{W}_{n,k}^a$ for some $k \ne \frac{n-1}{2}$. Making use of the automorphism $\gamma$ from $\mathbb{W}_{2n}^s$ to $\Sg_{\bA^2}\{(a,b),(b,a)\}$, we see that every nontrivial $g \in \Clo_2^{\pi_1}(f)$ is determined by the $j \in \ZZ/2n$ such that
\[
(g(a,b), g(b,a)) = g(\gamma(a), \gamma(b)) = \gamma(c_j).
\]
If $j$ is even, say $j = 2i$, then by the definition of $\gamma$ we have $g(a,b) = c_i$ and $g(b,a) = d_{i+k}$, so $(A,f)$ is isomorphic to $(A,g)$ by shifting the indices on the $c$s and $d$s forward by $i$ modulo $n$. If $j$ is odd, say $j = 2i+1$, then we have $g(a,b) = c_{i+k+1}$ and $g(b,a) = d_i$, so $\mathbb{W}_{n,n-1-k}^a$ is isomorphic to $(A,g)$ by shifting the indices forward by $i+k+1$ modulo $n$, and composing this isomorphism with $\alpha$ gives us an isomorphism between $(A,f)$ and $(A,g)$.
\end{proof}

\begin{rem} The author has become convinced, through a process of exhaustive (and exhausting) casework, of the following claim: up to isomorphism, the only clone-minimal inert dispersive algebras generated by two elements with size $\le 6$ are $\cF_\cD(x,y)$, its three element subalgebra $\{x,xy,yx\}$, and the wheel algebras $\mathbb{W}_2^a$, $\mathbb{W}_2^s$. It would be wonderful to have a formal proof of this claim that was less than infinitely long.
\end{rem}

\begin{comment}
\begin{prop} Up to isomorphism, the only minimal inert dispersive algebras generated by two elements with size $\le 6$ are $\cF_\cD(x,y)$, its three element subalgebra $\{x,xy,yx\}$, and the wheel algebras $\mathbb{W}_2^a$, $\mathbb{W}_2^s$.
\end{prop}
\begin{proof}[Sketch] The proof is by exhaustive (and exhausting) casework. First we make the assumption that $f(f(x,y),y) \approx f(x,y)$ and $f(x,y)$ is in a maximal strongly connected component of $\Clo_2^{\pi_1}(\bA)$. Then we draw all possible digraphs $D_\bA$ for which the outdegrees of vertices are sufficiently small based on their indegrees and the number of elements in the right orbits of the generating elements of $\bA$. Then for each such digraph, we consider possible labellings of the edges, taking advantage of chances to apply this result inductively whenever a pair of elements generates a proper subalgebra of $\bA$ (which can often be detected from the digraph $D_\bA$ without using the edge labels). The toughest cases to rule out are the ones in which each right orbit is strongly connected.
\end{proof}
% TODO: prove?
\end{comment}

We'll end this subsection with a few conjectures about clone-minimal dispersive algebras which seem to be fiendishly difficult to either prove or construct counterexamples to. Recall that a \emph{weakly connected component} of a directed graph is just a connected component of the undirected graph which we get by forgetting the directions of the edges.

\begin{conj} If $\bA$ is a clone-minimal dispersive algebra, then for any $a \ne b \in \bA$, there is a surjective homomorphism from $\Sg\{a,b\}$ to a two element projection algebra. Equivalently, the digraph $D_{\Sg\{a,b\}}$ has exactly two weakly connected components, that is, $O(a) \cap O(b) = \emptyset$ in $\Sg_\bA\{a,b\}$.
\end{conj}

\begin{conj} If $\bA$ is a clone-minimal dispersive algebra, and if there is some $c \in \Sg\{a,b\}$ such that $\{c,d\}$ is a two element projection subalgebra for all $d \in \Sg\{a,b\}$, then $\Sg\{a,b\} = \{a,b,c\}$. It's enough to show that at least one of $(a,c), (b,c)$ is contained in $\Sg_{\bA^2}\{(a,b),(b,a)\}$.
\end{conj}

\begin{conj} If $\bA$ is a clone-minimal dispersive algebra which is generated by two elements, then $D_\bA$ has at least three strongly connected components.
\end{conj}

\section{Proof of the classification}\label{s-classification}

Since the proof is long, we outline the main ideas below.
\begin{itemize}
\item First we prove that if a clone-minimal algebra $\bA$ is not Taylor and not a rectangular band, then $\Clo_2^{\pi_1}(\bA)$ becomes a robust concept (Proposition \ref{prop-pi1-exact} and Corollary \ref{not-rect-pi1}).

\item Next we show that if $\Clo(\bA)$ also does not contain a nontrivial partial semilattice operation, then \eqref{D2} holds in the free algebra on two generators (Theorem \ref{pi1}). The core of the argument (Lemma \ref{lem-pi1-not-pi2}) is based on Proposition \ref{minimal-nontrivial} and Corollary \ref{cor-linked}.

\item We then consider the case where either of the composition operators $*_c, *_1$ gives $\Clo_2^{\pi_1}(\bA)$ a group structure, and prove that in this case $\bA$ must be a $p$-cyclic groupoid (Theorem \ref{p-cyclic}, Corollary \ref{groupy-2}). Finiteness is used in an unexpected way in Lemma \ref{find-fixed-point}, which is proved by a counting argument involving the orbit-stabilizer theorem.

\item If $\Clo_2^{\pi_1}(\bA)$ does not form a group under $*_c, *_1$, then we can apply Corollary \ref{binary-iteration} to find nontrivial operations satisfying $g(g(x,y),y) \approx g(x,y)$ or $t(t(x,y),t(y,x)) \approx t(x,y)$. We show that in this case, any deviation from being dispersive forces $\bA$ to be a meld (Theorem \ref{graphy}, Lemmas \ref{graphy-gen}, \ref{graphy-gen-2}). The proof technique relies heavily on Proposition \ref{minimal-nontrivial}.

\item We complete the proof by showing that if $\bA$ is not Taylor, not a rectangular band, has no nontrivial partial semilattice operation, is not a $p$-cyclic groupoid, and is not a meld, then $\bA$ must be a dispersive algebra (Theorem \ref{all-but-dispersive}).
\end{itemize}

%Recall from Definition \ref{set-defn} that we say that an algebra $\bA$ is a \emph{projection algebra} if all of its basic operations are projections, and that we say that $\bA$ is \emph{nontrivial} if it is not a projection algebra.

\begin{comment}
We will frequently make use of Proposition \ref{absorb-prop}, which states that if $\bA$ is a minimal clone and $\bB \in \Var(\bA)$ is nontrivial, then for any term $t \in \Clo(\bA)$ such that the identity
\[
t^\bB(x_1, ..., x_n) \approx x_1
\]
holds in $\bB$, we have $t(x_1, ..., x_n) \approx x_1$ in $\bA$ as well.

\begin{defn} We say that an idempotent binary operation $f$ is a \emph{rectangular band} if it satisfies the identity
\[
f(f(x,y),f(z,w)) \approx f(x,w).
\]
\end{defn}
\end{comment}

\begin{prop}\label{prop-rectangular-projections} If a clone-minimal algebra $\bA$ has a term $f$ such that there are algebras $\bB_1, \bB_2 \in \Var(\bA)$, each with at least two elements, such that $f^{\bB_1}$ is first projection and $f^{\bB_2}$ is second projection, then $\bA$ is a rectangular band.
\end{prop}
\begin{proof} Since $f$ acts as different projections on $\bB_1$ and $\bB_2$, $f$ must be nontrivial, so $\Clo(\bA) = \Clo(f)$. Assume without loss of generality that $|\bB_1| = |\bB_2| = 2$, and let $\bB = \bB_1\times \bB_2$. Also, assume that $f$ is binary (otherwise replace it with $f(x,y,...,y)$). On $\bB$, $f$ satisfies the identity
\[
f(f(f(u,x),y),f(z,f(w,u))) \approx u,
\]
so this identity must hold on $\bA$ as well, by Proposition \ref{absorb-prop}. By Proposition \ref{rect-band-absorb}, this identity implies that $\bA$ is a rectangular band.
\begin{comment} Similarly, we have
\[
f(f(x,w),x) \approx x
\]
and
\[
f(w,f(x,w)) \approx w,
\]
so
\[
f(f(x,y),f(z,w)) \approx f(f(f(f(x,w),x),y),f(z,f(w,f(x,w)))) \approx f(x,w).\qedhere
\]
\end{comment}
\end{proof}

%Thus for non-Taylor clone-minimal algebras $\bA$ which are \emph{not} rectangular bands, for every operation $t \in \Clo(\bA)$ there is a unique $i$ such that $t$ acts as the $i$th projection on some projection algebra $\bB \in \Var(\bA)$ with $|\bB| \ge 2$. This prompts the following definition.

\begin{comment}
\begin{defn}\label{defn-clo2-pi1} For any algebra $\bA$, let $\Clo_2^{\pi_1}(\bA)$ be the set of binary terms of $\bA$ which restrict to the first projection on some algebra $\bB \in \Var(\bA)$ of size at least $2$.
\end{defn}

%Proposition \ref{prop-pi1} can be used to clarify the relationship between $\Clo_2^{\pi_1}(\bA)$ and $\Clo_2^{\pi_1}(f)$.

\begin{prop} If $f \in \Clo_2^{\pi_1}(\bA)$ and $g \in \Clo_2^{\pi_1}(f)$, then $g \in \Clo_2^{\pi_1}(\bA)$.
\end{prop}
\begin{proof} If $f$ acts as first projection on some $\bB \in \Var(\bA)$, then by Proposition \ref{prop-pi1} so does $g$.
\end{proof}
\end{comment}

%Recall from Definition \ref{defn-clo2-pi1} that $\Clo_2^{\pi_1}(\bA)$ is defined to be the set of binary terms of $\bA$ which restrict to the first projection on some algebra $\bB \in \Var(\bA)$ of size at least $2$.

\begin{prop}\label{prop-pi1-exact} If $\bA$ is a binary clone-minimal algebra which is not Taylor and is not a rectangular band, then for every $f \in \Clo_2(\bA)$, exactly one of $f(x,y) \in \Clo_2^{\pi_1}(\bA)$ or $f(y,x) \in \Clo_2^{\pi_1}(\bA)$ is true.

In this case, under the isomorphism between $\Clo_2(\bA)$ and the free algebra $\cF_{\cV(\bA)}(x,y)$ which sends $\pi_1$ to $x$ and $\pi_2$ to $y$, $\Clo_2^{\pi_1}(\bA)$ is a congruence class of the unique congruence $\sim$ on $\cF_{\cV(\bA)}(x,y)$ such that $\cF_{\cV(\bA)}(x,y)/\!\!\sim$ is a two-element projection algebra.
\end{prop}
\begin{proof} If $\bA$ is not Taylor, then there must be a two-element projection algebra $\bB$ in $\Var(\bA)$, so for every $f \in \Clo_2(\bA)$, at least one of $f(x,y) \in \Clo_2^{\pi_1}(\bA)$ or $f(y,x) \in \Clo_2^{\pi_1}(\bA)$ is true. By Proposition \ref{prop-rectangular-projections}, they can't both be true unless $\bA$ is a rectangular band.

For the second statement, since $\bB$ is generated by two elements, there is a surjective homomorphism $\alpha$ from $\cF_{\cV(\bA)}(x,y)$ to $\bB$. Then for $f \in \Clo_2^{\pi_1}(\bA)$, we have $\alpha(f(x,y)) = f(\alpha(x),\alpha(y)) = \alpha(x)$ since $f$ restricts to first projection on $\bB$, so $\Clo_2^{\pi_1}(\bA) = \alpha^{-1}(\alpha(x))$ is a congruence class of $\ker \alpha$.
\end{proof}

\begin{cor}\label{not-rect-pi1} If $\bA$ is a binary clone-minimal algebra which is not a rectangular band, then every $f \in \Clo_2^{\pi_1}(\bA)$ acts as first projection on every projection algebra $\bB \in \Var(\bA)$.
\end{cor}

It is now finally time to apply Proposition \ref{minimal-nontrivial} and Corollary \ref{cor-linked}.

\begin{lem}\label{lem-pi1-not-pi2} Suppose that $\Clo(\bA) = \Clo(f)$ is a binary minimal clone which does not contain any nontrivial partial semilattice operations, and assume further that $f$ acts as first projection on every proper subalgebra or quotient of $\bA$. Then for any $a \ne b \in \bA$ we have
\[
f(a,b) \ne b.
\]
\end{lem}
\begin{proof} Suppose for contradiction that $a \ne b$ and $f(a,b) = b$. Since $f$ acts as first projection on every proper subalgebra of $\bA$, we must have $\Sg_\bA\{a,b\} = \bA$.%Then $\{a,b\}$ can't be a projection subalgebra of $\bA$. Since every proper subalgebra of $\bA$ is a projection algebra, we must therefore have $\Sg_\bA\{a,b\} = \bA$.% (note that this together with $f \in \Clo_2^{\pi_1}(\bA)$ implies $f \ne \pi_1, \pi_2$, so $f$ generates $\Clo(\bA)$)

Define $\bS \le_{sd} \bA^2$ by
\[
\bS = \Sg_{\bA^2}\Big\{\begin{bmatrix} a\\ b\end{bmatrix}, \begin{bmatrix} b\\ a\end{bmatrix}\Big\}.
\]
If $\bS$ is the graph of an automorphism of $\bA$ then we must have $f(b,a) = a$, so $\{a,b\}$ is a subalgebra of $\bA$ on which $f$ acts as the second projection, contradicting our assumption on $f$.

If $\bS$ is not linked, then the linking congruence of $\bS$ defines a nontrivial congruence $\sim$ on $\bA$. Since $f$ acts as first projection on $\bA/\!\!\sim$, we have $f(a/\!\!\sim, b/\!\!\sim) = a/\!\!\sim$. But then $b = f(a,b) \sim a$ and since $a,b$ generate $\bA$ we see that every element of $\bA$ is in the same congruence class of $\sim$, which contradicts the assumption that $\sim$ is nontrivial.

Thus $\bS$ must be linked. Then by Corollary \ref{cor-linked} and the fact that $\bA$ has a proper singleton subalgebra (since $\bA$ is idempotent and nontrivial), we see that $\bA$ has a proper subalgebra $\bB < \bA$ such that $\bB + \bS = \bA$.
\begin{comment}
 A standard argument shows that if $\bS$ is linked and $\bA$ is finite, then there must be some proper subalgebra $\bB < \bA$ such that
\[
\pi_2(\bS \cap (\bB \times \bA)) = \pi_1(\bS \cap (\bA \times \bB)) = \bA.
\]
The standard argument is as follows: first, let $\bB' \le \bA$ be maximal among proper subalgebras of $\bA$. If
\[
\pi_2(\bS \cap (\bB' \times \bA)) = \bA,
\]
then we can take $\bB = \bB'$, since $\bS$ is symmetric. Otherwise, we define $\bB$ by
\[
\bB = \pi_2(\bS \cap (\bB' \times \bA)),
\]
and note that since $\bS$ is linked, $\bB'$ can't be a union of linked components of $\bS$, so $\bB'$ must be properly contained in $\pi_1(\bS \cap (\bA \times \bB))$. Since $\bB'$ was maximal among proper subalgebras of $\bA$, we have
\[
\pi_1(\bS \cap (\bA \times \bB)) = \bA,
\]
and once again the symmetry of $\bS$ shows that $\bB$ is acceptable.
\end{comment}

Choose $c,d \in \bB$ with $a \in c + \bS$ and $b \in d + \bS$, that is, with $(a,c), (b,d) \in \bS$. Since $f$ acts as first projection on $\bB$, we have $f(c,d) = c$, so
\[
\begin{bmatrix} b\\ c \end{bmatrix} = f\Big(\begin{bmatrix} a\\ c\end{bmatrix}, \begin{bmatrix} b\\ d\end{bmatrix}\Big) \in \bS.
\]
If $b,c$ generate $\bA$, then from $(a,b), (a,c) \in \bS$ we see that $(a,a) \in \bS$, and by Lemma \ref{lem-semilattice-iteration} this contradicts the assumption that $\bA$ has no nontrivial partial semilattice terms. Otherwise $f$ acts as first projection on $\{b,c\}$, so
\[
\begin{bmatrix} b\\ b \end{bmatrix} = f\Big(\begin{bmatrix} a\\ b\end{bmatrix}, \begin{bmatrix} b\\ c\end{bmatrix}\Big) \in \bS,
\]
and again by Lemma \ref{lem-semilattice-iteration} this contradicts the assumption that $\bA$ has no nontrivial partial semilattice terms.
\end{proof}

\begin{cor}\label{cor-pi1} Suppose that $\Clo(\bA) = \Clo(f)$ is a binary minimal clone which does not contain any nontrivial partial semilattice operations, and assume further that $f$ acts as first projection on every proper subalgebra or quotient of $\bA$. If $\pi_2 \not\in \Clo_2^{\pi_1}(f)$, then for any $g,h \in \Clo_2^{\pi_1}(f)$ we have
\[
f(g(x,y), h(x,y)) \approx g(x,y).
\]
\end{cor}
\begin{proof} Suppose for the sake of contradiction that there are $g,h \in \Clo_2^{\pi_1}(f)$ such that $f(g(x,y), h(x,y)) \not\approx g(x,y)$. In particular, there must be some $a,b \in \bA$ such that
\[
f(g(a,b),h(a,b)) \ne g(a,b).
\]
Since $f$ acts as first projection on every proper subalgebra of $\bA$, the elements $g(a,b)$ and $h(a,b)$ must generate $\bA$. Thus there is some $t \in \Clo_2(f)$ such that
\[
t(g(a,b),h(a,b)) = b.
\]
Since $t \in \Clo_2(f)$, either $t(x,y) \in \Clo_2^{\pi_1}(f)$ or $t(y,x) \in \Clo_2^{\pi_1}(f)$, and either way we see that the term $u$ defined by
\[
u(x,y) = t(g(x,y),h(x,y))
\]
has $u \in \Clo_2^{\pi_1}(f)$ and $u(a,b) = b$. Then $u$ is nontrivial since $\pi_2 \not\in \Clo_2^{\pi_1}(f)$, so Lemma \ref{lem-pi1-not-pi2} applied to $u$ gives us a contradiction.
\end{proof}

\begin{thm}\label{pi1} If $\Clo(\bA)$ is a binary minimal clone which is not a rectangular band and which does not contain any nontrivial partial semilattice operations, then for any $f,g \in \Clo_2^{\pi_1}(\bA)$ we have
\[
f(x, g(x,y)) \approx x.
\]
If $f,g$ are nontrivial, then we also have $f(g(x,y),x) \not\approx x$, and more generally for any $h \in \Clo_2^{\pi_1}(\bA)$ we have
\[
f(g(x,y),h(x,y)) \not\approx x.
\]
\end{thm}
\begin{proof} By the assumption $f \in \Clo_2^{\pi_1}(\bA)$, we have $\Clo_2^{\pi_1}(f) \subseteq \Clo_2^{\pi_1}(\bA)$, so $\pi_2 \not\in \Clo_2^{\pi_1}(f)$.

By Proposition \ref{minimal-nontrivial} we can choose some $\bB \in \Var(\bA)$ such that $\bB$ is not a projection algebra, but such that every proper subalgebra or quotient of $\bB$ is a projection algebra. By Lemma \ref{lem-semilattice-iteration} $\bB$ can't have any partial semilattice terms if $\bA$ doesn't have any partial semilattice terms, and by Corollary \ref{not-rect-pi1} $f$ and $g$ act as first projection on every projection subalgebra or quotient of $\bB$ if $\bA$ isn't a rectangular band. By Proposition \ref{absorb-prop}, the terms $f,g$ are nontrivial in $\bB$ if and only if they are nontrivial in $\bA$, and for every $t \in \Clo_2^{\pi_1}(f)$ we see that $t$ can't restrict to $\pi_2$ on $\bB$.

By Corollary \ref{cor-pi1} applied to $\bB$, the absorption identity $f(x,g(x,y)) \approx x$ holds in $\bB$, so by Proposition \ref{absorb-prop} it also holds in $\bA$. Additionally, by Corollary \ref{cor-pi1} the identity
\[
f(g(x,y),h(x,y)) \approx g(x,y)
\]
holds in $\bB$, so
\[
f(g(x,y),h(x,y)) \not\approx x
\]
on $\bB$, since $g$ is nontrivial on $\bB$. Thus we also have
\[
f(g(x,y),h(x,y)) \not\approx x
\]
on $\bA$.
\begin{comment}
%First we prove that $f(x,g(x,y)) \approx x$ holds on $\bB$. 
Suppose for contradiction that there are $a, b \in \bB$ such that
\[
f(a,g(a,b)) \ne a.
\]
Then since $f$ acts as first projection on every proper subalgebra of $\bB$ we must have $\Sg_\bB\{a, g(a,b)\} = \bB$, so there is some binary term $t$ such that
\[
t(a,g(a,b)) = b.
\]
Defining the binary term $h$ by
\[
h(x,y) = t(x,g(x,y)),
\]
we see that since one of $t(x,y) \in \Clo_2^{\pi_1}(f)$ or $t(y,x) \in \Clo_2^{\pi_1}(f)$ holds, we have $h \in \Clo_2^{\pi_1}(\bA)$ either way. But we also have
\[
h(a,b) = b,
\]
which contradicts Lemma \ref{lem-pi1-not-pi2} applied to $\bB$.

Now suppose, for contradiction, that
\[
f(g(x,y),h(x,y)) \approx x
\]
for some nontrivial $f,g \in \Clo_2^{\pi_1}(\bA)$ and some (possibly trivial) $h \in \Clo_2^{\pi_1}(\bA)$. Again, note that by Proposition \ref{absorb-prop} $f,g$ are nontrivial on $\bA$ if and only if they are nontrivial on $\bB$. Since $g$ is nontrivial on $\bB$, there must be some $a,b \in \bB$ such that
\[
g(a,b) \ne a.
\]
Then since
\[
f(g(a,b), h(a,b)) = a \ne g(a,b),
\]
we see that $\{g(a,b), h(a,b)\}$ can't be a projection subalgebra of $\bB$. By the choice of $\bB$, this means that
\[
\Sg_\bB\{g(a,b), h(a,b)\} = \bB,
\]
so there is some binary term $t$ such that
\[
t(g(a,b), h(a,b)) = b.
\]
Defining the binary term $u$ by
\[
u(x,y) = t(g(x,y),h(x,y)),
\]
we see that $u \in \Clo_2^{\pi_1}(\bA)$ and
\[
u(a,b) = b,
\]
which contradicts Lemma \ref{lem-pi1-not-pi2}.
\end{comment}
\end{proof}

Recall from Definition \ref{p-cyclic-defn} that we say that an idempotent binary operation $f$ is a \emph{$p$-cyclic groupoid} if it satisfies the identities
\begin{align*}
f(x,f(y,z)) &\approx f(x,y),\\
f(f(x,y),z) &\approx f(f(x,z),y),
\end{align*}
and
\[
f(\cdots f(f(x,y),y)\cdots, y) \approx x,
\]
where there are $p$ $y$s in the last identity.

\begin{thm}\label{p-cyclic} If a binary minimal clone on a finite set is not a rectangular band and does not have any nontrivial term $f$ satisfying the identity
\[
f(f(x,y),f(y,x)) \approx f(x,y),
\]
then it is a $p$-cyclic groupoid for some prime $p$.
\end{thm}

We will prove this in a series of lemmas, all using the following assumption:

\begin{itemize}
\item[($*$)] $\Clo(\bA)$ is a binary minimal clone which is not a rectangular band and $\bA$ has no nontrivial term operation satisfying the identity $f(f(x,y),f(y,x)) \approx f(x,y)$.
\end{itemize}

We will make use of the circular composition operator $*_c$, which we defined in equation \eqref{star-c} by
\[
(f *_c g)(x,y) = f(g(x,y),g(y,x)).
\]

\begin{comment}
\begin{defn} For any $f,g \in \Clo_2(\bA)$, define their circular composition $f *_c g \in \Clo_2(\bA)$ by
\[
(f *_c g)(x,y) = f(g(x,y),g(y,x)).
\]
\end{defn}
\end{comment}

\begin{prop}\label{star-reverse} If $\bA$ satisfies $(*)$, then $\Clo_2(\bA)$ and $\Clo_2^{\pi_1}(\bA)$ form groups under $*_c$. In particular, for any binary term $f \in \Clo(\bA)$ there exists a binary term $f^- \in \Clo(\bA)$ such that
\[
f^-(f(x,y),f(y,x)) \approx f(f^-(x,y),f^-(y,x)) \approx x.
\]
\end{prop}
\begin{proof} First we check that the circular composition $*_c$ is associative:
\begin{align*}
\big(f*_c(g*_ch)\big)(x,y) &\approx f\big((g*_ch)(x,y), (g*_ch)(y,x)\big)\\
&\approx f\Big(g\big(h(x,y),h(y,x)\big),g\big(h(y,x),h(x,y)\big)\Big)\\
&\approx (f*_cg)\big(h(x,y), h(y,x)\big)\\
&\approx \big((f*_cg)*_ch\big)(x,y).
\end{align*}

Since $\Clo_2(\bA)$ and $\Clo_2^{\pi_1}(\bA)$ are closed under $*_c$, we see that $\Clo_2(\bA)$ and $\Clo_2^{\pi_1}(\bA)$ form semigroups under $*_c$, with $\pi_1$ as the identity. To see that they form groups, note that by the finiteness of $\Clo_2(\bA)$ we can apply Proposition \ref{unary-iteration} to see that there is some $N > 0$ such that $f^{*_cN} = f^{*_c2N}$ for all $f$ in $\Clo_2(\bA)$. If $f^{*_cN} \ne \pi_1$ for any $f \in \Clo_2(\bA)$, then $f^{*_cN}$ contradicts the assumption ($*$). Otherwise, if we take $f^- = f^{*_c(N-1)}$ then we have $f *_c f^- = f^- *_c f = f^{*N} = \pi_1$ for all $f \in \Clo_2(\bA)$.
\end{proof}

\begin{prop}\label{star-not-taylor-semi} If $\bA$ satisfies $(*)$, then $\bA$ is not a Taylor algebra and $\bA$ has no nontrivial partial semilattice operations.
\end{prop}
\begin{proof} That $\bA$ is not Taylor follows from the fact that every binary minimal clone which is also Taylor has a term $f$ such that $f(x,y) \approx f(y,x)$, by Theorem \ref{thm-taylor-minimal}.

Suppose for contradiction that $\bA$ has a nontrivial partial semilattice operation $s$, then by the Proposition \ref{star-reverse} there is a term $s^- \in \Clo(\bA)$ such that
\[
s^-(s(x,y),s(y,x)) \approx x.
\]
Replacing $y$ with $s(x,y)$ in the identity above, we see that
\[
%s(x,y) \approx s^-(s(x,y),s(x,y)) \approx s^-(s(x,s(x,y)),s(s(x,y),x)) \approx x,
x \approx s^-(s(x,s(x,y)),s(s(x,y),x)) \approx s^-(s(x,y),s(x,y)) \approx s(x,y),
\]
contradicting the assumption that $s$ is nontrivial.
\end{proof}

\begin{prop} If $\bA$ satisfies $(*)$, $\bB \in \Var(\bA)$ is not a projection algebra, and $f, g \in \Clo_2(\bA)$ satisfy $f^\bB = g^\bB$, then $f(x,y) \approx g(x,y)$.
\end{prop}
\begin{proof} We have
\[
f^-(g(x,y),g(y,x)) = x
\]
for any $x,y \in \bB$, so by Proposition \ref{absorb-prop} this absorption identity also holds in $\bA$% (otherwise, $f^-(f,g)$ would generate a strictly smaller nontrivial clone)
. Thus, we have
\begin{align*}
f(x,y) &\approx f(f^-(g(x,y),g(y,x)), f^-(g(y,x),g(x,y)))\\
&\approx g(x,y).\qedhere
\end{align*}
\end{proof}

\begin{prop}\label{prop-check-two} If $\bA$ satisfies $(*)$ and $t \in \Clo_3(\bA)$, $f \in \Clo_2(\bA)$ satisfy
\[
t(x,y,z) = f(x,y)
\]
whenever two of $x,y,z$ are equal, then $t(x,y,z) \approx f(x,y)$.
\end{prop}
\begin{proof} Not that the assumption of clone-minimality implies that $\bA$ has no ternary semiprojections. We can then apply Proposition \ref{prop-no-semiprojection} to conclude that the identities
\[
f^-(t(x,y,z),f(y,x)) \approx x
\]
and
\[
f^-(f(y,x),t(x,y,z)) \approx y
\]
must hold, since otherwise the terms on the left hand sides would be semiprojections. Thus we have
\begin{align*}
t(x,y,z) &\approx f(f^-(t(x,y,z),f(y,x)), f^-(f(y,x),t(x,y,z)))\\
&\approx f(x,y).\qedhere
\end{align*}
\end{proof}

\begin{lem}\label{fixed-point} If $\bA$ satisfies $(*)$ and there are nontrivial terms $f,g \in \Clo_2^{\pi_1}(\bA)$ satisfying
\[
f(x,g(y,x)) \approx f(x,y),
\]
then $\bA$ is a $p$-cyclic groupoid for some prime $p$.
\end{lem}
\begin{proof} By Theorem \ref{pi1} and Proposition \ref{star-not-taylor-semi}, we have
\[
f(x,g(x,y)) \approx x,
\]
so
\[
f(x,g(y,z)) = f(x,y)
\]
whenever two of $x,y,z$ are equal. Thus, by Proposition \ref{prop-check-two} we have
\[
f(x,g(y,z)) \approx f(x,y).
\]
By iteratively applying the above identity we deduce that
\begin{align*}
f(x,g(g(\cdots g(y,z_1)\cdots,z_{n-1}),z_n)) &\approx f(x,g(\cdots g(y,z_1)\cdots,z_{n-1}))\\
&\cdots\\
&\approx f(x,g(y,z_1))\\
&\approx f(x,y)
\end{align*}
for any sequence $z_1, ..., z_n$. Since $g$ is nontrivial and $\bA$ is clone-minimal, for any $h \in \Clo_2^{\pi_1}(\bA)$ we have $h \in \Clo_2^{\pi_1}(g)$, so 
\begin{equation}
h \in \Clo_2^{\pi_1}(\bA) \;\;\; \implies \;\;\; f(x,h(y,z)) \approx f(x,y).\label{eq-right-orbit}
\end{equation}

Then for any $h \in \Clo_2^{\pi_1}(\bA)$ we also have
\[
f(x,h^-(y,z)) \approx f(x,y)
\]
by \eqref{eq-right-orbit} since $h^- \in \Clo_2^{\pi_1}(\bA)$, so
\begin{align*}
f(h(x,y),x) &\approx f\big(h(x,y), (h^- *_c h)(x,y)\big)\\
&\approx f\Big(h(x,y),h^-\big(h(x,y),h(y,x)\big)\Big)\\
&\approx f\big(h(x,y),h(x,y)\big)\\
&\approx h(x,y),
\end{align*}
and in particular we have
\[
f(f(x,y),z) = f(f(x,z),y)
\]
whenever two of $x,y,z$ are equal. Since $f^- \in \Clo_2^{\pi_1}(\bA)$, we have
\[
f(z,f^-(y,x)) \approx f(z,y)
\]
by \eqref{eq-right-orbit}, and plugging in $z = f^-(x,y)$ we get
\begin{align*}
f(f^-(x,y),y) &\approx f(f^-(x,y),f^-(y,x))\\
&\approx (f *_c f^-)(x,y)\\
&\approx x.
\end{align*}
Since $\bA$ is finite, for any $y \in \bA$ the unary polynomials $x \mapsto f(x,y)$ and $x \mapsto f^-(x,y)$ must therefore be inverse to each other, so we also have
\[
f^-(f(x,y),y) \approx x.
\]
\begin{comment}
and since there is some $n > 0$ such that $f^- = f^{*n}$, we have
\begin{align*}
f^-(f(x,y),y) &\approx (f^{*(n-1)}*f)(f(x,y), y)\\
&\approx f^{*(n-1)}(f(f(x,y), y), f(y, f(x,y)))\\
&\approx f^{*(n-1)}(f(f(x,y), f(y,x)), f(y, x))\\
&\approx \cdots\\
&\approx f^-(f(x,y),f(y,x))\\
&\approx x.
\end{align*}
(using here that $f^- = f^{*n}$ for some $n$ - alternatively, we could deduce that $f^-(f(x,y),y) \approx x$ from $f(f^-(x,y),y) \approx x$ from the finiteness of $\cF_\bA(x,y)$)
\end{comment}
Then the equation $f(f(x,y),z) = f(f(x,z),y)$ implies
\begin{align*}
f^-(f^-(f(f(x,y),z),y),z) &= f^-(f^-(f(f(x,z),y),y),z)\\
&= f^-(f(x,z),z)\\
&= x
\end{align*}
whenever two of $x,y,z$ are equal. Since $\bA$ has no nontrivial semiprojections, the equation above holds identically by Proposition \ref{prop-no-semiprojection}. Therefore we have
\begin{align*}
f(f(x,y),z) &\approx f(f^-(f(f(x,y),z),y),y)\\
&\approx f(f(f^-(f^-(f(f(x,y),z),y),z),z),y)\\
&\approx f(f(x,z),y).
\end{align*}
This proves the identity
\begin{equation}
f(f(x,y),z) \approx f(f(x,z),y).\label{eq-commutative}
\end{equation}
Now define a sequence of functions $f_n$ by $f_0 = \pi_1$, $f_1 = f$, and
\[
f_{n+1}(x,y) = f(f_n(x,y),y),
\]
so that $f_n = f^{*_1n}$. In fact, we also have $f_n = f^{*_cn}$, since $f_n \in \Clo_2^{\pi_1}(\bA)$ implies
\[
f(f_n(x,y),y) \approx f(f_n(x,y),f_n(y,x))
\]
by \eqref{eq-right-orbit}.

We claim that every $t \in \Clo_2^{\pi_1}(\bA)$ can be written as $f_n$ for some $n$. We prove this by induction on the construction of $t$ as a term in $\Clo_2^{\pi_1}(f)$. For the inductive step, there are two cases: either $t(x,y) \approx f(f_i(x,y), f_j(x,y))$ for some $i,j$, or $t(x,y) \approx f(f_i(x,y), f_j(y,x))$ for some $i,j$. In the first case, we can apply \eqref{eq-right-orbit} and \eqref{eq-commutative} to see
\begin{align*}
f(f_i(x,y), f_j(x,y)) &\approx f(f_i(x,y),x)\\
&\approx f(f(f_{i-1}(x,y),x),y)\\
&\approx f_2(f(f_{i-2}(x,y),x),y)\\
&\cdots\\
&\approx f_i(f(x,x),y)\\
&\approx f_i(x,y).
\end{align*}
In the second case, we can apply \eqref{eq-right-orbit} to see
\[
f(f_i(x,y),f_j(y,x)) \approx f(f_i(x,y),y) \approx f_{i+1}(x,y).
\]
Letting $p = |\Clo_2^{\pi_1}(\bA)|$, we see that $f^- \approx f_{p-1}$ and $f_p(x,y) \approx x$, so together with \eqref{eq-right-orbit} and \eqref{eq-commutative} we see that $(A,f)$ is a $p$-cyclic groupoid.

To finish the proof, we just need to check that $p$ is prime. If $d$ is a nontrivial divisor of $p$ then the $p$-cyclic groupoid identities for $f$ imply that $f_d$ is nontrivial and $(A,f_d)$ forms a $(p/d)$-cyclic groupoid - but in this case, $f \not\in \Clo(f_d)$, contradicting the assumption that $\bA$ is a minimal clone.
\end{proof}

\begin{lem}\label{find-fixed-point} If $\bA$ satisfies $(*)$ and doesn't have any nontrivial term $f \in \Clo_2^{\pi_1}(\bA)$ which satisfies the identity
\[
f(f(x,y),y) \approx f(x,y),
\]
then $\bA$ is a $p$-cyclic groupoid.
\end{lem}
\begin{proof} We will make use of the composition operator $*_1$, which we defined in equation \eqref{star-1} by%We define a second composition law $*_1$ on $\Clo_2^{\pi_1}(\bA)$ by
\[
(f*_1g)(x,y) = f(g(x,y),y).
\]
The assumption implies that $G = (\Clo_2^{\pi_1}(\bA), *_1)$ is a group, by Proposition \ref{unary-iteration}.% We'll first use this to show that the free algebra $\cF_{\bA}(x,y)$ has a large automorphism group: for any $f \in \Clo_2^{\pi_1}(\bA)$, the fact that $G$ forms a group implies that $x$ is in the subalgebra generated by $f(x,y)$ and $y$, so
%\[
%\Sg_{\cF_\bA(x,y)}(f(x,y),y) = \cF_\bA(x,y).
%\]
%Thus the homomorphism $\cF_\bA(x,y) \rightarrow \cF_\bA(x,y)$ which sends $x$ to $f(x,y)$ and $y$ to $y$ is surjective, so it must be an automorphism (by the finiteness of $\cF_\bA(x,y)$). Similarly, the map which sends $x$ to $x$ and $y$ to $f(y,x)$ is also an automorphism. Since this holds for any $f \in \Clo_2^{\pi_1}(\bA)$, we can compose automorphisms of both types to show that for any $f,g \in \Clo_2^{\pi_1}(\bA)$ the map sending $x$ to $f(x,y)$ and $y$ to $g(y,x)$ is an automorphism.

%Now let $\bB \in \Var(\bA)$ be a nontrivial algebra with $|\bB|$ as small as possible. Since $\Clo(\bA)$ is generated by a binary operation, $\bB$ must be generated by two elements, so there is a surjective homomorphism $s: \cF_\bA(x,y) \twoheadrightarrow \bB$. For any $f, g \in \Clo_2^{\pi_1}(\bA)$, we have $\Sg_{\cF_\bA(x,y)}(f(x,y),g(y,x)) = \cF_\bA(x,y)$, so we must have $s(f(x,y)) \ne s(g(y,x))$. Thus, $\bB$ must have a surjective homomorphism to a two element projection algebra, with one congruence class given by $s(\Clo_2^{\pi_1}(\bA))$. Since this congruence class is a strictly smaller algebra than $\bB$, it must be a projection algebra, so for any $f,g,h \in \Clo_2^{\pi_1}(\bA)$ we have
%\[
%f^\bB(g^\bB(x,y),h^\bB(x,y)) = g^\bB(x,y)
%\]
%for all $x,y \in \bB$, and since $\bB$ is nontrivial this shows that the identity
%\[
%f(g(x,y),h(x,y)) \approx g(x,y)
%\]
%holds in $\bA$ as well. In particular, for any $f,g \in \Clo_2^{\pi_1}(\bA)$ we have
%\[
%f(x,g(x,y)) \approx x.
%\]

Now we define a right action $\cdot$ of $G$ on $\Clo_2^{\pi_1}(\bA)$ by
\[
(f \cdot g)(x,y) = f(x,g(y,x)).
\]
First we verify that this is a right action:
\begin{align*}
(f\cdot (g*_1h))(x,y) &= f(x,(g*_1h)(y,x))\\
&= f(x,g(h(y,x),x))\\
&= (f \cdot g)(x,h(y,x))\\
&= ((f \cdot g)\cdot h)(x,y).
\end{align*}
Note that $\{\pi_1\}$ is an orbit of this action, so every nontrivial orbit has size at most $|G| - 1$. Thus, by the orbit-stabilizer theorem there must be some nontrivial $g \in G$ and $f \in \Clo_2^{\pi_1}(\bA)$ such that $f \cdot g = f$, i.e.
\[
f(x,g(y,x)) \approx f(x,y).
\]
Now we can apply Lemma \ref{fixed-point} to finish the argument.
\end{proof}

%\begin{lem} If $\bA$ is a minimal clone with no nontrivial partial semilattice operations, and if $u$ is a term of $\bA$ which satisfies the identity
%\[
%u(u(x,y),x) \approx u(x,y),
%\]
%then $u$ also satisfies the identity
%\[
%u(x,u(x,y)) \approx x,
%\]
%so $\{x,u(x,y)\}$ is a subalgebra of $\cF_\bA(x,y)$ which is term-equivalent to a projection algebra.
%\end{lem}
%\begin{proof} We may assume that $u$ is nontrivial. Since $u^\infty$ is a partial semilattice operation, we must have $u^\infty(x,y) \approx x$. Defining $u^n$ as in the definition of $u^\infty$, we see that there must be some $n \ge 1$ such that $u^n$ is nontrivial and $u^{n+1}(x,y) \approx x$. Then we have
%\[
%u^n(x,u(x,y)) = u^{n+1}(x,y) \approx x
%\]
%and
%\[
%u^n(u(x,y),x) \approx u(x,y),
%\]
%so $\{x,u(x,y)\}$ is closed under the operation $u^n$. Since $u^n$ is nontrivial and $\bA$ is a minimal clone, we have $\Clo(\bA) = \Clo(u^n)$, and we are done.
%\end{proof}

\begin{lem}\label{no-fixed-point} There is no $\bA$ which satisfies $(*)$ and has a nontrivial term $f \in \Clo_2^{\pi_1}(\bA)$ which satisfies the identity
\[
f(f(x,y),y) \approx f(x,y).
\]
\end{lem}
\begin{proof} Suppose for contradiction that such $\bA$ and $f$ existed. Note that such an $\bA$ can't be a $p$-cyclic groupoid, by the explicit description of the free $p$-cyclic groupoid on two generators.

Define a term $t \in \Clo_2^{\pi_1}(\bA)$ by
\[
t(x,y) = f(x,f(y,x)).
\]
%and
%\[
%u(x,y) = f(f(x,y),x).
%\]
%Since $u$ satisfies the identity $u(u(x,y),x) \approx u(x,y)$, by the previous lemma we have
%\[
%f(x,u(x,y)) \approx x.
%\]
Now we have
%\begin{align*}
%t(x,f(x,y)) &= f(x,f(f(x,y),x))\\
%&= f(x,u(x,y))\\
%&\approx x
%\end{align*}
%and
\begin{align*}
t(x,f(y,x)) &\approx f(x,f(f(y,x),x))\\
&\approx f(x,f(y,x))\\
&\approx t(x,y),
\end{align*}
so by Lemma \ref{fixed-point} $t$ must be trivial since $\bA$ is not a $p$-cyclic groupoid. Thus we have
\[
f(x,f(y,x)) \approx x,
\]
so we have
\[
f(f(x,z),f(y,z)) = f(x,z)
\]
whenever two of $x,y,z$ are equal, hence this equality holds identically by Proposition \ref{prop-check-two}.

Also, by Theorem \ref{pi1} (which applies by Proposition \ref{star-not-taylor-semi}) we have
\[
f(x,f(x,y)) \approx x,
\]
but these identities imply that
\begin{align*}
f(x,y) &\approx f(f(x,f(x,y)),f(y,f(x,y)))\\
&\approx f(x,f(x,y))\\
&\approx x,
\end{align*}
contradicting the assumption that $f$ was nontrivial.
%
%Now define terms $f^n$ as in the definition of $f^\infty$. We will show by induction that
%\[
%f(x,f^n(y,x)) \approx x
%\]
%for all $n \ge 1$. The base case $n = 1$ has already been proved, and for the inductive step we have
%\begin{align*}
%f(x,f^{n+1}(y,x)) &= f(x, f(y,f^n(y,x)))\\
%&\approx f(f(x,f^n(y,x)), f(y,f^n(y,x)))\\
%&\approx f(x,f^n(y,x))\\
%&\approx x.
%\end{align*}
%Thus we have
%\[
%f(x,f^\infty(y,x)) \approx x,
%\]
%so
%\[
%f^\infty(x,f^\infty(y,z)) = f^\infty(x,z)
%\]
%whenever two of $x,y,z$ are equal, hence always.
%
%Note that $f^\infty$ must be nontrivial, since otherwise we would have
%\[
%f(x,y) \approx f(x,f^\infty(y,x)) \approx x,
%\]
%contradicting the fact that $f$ is nontrivial. We will show that there is a nontrivial congruence $\sim$ on $\cF_\bA(x,y)$ such that $f^\infty$ acts as the second projection on the quotient $\cF_\bA(x,y)/\!\!\sim$ (contradicting the assumption that $\bA$ is not a rectangular band). We define $\sim$ by
%\[
%f^\infty(u,w) \sim f^\infty(v,w)
%\]
%for all $u,v,w \in \cF_\bA(x,y)$. Note that for any $u,v \in \cF_\bA(x,y)$ we either have $u \sim x$ or $u \sim y$, and the $\sim$-class of $f^\infty(u,v)$ only depends on the $\sim$-class of $v$. To finish, suppose for a contradiction that there were $u,v \in \cF_\bA(x,y)$ with
%\[
%f^\infty(u,x) = f^\infty(v,y).
%\]
%Then we would have
%\begin{align*}
%x &= f^\infty(x,f^\infty(u,x))\\
%&= f^\infty(x,f^\infty(v,y))\\
%&= f^\infty(x,y)
%\end{align*}
%in the free algebra $\cF_\bA(x,y)$, contradicting the fact that $f^\infty$ is nontrivial.
\end{proof}

%This finishes up the proof of Theorem \ref{p-cyclic}.
%Now we can put the pieces together to prove Theorem \ref{p-cyclic}.

\begin{proof}[Proof of Theorem \ref{p-cyclic}] Suppose that $\Clo(\bA)$ is a binary minimal clone which is not a rectangular band and which has no nontrivial $f \in \Clo_2(\bA)$ which satisfies the identity $f(f(x,y),f(y,x)) \approx f(x,y)$, that is, suppose that $\bA$ satisfies $(*)$. By Lemma \ref{no-fixed-point}, there is also no nontrivial $f \in \Clo_2^{\pi_1}(\bA)$ satisfying $f(f(x,y),y) \approx f(x,y)$. Then by Lemma \ref{find-fixed-point} $\bA$ is a $p$-cyclic groupoid for some prime $p$.
\end{proof}

\begin{cor}\label{groupy-2} If $\bA$ is a clone-minimal binary algebra which is not Taylor, not a rectangular band, and has no nontrivial term $g \in \Clo_2^{\pi_1}(\bA)$ satisfying the identity
\[
g(g(x,y),y) \approx g(x,y),
\]
then $\bA$ is a $p$-cyclic groupoid.
\end{cor}
\begin{proof} As in Lemma \ref{find-fixed-point}, the assumption implies that $G = (\Clo_2^{\pi_1}(\bA), *_1)$ is a group.

We will show that the free algebra $\cF_{\cV(\bA)}(x,y)$ has a large automorphism group: for any $f \in \Clo_2^{\pi_1}(\bA)$, the fact that $f$ has an inverse in $G$ implies that $x$ is in the subalgebra generated by $f(x,y)$ and $y$, so
\[
\Sg_{\cF_{\cV(\bA)}(x,y)}\{f(x,y),y\} = \cF_{\cV(\bA)}(x,y).
\]
Thus the homomorphism $\cF_{\cV(\bA)}(x,y) \rightarrow \cF_{\cV(\bA)}(x,y)$ which sends $x$ to $f(x,y)$ and $y$ to $y$ is surjective, so it must be an automorphism (by the finiteness of $\cF_{\cV(\bA)}(x,y)$). Symmetrically, the homomorphism which sends $x$ to $x$ and $y$ to $f(y,x)$ is also an automorphism.

We claim that for any $f,g \in \Clo_2^{\pi_1}(\bA)$ the homomorphism sending $(x, y)$ to $(f(x,y), g(y,x))$ is an automorphism. For this, note that since $x$ and $g(y,x)$ generate $\cF_{\cV(\bA)}(x,y)$ we have
\[
f(x,y) \in \Sg_{\cF_{\cV(\bA)}(x,y)}\{x, g(y,x)\}.
\]
Thus there is a binary term $h$ such that
\[
f(x,y) \approx h(x, g(y,x)),
\]
and any such $h$ must be in $\Clo_2^{\pi_1}(\bA)$ if $f,g \in \Clo_2^{\pi_1}(\bA)$. Composing the automorphism $\alpha$ which sends $x$ to $h(x,y)$ and $y$ to $y$ with the automorphism $\beta$ which sends $x$ to $x$ and $y$ to $g(y,x)$, we find
\[
\beta(\alpha(x)) \approx \beta(h(x,y)) \approx h(\beta(x), \beta(y)) \approx h(x, g(y,x)) \approx f(x,y)
\]
and
\[
\beta(\alpha(y)) \approx \beta(y) \approx g(y,x),
\]
so $\beta \circ \alpha$ sends $(x,y)$ to $(f(x,y), g(y,x))$.

In particular, for any $f \in \Clo_2(\bA)$ there is an automorphism $\gamma : \cF_{\cV(\bA)}(x,y) \rightarrow \cF_{\cV(\bA)}(x,y)$ such that
\[
\gamma(x) = f(x,y), \;\;\; \gamma(y) = f(y,x).
\]
Then we see that $f$ cannot satisfy the identity $f(f(x,y),f(y,x)) \approx f(x,y)$ unless $f$ also satisfies the identity
\begin{align*}
f(x,y) &\approx f\big(\gamma^{-1}(f(x,y)), \gamma^{-1}(f(y,x))\big)\\
&\approx \gamma^{-1}\big(f(f(x,y), f(y,x))\big)\\
&\approx \gamma^{-1}(f(x,y))\\
&\approx x,
\end{align*}
in which case $f$ is trivial. Now we can apply Theorem \ref{p-cyclic} to see that $\bA$ is a $p$-cyclic groupoid.
\end{proof}

A cute application of what we have proved so far is the following connectivity result about the undirected graph $\cG_\bA$ of two-element subalgebras of $\bA$.
\begin{comment}
 To state it, we will associate a graph $\cG_\bA$ to every algebra, which generalizes the graphs we associated to melds.

\begin{defn} For any algebra $\bA$, we let $\cG_\bA$ be the graph of two element subalgebras of $\bA$.
\end{defn}
\end{comment}

\begin{thm}\label{connect} If $\Clo(\bA)$ is an idempotent minimal clone which is neither an affine algebra over $\bF_p$ nor a $p$-cyclic groupoid, then the graph $\cG_\bA$ is connected.
\end{thm}
\begin{proof} First suppose that $\bA$ has no nontrivial binary term operations. In this case every two-element subset $\{a,b\}$ must be a projection subalgebra of $\bA$, so $\cG_\bA$ is a clique. So from here on we assume that $\bA$ has a nontrivial binary term operation.

Next suppose that $\bA = (A,\cdot)$ is a rectangular band. By Proposition \ref{rect-band-structure} we can write $\bA \cong \bA_1 \times \bA_2$, where $\bA_i$ are projection algebras such that $\cdot$ acts as $\pi_i$ on $\bA_i$, so graph $\cG_\bA$ is what is known as the ``Cartesian product'' of the cliques $\cG_{\bA_1}$ and $\cG_{\bA_2}$. More concretely, for any $a,b \in \bA$, we see that $\cdot$ acts as second projection on $\{a, a\cdot b\}$ and $\cdot$ acts as first projection on $\{a\cdot b, b\}$, so $a$ and $b$ have distance at most two in $\cG_\bA$.

From here on we assume that $\bA$ is not a rectangular band. By Theorem \ref{p-cyclic} and the assumption that $\bA$ is not a $p$-cyclic groupoid, we see that $\bA$ has a nontrivial term $f$ satisfying the identity
\[
f(f(x,y),f(y,x)) \approx f(x,y).
\]
Then $f$ acts as first projection on the set $\{f(x,y),f(y,x)\}$. Since $f$ generates $\Clo(\bA)$, the set $\{f(x,y),f(y,x)\}$ therefore forms a projection subalgebra of the free algebra $\cF_{\cV(\bA)}(x,y)$.

If $\bA$ has a nontrivial partial semilattice operation $s$, then the set $\{x,s(x,y)\}$ is a two-element semilattice subalgebra of $\cF_{\cV(\bA)}(x,y)$ by Proposition \ref{prop-partial-semilattice}.
Since $f \in \Clo(s)$, we see that in this case $f(x,y)$ is connected by a path in $\cG_{\cF_{\cV(\bA)}(x,y)}$ (consisting of two-element semilattice subalgebras) to at least one of $x,y$. Then since $f(x,y)$ is adjacent or equal to $f(y,x)$ we see that $x$ and $y$ must be connected to each other in $\cG_{\cF_{\cV(\bA)}(x,y)}$. Thus every pair of elements of $\cG_\bA$ are connected within the subalgebra $\Sg_\bA\{a,b\}$ which they generate.

If $\bA$ is Taylor, then we must have $f(x,y) \approx f(y,x)$, since otherwise $f$ acts as first projection on some two-element set $\{f(a,b), f(b,a)\}$. In fact, by Theorem \ref{thm-taylor-minimal} and the assumption that $\bA$ is not affine, in this case $\bA$ is a spiral. Then by the definition of a spiral (Definition \ref{defn-spiral}), there is some two element semilattice in $\Var(\bA)$, so by Lemma \ref{lem-semilattice-iteration} we see that $\bA$ has a nontrivial partial semilattice term $s$, so $\cG_\bA$ is connected by the previous paragraph.

Now suppose that $\bA$ has no nontrivial partial semilattice operations, and is not Taylor. By Corollary \ref{groupy-2}, $\bA$ has a nontrivial $g \in \Clo_2^{\pi_1}(\bA)$ satisfying
\[
g(g(x,y),y) \approx g(x,y).
\]
Define a term $u \in \Clo_2^{\pi_1}(\bA)$ by
\[
u(x,y) = g(g(x,y),x).
\]
Then $u$ is a nontrivial term (by Theorem \ref{pi1} and the fact that $g$ is nontrivial) which satisfies the identity
\[
u(u(x,y),x) \approx u(x,y)
\]
by Proposition \ref{t-to-u}. By Theorem \ref{pi1} we have
\[
u(x,u(x,y)) \approx x,
\]
so $\{x, u(x,y)\}$ is a two-element projection subalgebra of $\cF_{\cV(\bA)}(x,y)$. Since $f \in \Clo(u)$, we can finish the proof with the same argument we used in the partial semilattice case.
\end{proof}

%Next we introduce a class of binary minimal clones which has a nice structure theory (which I think is new).

% possible names: agma algebra? barb? meld? -> meld!

%\begin{defn} We'll call an algebra $\bA = (A,f)$ a \emph{meld} if $f$ is idempotent and satisfies the identity
%\begin{align}
%f(f(x,y),f(z,x)) \approx f(x,y).\label{G}\tag{$\cG$}
%\end{align}
%\end{defn}

\begin{thm}\label{graphy} Suppose $\bA$ has no ternary semiprojections and has a nontrivial idempotent term $f$ which satisfies the identities
\begin{equation}
f(x,f(x,y)) \approx f(x,f(y,x)) \approx x.\label{eq-to-meld}
\end{equation}
Then $f$ satisfies the identity
\[
f(f(x,y),f(z,x)) \approx f(x,y),
\]
that is, the reduct $(A, f)$ is a meld.
%Conversely, every meld defines a minimal clone.
\end{thm}
\begin{proof} The identities \eqref{eq-to-meld} imply that
\begin{align}
f(f(x,y),y) &\approx f(f(x,y),f(y,f(x,y)))\nonumber\\
&\approx f(x,y)\nonumber%\label{eq-to-meld-1}
\end{align}
and
\begin{align}
f(f(x,y),x) &\approx f(f(x,y),f(x,f(x,y)))\nonumber\\
&\approx f(x,y).\nonumber%\label{eq-to-meld-2}
\end{align}
Thus, we have
\[
f(f(y,x),z) = f(y,x)
\]
whenever $z$ is equal to one of $x$ or $y$, so by \eqref{eq-to-meld} we have
\begin{equation}
f(x,f(f(y,x),z)) = x\label{eq-to-meld-3}
\end{equation}
whenever two of $x,y,z$ are equal. Since $\bA$ has no ternary semiprojections, \eqref{eq-to-meld-3} must hold identically by Proposition \ref{prop-no-semiprojection}. By Proposition \ref{prop-meld-absorption}, the identity \eqref{eq-to-meld-3} implies that $f$ is a meld.
\end{proof}

\begin{prop}\label{prop-meld-in-var} If $\Clo(\bA)$ is a minimal clone and some nontrivial $\bB \in \Var(\bA)$ is a meld, then so is $\bA$.
\end{prop}
\begin{proof} By Proposition \ref{absorb-prop}, this follows from Proposition \ref{prop-meld-absorption}, which shows that the variety of melds can be defined by absorption identities.
\end{proof}

\begin{lem}\label{graphy-gen} If $\Clo(\bA)$ is a binary minimal clone, not a rectangular band, such that there are nontrivial terms $f,g \in \Clo_2^{\pi_1}(\bA)$ satisfying the identity
\begin{equation}
f(x,g(y,x)) \approx x,\label{eq-graphy}
\end{equation}
then $\bA$ is a meld.
\end{lem}
\begin{proof} Suppose for the sake of contradiction that $\bA$ is not a meld. By Proposition \ref{minimal-nontrivial}, there is some nontrivial subquotient $\bB \in HS(\bA)$ such that every proper subalgebra or quotient of $\bB$ is a projection algebra. Note that $f,g$ act as first projection on every projection algebra in $\Var(\bA)$ by Corollary \ref{not-rect-pi1}, and that $f, g$ are nontrivial on $\bB$ by Proposition \ref{absorb-prop}. By Proposition \ref{prop-meld-in-var}, if $\bA$ is not a meld then $\bB$ is not a meld either.% We will first argue similarly to the proof of Theorem \ref{pi1} to show that the image of $\Clo_2^{\pi_1}(\bA)$ in the free algebra $\cF_{\cV(\bB)}(x,y)$ is a projection algebra.

% By Theorem \ref{connect} we see that the graph $\cG_\bA$ of two-element subalgebras of $\bA$ is connected

%Suppose for the sake of contradiction that $\bA$ is not a meld. By Proposition \ref{minimal-nontrivial} and Proposition \ref{prop-meld-in-var}, we may assume without loss of generality that every proper subalgebra or quotient of $\bA$ is a projection algebra (otherwise, we replace $\bA$ by a subquotient $\bB$ of $\bA$). We will first argue similarly to the proof of Theorem \ref{pi1} to show that $\Clo_2^{\pi_1}(\bA)$, considered as a subalgebra of $\cF_{\cV(\bA)}(x,y)$, is a projection algebra.

Note that $\bB$ can't have any two-element semilattice subalgebra $\{a,b\}$ with absorbing element $b$, since every nontrivial binary operation must act as the semilattice operation on $\{a,b\}$ if $\Clo(\bA)$ is a minimal clone, which would imply that
\[
a = f(a,g(b,a)) = f(a,b) = b.
\]
In particular, $\bB$ has no nontrivial partial semilattice terms by Proposition \ref{prop-partial-semilattice}. Since $f \in \Clo_2^{\pi_1}(\bA)$ we have $\pi_2 \not\in \Clo_2^{\pi_1}(f)$, and by Proposition \ref{absorb-prop} we see that there is no $t \in \Clo_2^{\pi_1}(f)$ which restricts to $\pi_2$ on $\bB$. Thus we can apply Corollary \ref{cor-pi1} to see that $\Clo_2^{\pi_1}(\bA) = \Clo_2^{\pi_1}(f)$ maps to a projection subalgebra of $\cF_{\cV(\bB)}(x,y)$.

Next we prove that for any $n \ge 0$ and any $z_1, ..., z_n$, we have
\begin{equation}
f(x, g(\cdots g(g(y,x),z_1)\cdots,z_n)) \approx x.\label{eq-semi-ind}
\end{equation}
We prove this by induction on $n$. It's enough to prove that this absorption identity holds on $\bB$ by Proposition \ref{absorb-prop}. Since $\bA$ has no semiprojections of arity at least $3$, $\bB$ also has no semiprojections of arity at least $3$ by Proposition \ref{absorb-prop}, so we can apply Proposition \ref{prop-no-semiprojection} to see that we just have to check \eqref{eq-semi-ind} when at most two distinct values occur among $x,y,z_1, ..., z_n \in \bB$. If $x = y$, then \eqref{eq-semi-ind} follows from the fact that $\Clo_2^{\pi_1}(\bA)$ maps to a projection subalgebra of $\cF_{\cV(\bB)}(x,y)$. If $z_n = x$, then \eqref{eq-semi-ind} follows from \eqref{eq-graphy} with $y$ replaced by $g(\cdots g(g(y,x),z_1)\cdots,z_{n-1})$. If $z_n = y$, then we have
\[
g(g(\cdots g(g(y,x),z_1)\cdots,z_{n-1}),y) = g(\cdots g(g(y,x),z_1)\cdots,z_{n-1})
\]
by the fact that $\Clo_2^{\pi_1}(\bA)$ maps to a projection subalgebra of $\cF_{\cV(\bB)}(x,y)$, so
\[
f(x, g(g(\cdots g(g(y,x),z_1)\cdots,z_{n-1}),y)) = f(x, g(\cdots g(g(y,x),z_1)\cdots,z_{n-1})) = x
\]
by the inductive hypothesis. This completes the proof of \eqref{eq-semi-ind}.

Since $g$ generates $\Clo(\bA)$, we see that for any $h \in \Clo_2^{\pi_1}(\bA)$ we have $h \in \Clo_2^{\pi_1}(g)$, so we can apply \eqref{eq-semi-ind} to see that
\begin{equation}
h \in \Clo_2^{\pi_1}(\bA) \;\;\; \implies \;\;\; f(x, h(g(y,x),x)) \approx x.\label{eq-semi-cons}
\end{equation}

Since $\bB$ is not a meld, $\bB$ must not satisfy the conditions of Theorem \ref{graphy}, so there are $a,b \in \bB$ with
\[
g(a,g(b,a)) \ne a.
\]
Then $\{a,g(b,a)\}$ is not a projection subalgebra of $\bB$ by Corollary \ref{not-rect-pi1}, so $a$ and $g(b,a)$ generate $\bB$. In particular, for any $c \in \bB$ there is some $h \in \Clo_2^{\pi_1}(\bA)$ such that either $h(a,g(b,a)) = c$ or $h(g(b,a),a) = c$. By \eqref{eq-semi-cons} and the fact that $\Clo_2^{\pi_1}(\bA)$ maps to a projection subalgebra of $\cF_{\cV(\bB)}(x,y)$, we have
\[
f(a,h(g(b,a),a)) = a = f(a,h(a,g(b,a))),
\]
so for all $c \in \bB$ we have $f(a,c) = a$, that is, $O(a) = \{a\}$. Since $g \in \Clo_2^{\pi_1}(f)$, we can apply Proposition \ref{prop-right-pi1} to see that for all $c \in \bB$ we have $g(a,c) = a$, but taking $c = g(b,a)$ contradicts $g(a,g(b,a)) \ne a$.
\end{proof}

\begin{lem}\label{graphy-gen-2} If $\Clo(\bA)$ is a binary minimal clone, not a rectangular band, such that there are $f,g,h \in \Clo_2^{\pi_1}(\bA)$ with $f,g$ nontrivial satisfying the identities
\begin{align}
f(h(x,y),g(y,x)) &\approx x,\label{eq-graphy-gen-2-0}\\
g(g(x,y),y) &\approx g(x,y),\label{eq-graphy-gen-2-1}
\end{align}
then $\bA$ is a meld.
\end{lem}
\begin{proof} Suppose for the sake of contradiction that $\bA$ is not a meld. By Proposition \ref{minimal-nontrivial}, there is some nontrivial subquotient $\bB \in HS(\bA)$ such that every proper subalgebra or quotient of $\bB$ is a projection algebra. As in the proof of Lemma \ref{graphy-gen}, $\bB$ is not a meld by Proposition \ref{prop-meld-in-var}, and $f, g$ are nontrivial on $\bB$ by Proposition \ref{absorb-prop} and act as first projection on any proper subalgebra or quotient of $\bB$ by Corollary \ref{not-rect-pi1}.

To see that $\bB$ has no nontrivial partial semilattice term operations, suppose for a contradiction that $\{a,b\}$ was a two-element semilattice subalgebra of $\bB$ with absorbing element $b$, and note that in this case $f,g$ would need to act as the semilattice operation on $\{a,b\}$, so
\[
a = f(h(a,b),g(b,a)) = f(h(a,b),b) \in f(\{a,b\}, b) = \{b\},
\]
which is impossible. Since $f \in \Clo_2^{\pi_1}(\bA)$ we have $\pi_2 \not\in \Clo_2^{\pi_1}(f)$, and by Proposition \ref{absorb-prop} we see that there is no $t \in \Clo_2^{\pi_1}(f)$ which restricts to $\pi_2$ on $\bB$. Thus we can apply Corollary \ref{cor-pi1} to see that $\Clo_2^{\pi_1}(\bA) = \Clo_2^{\pi_1}(f)$ maps to a projection subalgebra of $\cF_{\cV(\bB)}(x,y)$. In particular, the identity
\[
g(g(x,y),x) \approx g(x,y)
\]
holds on $\bB$, and together with \eqref{eq-graphy-gen-2-1} we see that
\begin{equation}
z_1, ..., z_n \in \{x,y\} \;\;\; \implies \;\;\; g(\cdots g(g(x,y),z_1)\cdots,z_n) \approx g(x,y)\label{eq-graphy-gen-2-orb}
\end{equation}
holds on $\bB$.

Now we will show that
\begin{equation}
f(h(x,y), g(\cdots g(g(y,x),z_1)\cdots,z_n)) \approx x.\label{eq-graphy-gen-2-2}
\end{equation}
Since this is an absorption identity, we only need to check \eqref{eq-graphy-gen-2-2} on $\bB$ by Proposition \ref{absorb-prop}, and since $\bB$ has no nontrivial semiproejctions of arity at least $3$ we only need to check \eqref{eq-graphy-gen-2-2} when at most two distinct values occur among $x,y,z_1, ..., z_n \in \bB$. If $x = y$, then \eqref{eq-graphy-gen-2-2} follows from Theorem \ref{pi1}. If $x \ne y$, then we must have $z_1, ..., z_n \in \{x,y\}$, so we can apply \eqref{eq-graphy-gen-2-orb} and \eqref{eq-graphy-gen-2-0} to see that
\[
f(h(x,y), g(\cdots g(g(g(y,x),z_1),z_2)\cdots,z_n)) \approx f(h(x,y), g(y,x)) \approx x
\]
on $\bB$. This completes the proof of \eqref{eq-graphy-gen-2-2}.

Since $g$ generates $\Clo(\bA)$, we can apply Proposition \ref{prop-right-pi1} to see that \eqref{eq-graphy-gen-2-2} implies
\begin{equation}
f(h(x,y), O(g(y,x))) = \{x\}\label{eq-graphy-gen-2-orb-1}
\end{equation}
in the free algebra $\cF_{\cV(\bA)}(x,y)$. Additionally, since $\Clo_2^{\pi_1}(\bA)$ maps to a projection subalgebra of $\cF_{\cV(\bB)}(x,y)$, we have
\begin{equation}
f(h(x,y), O(x)) = \{h(x,y)\}\label{eq-graphy-gen-2-orb-2}
\end{equation}
in the free algebra $\cF_{\cV(\bB)}(x,y)$.

\begin{comment}
 If $z_1 = x$ then by \eqref{eq-graphy-gen-2-1} we have
\[
f(h(x,y), g(\cdots g(g(g(y,x),z_1),z_2)\cdots,z_n)) \approx f(h(x,y), g(\cdots g(g(y,x),z_2)\cdots,z_n)),
\]
so this case of \eqref{eq-graphy-gen-2-2} follows from the inductive hypothesis.
\end{comment}
% this time using the identity $g(g(y,x),x) \approx g(g(y,x),y) \approx g(y,x)$.

Since $\bB$ is not a meld, $\bB$ must not satisfy the conditions of Theorem \ref{graphy}, so there are $a,b \in \bB$ with
\[
g(a,g(b,a)) \ne a.
\]
Since $g$ acts as first projection on every proper subalgebra of $\bB$, we see that $a$ and $g(b,a)$ generate $\bB$, so
\[
\bB = O(a) \cup O(g(b,a)).
\]
Then by \eqref{eq-graphy-gen-2-orb-1} and \eqref{eq-graphy-gen-2-orb-2} we see that for all $c \in \bB$ we have
\begin{equation}
f(h(a,b),c) = \begin{cases} h(a,b) & c \in O(a),\\ a & c \in O(g(b,a)).\end{cases}\label{eq-graphy-gen-2-cases}
%\in \{a, h(a,b)\}.
\end{equation}
If $h(a,b) = a$ then as in Lemma \ref{graphy-gen} we have $O(a) = \{a\}$, so by Proposition \ref{prop-right-pi1} we have $a \ne g(a,g(b,a)) \in O(a) = \{a\}$, which is a contradiction.% in Lemma \ref{graphy-gen}.

%Otherwise, the argument of Theorem \ref{pi1} shows that $\bA$ has a congruence $\sim$ such that $\bA/\!\!\sim$ is a two element projection algebra, and that for any $c \sim b$ we have
Otherwise, if $h(a,b) \ne a$, then \eqref{eq-graphy-gen-2-cases} implies that $O(a)$ and $O(g(b,a))$ must be disjoint, so $\bB$ has a congruence $\sim$ with congruence classes $O(a)$ and $O(g(b,a))$ such that $\bB/\!\!\sim$ is a two element projection algebra. From $g(b,a) \in O(b)$, we see that $b \sim g(b,a)$, so we have
\[
c \sim b \;\;\; \implies \;\;\; f(h(a,b),c) = a \ne h(a,b),
\]
which implies
\begin{equation}
c \sim b \;\;\; \implies \;\;\; \Sg_\bB\{h(a,b),c\} = \bB\label{eq-graphy-gen-2-big}
\end{equation}
since $f$ acts as first projection on every proper subalgebra of $\bB$.

By \eqref{eq-graphy-gen-2-1} and the assumption that $g$ is nontrivial we see that $\bA$ can't be a $p$-cyclic groupoid, so by Theorem \ref{p-cyclic} $\bA$ has a nontrivial term $t$ satisfying
\begin{equation}
t(t(x,y),t(y,x)) \approx t(x,y).\label{eq-graphy-gen-2-inf}
\end{equation}
From $h \in \Clo(t)$ and $h(a,b) \ne a,b$, we see that there are $u,v \in \bB\setminus\{h(a,b)\}$ with $t(u,v) = h(a,b)$. Since each $\sim$-class of $\bB$ is a projection subalgebra, we have $u \not\sim v$, so since $t$ acts as some projection on $\bB/\!\!\sim$ and $t(u,v) = h(a,b) \sim a$, we must have
\[
t(v,u) \sim b.
\]
But then we have
\[
\bB = \Sg_\bB\{h(a,b), t(v,u)\} = \Sg_\bB\{t(u,v), t(v,u)\} = \{t(u,v), t(v,u)\} \ne \bB,
\]
where the first equation follows from $t(v,u) \sim b$ and \eqref{eq-graphy-gen-2-big}, and the last equation follows from the fact that $\{t(u,v), t(v,u)\}$ is closed under $t$ by \eqref{eq-graphy-gen-2-inf}. This contradiction completes the proof.
\end{proof}

\begin{thm}\label{all-but-dispersive} If $\bA$ is a binary clone-minimal algebra which is not Taylor, has no partial semilattice operations, is not a rectangular band or a $p$-cyclic groupoid, and is not a meld, then $\bA$ is term-equivalent to a dispersive algebra.
\end{thm}
\begin{proof} Since $\bA$ is not Taylor, $\Clo_2^{\pi_1}(\bA)$ contains some nontrivial operation $f$. First we will show that any $f \in \Clo_2^{\pi_1}(\bA)$ satisfies \eqref{D2}, that is, that
\[
f(x,f(\cdots f(f(x,y_1),y_2)\cdots,y_n)) \approx x.
\]
Since $\bA$ has no semiprojections of arity at least $3$, we can apply Proposition \ref{prop-no-semiprojection} to see that it's enough to check \eqref{D2} when at most two distinct values occur among $x,y_1,...,y_n$. Supposing that $x,y_1,...,y_n \in \{x,y\}$ and setting $g(x,y) = f(\cdots f(f(x,y_1),y_2)\cdots,y_n)$, we can apply Theorem \ref{pi1} (since $\bA$ is not a rectangular band and has no partial semilattice operation) to see that
\[
f(x,f(\cdots f(f(x,y_1),y_2)\cdots,y_n)) \approx f(x,g(x,y)) \approx x,
\]
which completes the proof of \eqref{D2} for $f$. Pick any nontrivial $f \in \Clo_2^{\pi_1}(\bA)$ and assume without loss of generality that $\bA = (A,f)$.

%By Theorem \ref{pi1} and the fact that $\bA$ has no semiprojections, any binary operation of $\bA$ satisfies the identity \eqref{D2}.

Next we show that the free algebra $\cF_{\cV(\bA)}(x,y)$ over $\bA$ has a surjective homomorphism to the four-element free algebra $\cF_\cD(x,y)$ from Definition \ref{defn-dispersive}. Since $f$ satisfies \eqref{D2}, we can apply Proposition \ref{dispersive-hom-check} to see that we just need to check three properties of $\cF_{\cV(\bA)}(x,y)$:
\begin{itemize}
\item[(a)] we need to check that the right orbits $O(x)$ and $O(y)$ are disjoint in $\cF_{\cV(\bA)}(x,y)$,

\item[(b)] we need to check that for any $h(x,y) \in O(x)$ and $g(y,x) \in O(y)$ we have $f(x,g(y,x)) \ne x$ and $f(y,h(x,y)) \ne y$, and

\item[(c)] we need to check that there is no directed edge from $O(x) \setminus \{x\}$ to $x$ or from $O(y) \setminus \{y\}$ to $y$ in the directed graph associated to $\cF_{\cV(\bA)}(x,y)$.
\end{itemize}
Property (a) follows from the fact that $\cF_{\cV(\bA)}(x,y)$ has a surjective homomorphism to a projection algebra on which $f$ acts as first projection, by the assumption $f \in \Clo_2^{\pi_1}(\bA)$. Since $\bA$ is neither a rectangular band nor a meld, Lemma \ref{graphy-gen} implies property (b). All that remains is to verify property (c), which states that for any nontrivial $g \in \Clo_2^{\pi_1}(\bA)$ and for any (possibly trivial) $u \in \Clo_2(\bA)$ we have
\begin{equation}
f(g(x,y), u(x,y)) \not\approx x.\label{eq-disp-c}
\end{equation}
By Theorem \ref{pi1}, the non-identity \eqref{eq-disp-c} holds if $u(x,y) \in \Clo_2^{\pi_1}(\bA)$, so we just need to verify it when $u(y,x) \in \Clo_2^{\pi_1}(\bA)$.

Now we handle the case $u = \pi_2$ in \eqref{eq-disp-c}. Suppose for a contradiction that there is a (possibly trivial) $g \in \Clo_2^{\pi_1}(\bA)$ such that
\begin{equation}
f(g(x,y),y) \approx x.\label{eq-disp-c-pi1}
\end{equation}
Since $\bA$ is not Taylor, not a rectangular band, and not a $p$-cyclic groupoid, we can apply Corollary \ref{groupy-2} to see that $\bA$ has a nontrivial term $t \in \Clo_2^{\pi_1}(\bA)$ satisfying
\[
t(t(x,y),y) \approx t(x,y).
\]
Substituting $t(y,x)$ for $y$ in \eqref{eq-disp-c-pi1}, we have
\[
f(g(x,t(y,x)),t(y,x)) \approx x,
\]
and taking $h(x,y) = g(x,t(y,x))$ we see that this contradicts Lemma \ref{graphy-gen-2} since $\bA$ is not a rectangular band or a meld.

To finish verifying \eqref{eq-disp-c}, we just need to handle the case where $u(y,x) \in \Clo_2^{\pi_1}(\bA)$ is nontrivial. So suppose for the sake of contradiction that there are nontrivial $g,h \in \Clo_2^{\pi_1}(\bA)$ such that
\begin{equation}
f(h(x,y),g(y,x)) \approx x.\label{eq-disp-c-var}
\end{equation}
Define $g_n(x,y)$ inductively by $g_0 = \pi_1, g_1 = g$, and
\[
g_{n+1}(x,y) = g(g_n(x,y),y),
\]
that is, $g_n = g^{*_1n}$. Then for any $n \ge 1$ we can substitute $g_{n-1}(y,x)$ for $y$ in \eqref{eq-disp-c-var} to get
\begin{equation}
f(h(x,g_{n-1}(y,x)),g_n(y,x)) \approx x.\label{eq-disp-c-var2}
\end{equation}
By Corollary \ref{binary-iteration}, there is some $N > 0$ such that $g_N$ satisfies the identity
\[
g_N(g_N(x,y),y) \approx g_N(x,y).
\]
If $g_N$ is nontrivial, then plugging $n = N$ into \eqref{eq-disp-c-var2} gives us a contradiction with Lemma \ref{graphy-gen-2} since $\bA$ is not a rectangular band or a meld. Thus $g_N$ is trivial, and since $g_N \in \Clo_2^{\pi_1}(\bA)$ and $\bA$ is not a rectangular band, Corollary \ref{not-rect-pi1} shows that we must have $g_N = \pi_1$. Then plugging $n=N$ into \eqref{eq-disp-c-var2} gives us the identity
\[
f(h(x,g_{N-1}(y,x)),y) \approx x,
\]
that is,
\[
f(g'(x,y),y) \approx x,
\]
where $g'(x,y) = h(x,g_{N-1}(y,x)) \in \Clo_2^{\pi_1}(\bA)$. But this is just another form of \eqref{eq-disp-c-pi1}, which we have already checked leads to a contradiction.
\begin{comment}
If there is any $n$ such that $g_n$ is trivial, then the binary term $g'(x,y) = h(x,g_{n-1}(y,x))$ must be trivial: otherwise we are in the case of the previous paragraph, since we have
\[
f(h(x,g_{n-1}(y,x)),g_n(y,x)) \approx f(g'(x,y),y).
\]
Since $h(x,g_{n-1}(y,x))$ is trivial if $g_n$ is trivial, we see that in this case
\begin{align*}
f(x,y) &\approx f(h(x,g_{n-1}(y,x)),g_n(y,x))\\
&\approx x,
\end{align*}
contradicting the assumption that $f$ is nontrivial. On the other hand, if there is no $n$ such that $g_n$ is trivial, then there is some $n > 0$ such that $g_n(g_n(x,y),y) \approx g_n(x,y)$, and we get a contradiction to Lemma \ref{graphy-gen-2}.

By the last two paragraphs, Lemma \ref{graphy-gen}, and Theorem \ref{pi1}, we see that there is a surjective homomorphism
\[
\cF_\bA(x,y) \twoheadrightarrow \cF_\cD(x,y),
\]
and this finishes the proof.
\end{comment}

Having verified properties (a), (b), (c), we can apply Proposition \ref{dispersive-hom-check} to see that there is a surjective homomorphism from $\cF_{\cV(\bA)}(x,y)$ to $\cF_{\cD}(x,y)$. To finish the proof, let $a,b \in \bA$ be any pair such that $\{a,b\}$ is not a two element subalgebra of $\bA$. Since $\bA$ is clone-minimal, for any binary term $g$ of $\bA$ such that $g(a,b) = a$ and $g(b,a) = b$, we have $g(x,y) \approx x$ (otherwise $g$ would generate a strictly smaller clone). By Proposition \ref{dispersive-hom}, the existence of a surjective homomorphism $\cF_{\cV(\bA)}(x,y) \twoheadrightarrow \cF_{\cD}(x,y)$ implies that the identity $g(x,y) \approx x$ must follow from \eqref{D2}. Since the identity $g(x,y) \approx x$ follows from \eqref{D2} for any binary term $g$ such that $g(a,b) = a$ and $g(b,a) = b$, we can apply Proposition \ref{dispersive-hom} again to see that there is a surjective homomorphism
\[
\Sg_{\bA^2}\Big\{\begin{bmatrix} a\\ b\end{bmatrix},\begin{bmatrix} b\\ a\end{bmatrix}\Big\} \twoheadrightarrow \cF_\cD(x,y).\qedhere
\]
\end{proof}

\begin{comment}
\begin{prop}[Proposition \ref{dispersive-hom}] If $\bA = (A,f)$ and $f$ is a binary term satisfying \eqref{D2}, then for $a,b \in \bA$, the following are equivalent:
\begin{itemize}
\item there exists a surjective homomorphism $\alpha : \Sg_{\bA^2}\{(a,b),(b,a)\} \twoheadrightarrow \cF_\cD(x,y)$,

\item for every binary term $g$ of $\bA$ such that $g(a,b) = a$ and $g(b,a) = b$, the identity $g(x,y) \approx x$ follows from \eqref{D2}.
\end{itemize}
\end{prop}

\begin{cor} If a binary minimal clone $\bA$ is a weakly dispersive algebra, then for any $a, b \in \bA$ which generate a nontrivial subalgebra of $\bA$ there is a surjective homomorphism
\[
\Sg_{\bA^2}\{(a,b),(b,a)\} \twoheadrightarrow \cF_\cD(x,y).
\]
Thus, a minimal clone is weakly dispersive iff it is dispersive.
\end{cor}
\end{comment}

We can now complete the proof of Theorem \ref{main-thm}.

\begin{proof}[Proof of Theorem \ref{main-thm}] By Theorem \ref{all-but-dispersive}, every binary clone-minimal $\bA = (A,f)$ which is not Taylor is term equivalent to either a rectangular band, a $p$-cyclic groupoid, a partial semilattice, a meld, or a dispersive algebra. By swapping the variables of $f$ if necessary, we may assume without loss of generality that $f \in \Clo_2^{\pi_1}(\bA)$. By the explicit descriptions of the free rectangular band, $p$-cyclic groupoid, or meld on two generators, we see that in each of these three cases $\bA = (A,f)$ is not just term equivalent to, but instead is actually equal to, one of a rectangular band, $p$-cyclic groupoid, or meld.

In the dispersive case, there is some nontrivial $g \in \Clo_2^{\pi_1}(\bA)$ which is dispersive. Swapping the variables of $f$ if necessary, we have $f \in \Clo_2^{\pi_1}(g)$ since $\Clo(\bA)$ is a minimal clone. Then by Proposition \ref{dispersive-reduct} we see that $f$ is dispersive as well. Thus $\bA$ is not just term-equivalent to a dispersive algebra, but is in fact dispersive itself.

That each of these five cases defines a nice property is verified in Theorem \ref{rect-band-nice} for rectangular bands, Theorem \ref{p-cyclic-nice} for $p$-cyclic groupoids, Theorem \ref{partial-semi-nice} for algebras with a partial semilattice operation, Theorem \ref{meld-nice} for melds, and Theorem \ref{dispersive-nice} for dispersive algebras.
\end{proof}

\section{Acknowledgments}

The author worked on this problem due to the encouragement of Dmitriy Zhuk. Thanks are also due to the anonymous reviewer who suffered through a poorly written first draft of this paper and provided many suggestions to improve the exposition. This material is based upon work supported by the NSF Mathematical Sciences Postdoctoral Research Fellowship under Grant No. (DMS-1705177).

\bibliographystyle{plain}
\bibliography{min-clone-arxiv}

\appendix

\section{Partial semilattices with few strongly connected components}\label{a-partial-semi}

In this appendix we continue the analysis of the labeled digraphs attached to clone-minimal partial semilattices that we began in Proposition \ref{semi-not-strongly-connected}.

\begin{prop}\label{semi-two-strongly-connected} Suppose that $\bA = (A,s)$ is a clone-minimal partial semilattice, and that $\bA$ is generated by two elements $a,b$. If the directed graph $D_\bA$ of two-element semilattice subalgebras of $\bA$ has just two strongly connected components, one containing $a$ and the other containing $b$, then we must have $A = \{a,b\}$.
\end{prop}
\begin{proof} Suppose for contradiction that $\bA$ is a counterexample of minimal cardinality. Let $U, V$ be the strongly connected components of $\bA$ containing $a,b$, respectively, so that the underlying set of $\bA$ is $U \sqcup V$. Since at least one of $U, V$ is a maximal strongly connected component of $\bA$, we may assume without loss of generality that $V$ is a maximal strongly connected component of $\bA$. As in Proposition \ref{semi-not-strongly-connected}, we consider the algebra
\[
\bS = \Sg_{\bA^2}\Big\{\begin{bmatrix} a\\ b \end{bmatrix}, \begin{bmatrix} b\\ a \end{bmatrix}\Big\}.
\]

If $V$ is not contained in a linked component of $\bS$, then if we let $\theta$ be the linking congruence of $\bS$ (considered as a congruence on $\bA$), we find that $|V/\theta| > 1$. In this case we find that $\bA/\theta$ is generated by the two elements $a/\theta$ and $b/\theta$, has at most two strongly connected components (each of which contains either $a/\theta$ or $b/\theta$), and at least one of those strongly connected components has at least two elements, which either contradicts Proposition \ref{semi-not-strongly-connected} or contradicts the assumption that $\bA$ is a minimal counterexample if $\theta$ is nontrivial. Thus we may assume that either the linking congruence $\theta$ of $\bS$ is trivial, or that $V$ is contained in a linked component of $\bS$.

If $\theta$ is trivial, then $\bS$ is the graph of an isomorphism. Thus $(a,b)$ is contained in a maximal strongly connected component of $\bS$, contradicting the last sentence of Proposition \ref{semilattice-structure}.% so we can find a nontrivial term $t \in \Clo(s)$ such that $t(a,b) = a$ and $t(b,a) = b$.

If $V$ is contained in a linked component of $\bS$, then we have two cases: either $U$ and $V$ are both maximal strongly connected components of $D_\bA$, or $U\preceq V$, i.e. there is some semilattice subalgebra from $U$ to $V$. In the first case, by Theorem \ref{strong-binary} we have $U\times V \subseteq \bS$, so once again $(a,b)$ is contained in a maximal strongly connected component of $\bS$ and we have a contradiction to Proposition \ref{semilattice-structure}. In the second case, we have $b \in O(a)$, so there is some $g \in \Clo_2^{\pi_1}(s)$ with
\[
g(a,b) = b
\]
by Proposition \ref{prop-right-pi1}. Then we have
\[
g\Big(\begin{bmatrix} a\\ b \end{bmatrix}, \begin{bmatrix} b\\ a \end{bmatrix}\Big) \in \begin{bmatrix} b\\ O(b) \end{bmatrix} \cap \bS,
\]
so $(V\times V) \cap \bS \ne \emptyset$. Applying Theorem \ref{strong-binary} we see that $V\times V \subseteq \bS$, so $(b,b) \in \bS$, which implies $a \rightarrow b$ by Proposition \ref{semilattice-structure}.
\end{proof}

\begin{cor}\label{semi-not-both-max} If $\bA = (A,s)$ is a clone-minimal partial semilattice, and if $\bA$ is generated by two elements $a,b$ with $|\bA| \ge 3$, then at least one of $a,b$ is not contained in a maximal strongly connected component of the digraph of semilattice subalgebras of $\bA$.
\end{cor}
\begin{proof} Suppose for a contradiction that $a$ and $b$ are both contained in maximal strongly connected components, that is, that $O(a)$ and $O(b)$ are strongly connected. Since $\bA = O(a) \cup O(b)$ by Proposition \ref{prop-right-gen}, if $O(a)$ and $O(b)$ are both strongly connected then $\bA$ has either one or two strongly connected components, each containing $a$ or $b$. If $O(a) = O(b)$ we get a contradiction to Proposition \ref{semi-not-strongly-connected}, and if $O(a) \ne O(b)$ then we get a contradiction to Proposition \ref{semi-two-strongly-connected}.
%Since $x \rightarrow s(x,y)$ for all $x,y$, every element of $\Sg_\bA\{a,b\}$ is reachable from either $a$ or $b$ via a directed sequence of semilattice edges. Thus if $a,b$ are both contained in maximal strongly connected components, then every element of $\bA$ is contained in either the strongly connected component containing $a$ or the strongly connected component containing $b$, which would contradict one of the previous two propositions (depending on whether these strongly connected components are equal to each other or not).
\end{proof}

There is still one way for a $2$-generated minimal partial semilattice clone to have just two strongly connected components which we have not ruled out: it could plausibly be the case that the generators $a,b$ are both contained in a minimal strongly connected component $U$, which has a semilattice edge to a maximal strongly connected component $V$ which contains neither $a$ nor $b$. This possibility is a bit tricky to rule out. We will need a few preliminary results.

\begin{prop}\label{prop-upwards-closed} If $\RR \le_{sd} \bA \times \bB$ is subdirect and $\bA = (A,s), \bB = (B,s)$ are binary algebras, then for any $a \in \bA$ the set $O(a) + \RR$ is a union of right orbits of $\bB$, that is, if $b \in O(a) + \RR$ and $b \rightarrow c$ in $D_\bB$, then $c \in O(a) + \RR$.
\end{prop}
\begin{proof} Suppose that $(a',b) \in \RR$ for some $a' \in O(a)$. If $b \rightarrow c$ is a directed edge of $D_\bB$, then we must have $c = s(b,d)$ for some $d \in \bB$. Since $\RR$ is subdirect, there is some $e \in \bA$ such that $(e,d) \in \RR$. Then we have
\[
s\Big(\begin{bmatrix} a'\\ b \end{bmatrix}, \begin{bmatrix} e\\ d \end{bmatrix}\Big) \in \begin{bmatrix} O(a)\\ c \end{bmatrix} \cap \RR,
\]
so $c \in O(a) + \RR$.
\end{proof}

\begin{prop}\label{semi-nice-terms} Suppose that $\bA = (A,s)$ is a clone-minimal partial semilattice. Then the binary term $f = s^{\infty_c} \in \Clo_2^{\pi_1}(s)$ is nontrivial, has $a \preceq f(a,b)$ in $D_\bA$ for all $a,b$, and for all $a,b$ either $f(a,b) = f(b,a)$ or the set $\{f(a,b),f(b,a)\}$ forms a $2$-element projection subalgebra of $\bA$, with $s$ acting as first projection.
\end{prop}
\begin{proof} That $a \preceq f(a,b)$ follows from Proposition \ref{prop-right-pi1}. Since $f$ is defined nontrivially from $s$ and $s$ satisfies no nontrivial absorption identities by Corollary \ref{cor-absorption-semilattice}, $f$ must not be a projection, so we have $s \in \Clo(f)$. Then since $f$ acts as first projection on $\{f(a,b),f(b,a)\}$, $s$ must also act as a projection on $\{f(a,b),f(b,a)\}$, and the partial semilattice identities imply that $s$ can't act as second projection on any two element set.
\end{proof}

\begin{prop}\label{prop-proj-upwards-closed} If $\RR \le \bA \times \bB$ and $\bA = (A,s), \bB = (B,s)$ are binary algebras, then for any $a,b \in \bA$ such that $s$ acts as first projection on $\{a,b\}$ and $\{b\} + \RR = \bB$, the set $\{a\} + \RR$ is a union of right orbits of $\bB$, that is, if $c \in \{a\} + \RR$ and $c \rightarrow d$ in $D_\bB$, then $d \in \{a\} + \RR$.
\end{prop}
\begin{proof} Suppose that $(a,c) \in \RR$. If $c \rightarrow d$ is a directed edge of $D_\bB$, then we must have $d = s(c,e)$ for some $e \in \bB$. Since $\{b\} + \RR = \bB$, we have $(b,e) \in \RR$. Then we have
\[
s\Big(\begin{bmatrix} a\\ c \end{bmatrix}, \begin{bmatrix} b\\ e \end{bmatrix}\Big) = \begin{bmatrix} a\\ d \end{bmatrix} \in \RR,
\]
so $c \in \{a\} + \RR$.
\end{proof}

\begin{defn} We say that a subalgebra $\bB \le \bA$ \emph{absorbs} an element $a \in \bA$ with respect to a binary operation $f$ if
\[
f(a,\bB), f(\bB,a) \subseteq \bB.
\]
\end{defn}

\begin{prop}\label{nontriv-absorb} If $\bB \le \bA$ absorbs $a \in \bA$ with respect to a binary idempotent operation $f$, then for every nontrivial $g \in \Clo_2(f)$, $\bB$ absorbs $a$ with respect to $g$ as well.
\end{prop}
\begin{proof} We prove this by induction on the construction of $g$ as a term built out of $f$. Write $g(x,y) = f(u(x,y), v(x,y))$. If both $u$ and $v$ are trivial, then they must be different projections, since otherwise $g$ is trivial (by idempotence), and so $g$ is $f(x,y)$ or $f(y,x)$. Otherwise, if say $u$ is nontrivial, then by the inductive hypothesis applied to $u$ we have
\[
g(a,\bB) \subseteq f(u(a,\bB),v(a,\bB)) \subseteq f(\bB,\{a\}\cup\bB) \subseteq \bB
\]
since $\bB$ is closed under $f$ and absorbs $a$ with respect to $f$.
\end{proof}

\begin{prop}\label{absorb-gen} Suppose that $\bA = (A,s)$ is a clone-minimal partial semilattice which is generated by two elements $a,b$. If there is some proper subalgebra $\bB < \bA$ which absorbs $a$ and $b$ with respect to $s$ such that
\[
\Sg_{\bA^2}\Big\{\begin{bmatrix} a\\ b \end{bmatrix}, \begin{bmatrix} b\\ a \end{bmatrix}\Big\} \cap \begin{bmatrix} \bB\\ \bB \end{bmatrix} \ne \emptyset,
\]
then $\bA = \{a,b\} \cup \bB$ and $\bA$ has a two-element semilattice as a quotient.
\end{prop}
\begin{proof} Let $t \in \Clo_2(s)$ be a term such that
\[
t\Big(\begin{bmatrix} a\\ b \end{bmatrix}, \begin{bmatrix} b\\ a \end{bmatrix}\Big) \in \begin{bmatrix} \bB\\ \bB \end{bmatrix}.
\]
Since $\bB$ is a proper subalgebra of $\bA$ it can't contain both $a$ and $b$, so $t$ must be nontrivial. In particular, by Proposition \ref{nontriv-absorb} we see that $\bB$ absorbs $a$ and $b$ with respect to $t$. Since $\{a\} \cup \bB$ and $\{b\}\cup \bB$ are closed under $t$ and since $t(a,b), t(b,a) \in \bB$, we see that $\{a,b\} \cup \bB$ is closed under $t$. Since $t$ is nontrivial, we have $s \in \Clo(t)$, so in fact we have $\bA = \{a,b\} \cup \bB$.

If $a \not\in \bB$, then the equivalence relation $\theta$ on $\bA$ with equivalence classes $\{a\}$ and $\{b\} \cup \bB$ is a congruence of $t$ such that $(A,t)/\theta$ is a two-element semilattice with absorbing element $\{b\}\cup\bB$, and since $s \in \Clo_2(t)$ is nontrivial the same is true for $s$. If $b \not\in \bB$ then we instead use the equivalence relation with classes $\{a\}\cup\bB$ and $\{b\}$.
\end{proof}

\begin{defn} If $\RR \le \bA \times \bB$ and $\bS \le \bB \times \bC$, then we define their \emph{relational composition} $\RR\circ \bS \le \bA\times \bC$ by
\[
\RR \circ \bS = \{(a,c) \mid \exists b\in \bB\text{ s.t. }(a,b) \in \RR \wedge (b,c) \in \bS\}.
\]
If $\RR \le \bA \times \bA$, then we write $\RR^{\circ k}$ as shorthand for $\RR \circ \cdots \circ \RR$, with $k$ copies of $\RR$.
\end{defn}

\begin{prop}\label{sym-link} If $\bS \le_{sd} \bA \times \bA$ is subdirect, symmetric, and linked, and if $\bA$ is finite, then there is some $n \ge 1$ such that $\bS^{\circ n} = \bA^2$. Additionally, for any such $n$ we have $\bS^{\circ m} = \bA^2$ for every $m \ge n$ as well.
\end{prop}
\begin{proof} Since $\bS$ is symmetric, $\bS^{\circ 2}$ is the set of pairs $(a,b)$ such that there is some $c \in \bA$ with $(a,c) \in \bS$ and $(b,c) \in \bS$. As a consequence, the linking congruence of $\bS$, considered as a congruence $\theta$ on $\bA$, is given by
\[
\theta = \bigcup_{k \ge 0} \bS^{\circ 2k},
\]
where we interpret $\bS^{\circ 0}$ as the diagonal $\{(a,a) \mid a \in \bA\}$. Since $\bS$ is subdirect, for every $a \in \bA$ there is some $b \in \bA$ such that $(a,b), (b,a) \in \bS$, so the diagonal $\bS^{\circ 0}$ is contained in $\bS^{\circ 2}$, so for all $k$ we have
\[
\bS^{\circ 2k} = \bS^{\circ 2k} \circ \bS^{\circ 0} \subseteq \bS^{\circ 2k} \circ \bS^{\circ 2} = \bS^{\circ 2(k+1)}.
\]
Since $\bA^2$ is finite, there is some $k$ such that $\theta = \bS^{\circ 2k}$. Since $\theta = \bA^2$ by the assumption that $\bS$ is linked, we see that $\bA^2 = \bS^{\circ 2k}$.

For the last statement, suppose that $\bS^{\circ n} = \bA^2$ for some $n$. Then since $\bS$ is subdirect, we have
\[
\bS^{\circ (n+1)} \supseteq \bA^2 \circ \bS = \bA^2.\qedhere
\]
\end{proof}

\begin{thm} Suppose that $\bA = (A,s)$ is a clone-minimal partial semilattice, and that $\bA$ is generated by two elements $a,b$. Then it is not possible for $a$ and $b$ to be in the same strongly connected component of the digraph $D_\bA$ of semilattice subalgebras of $\bA$.
\end{thm}
\begin{proof} Suppose for contradiction that $\bA$ is a counterexample with $|\bA|$ minimal. If $\bA$ has a nontrivial quotient $\bA/\theta$, then $a/\theta, b/\theta$ are in the same strongly connected component of $\bA/\theta$ and generate $\bA/\theta$, so $\bA/\theta$ is a smaller counterexample, contradicting the choice of $\bA$. Thus we may assume that $\bA$ is simple.

Define $\bS \le_{sd} \bA \times \bA$ by
\[
\bS = \Sg_{\bA^2}\Big\{\begin{bmatrix} a\\ b \end{bmatrix}, \begin{bmatrix} b\\ a \end{bmatrix}\Big\}.
\]
Note that $\bS \ne \bA^2$ - otherwise, we have $(b,b) \in \bS$, so $a \rightarrow b$ by Proposition \ref{semilattice-structure}, which implies $\bA = \{a,b\}$, a contradiction.

First we show that $\bS$ is linked. Suppose for a contradiction that $\bS$ is not linked. Then since $\bA$ is simple, the linking congruence of $\bS$ on $\bA$ must be the trivial congruence, so $\bS$ is the graph of an isomorphism. In this case, we find that $(a,b)$ and $(b,a)$ are in the same strongly connected component of $\bS$ (since $\bS$ is isomorphic to $\bA$), contradicting the last sentence of Proposition \ref{semilattice-structure}.

By Proposition \ref{sym-link}, there is some $k \ge 1$ such that $\bS^{\circ k} \ne \bA^2$ and $\bS^{\circ (k+1)} = \bA^2$. Since $2k \ge k+1$, we then have $\bS^{\circ 2k} = \bA^2$ as well, so $(a,b) \in \bS^{\circ 2k}$, which implies that there is some $c$ such that $(a,c), (c,b) \in \bS^{\circ k}$. Since $\bS^{\circ k}$ is symmetric, we then have 
\[
a,b \in \{c\} + \bS^{\circ k},
\]
so in fact $\{c\} + \bS^{\circ k} = \bA$. Let $\bC$ be the set of all elements $c'$ such that $\{c'\} + \bS^{\circ k} = \bA$ - this is automatically a subalgebra of $\bA$, since
\[
\bC = (\{a\} + \bS^{\circ k}) \cap (\{b\} + \bS^{\circ k}).
\]
Note that by our choice of $k$ we have $\bC \ne \bA$. Our goal is to apply Proposition \ref{absorb-gen} to $\bC$.

First we check that $(\bC \times \bC) \cap \bS \ne \emptyset$. If $k = 1$, this follows from
\[
\bC \times \bC \subseteq \bC \times \bA \subseteq \bS^{\circ k} = \bS.
\]
If $k = 2$, then by the choice of $k$ we have $(a,a) \in \bS^{\circ 3}$, so there are $c,d$ such that $(a,c), (c,d), (d,a) \in \bS$, and since $(a,b) \in \bS$, we see that
\[
a \in \{d\} + \bS \subseteq \{c\} + \bS^{\circ 2}
\]
and
\[
b \in \{a\} + \bS \subseteq \{c\} + \bS^{\circ 2},
\]
so $c \in \bC$, and similarly $d \in \bC$, that is, $(c,d) \in (\bC\times\bC) \cap \bS$. If $k \ge 3$, then from $2(k-1) \ge k+1$ we see that $\bS^{\circ 2(k-1)} = \bA^2$, so $(a,b) \in \bS^{\circ 2(k-1)}$, so there is some $c$ such that $(a,c), (c,b) \subseteq \bS^{\circ (k-1)}$. For this $c$, we have
\[
a \in \{b\} + \bS \subseteq \{c\} + \bS^{\circ k}
\]
and similarly $b \in \{c\} + \bS^{\circ k}$, so $c \in \bC$. Picking any $d$ such that $(c,d) \in \bS$, we also have
\[
a \in \{c\} + \bS^{\circ (k-1)} \subseteq \{d\} + \bS^{\circ k}
\]
and similarly $b \in \{d\} + \bS^{\circ k}$, so we have $(c,d) \in (\bC \times \bC) \cap \bS$. Thus we have
\begin{equation}
(\bC \times \bC) \cap \bS \ne \emptyset\label{eq-semi-meet-S}
\end{equation}
regardless of the value of $k$.

Next we show that $\bC$ is a union of right orbits of $\bA$, i.e. that for any $c \in \bC$ and any $d$ with $c \rightarrow d$, we have $d \in \bC$. Since $\{c\} + \bS^{\circ k} = \bA$, we have
\[
s(\bA, \{d\} + \bS^{\circ k}) = s(\{c\} + \bS^{\circ k}, \{d\} + \bS^{\circ k}) \subseteq \{s(c,d)\} + \bS^{\circ k} = \{d\} + \bS^{\circ k},
\]
and similarly
\[
s(\{d\} + \bS^{\circ k}, \bA) \subseteq \{s(d,c)\} + \bS^{\circ k} = \{d\} + \bS^{\circ k}.
\]
Thus the subalgebra $\bB = \{d\} + \bS^{\circ k}$ absorbs every element of $\bA$ with respect to $s$. Since
\[
(d,d) \in \bA^2 = \bS^{\circ (2k+1)},
\]
we have $(\bB \times \bB) \cap \bS \ne \emptyset$. Then if $\bB \ne \bA$ then we can apply Proposition \ref{absorb-gen} to see that $\bA$ has a quotient which is a two-element semilattice, which is a contradiction to the assumption that $\bA$ is simple and that $a,b$ are in the same strongly connected component of $\bA$. Thus we must have $\bB = \bA$, so $d \in \bC$. Since this was true for any $c \rightarrow d$ with $c \in \bC$, we see that $\bC$ is a union of right orbits of $\bA$.

Now let $f = s^{\infty_c} \in \Clo_2^{\pi_1}(s)$ as in Proposition \ref{semi-nice-terms}. We will show that we have $f(a,\bC) \subseteq \bC$ and $f(b,\bC) \subseteq \bC$. Suppose that $c \in \bC$, and let $f(a,c) = d, f(c,a) = e$. Since $\bC$ is a union of right orbits of $\bA$, we have
\[
e = f(c,a) \in O(c) \subseteq \bC,
\]
so $\{e\} + \bS^{\circ k} = \bA$. Since $s$ acts as first projection on $\{d,e\}$ by Proposition \ref{semi-nice-terms} and since $\{e\} + \bS^{\circ k} = \bA$, we can apply Proposition \ref{prop-proj-upwards-closed} to see that $\{d\} + \bS^{\circ k}$ is a union of right orbits of $\bA$. Since $O(a) = O(b) = \bA$ by Proposition \ref{prop-right-gen} and the fact that $a,b$ are in the same strongly connected component of $D_\bA$, we just need to check that one of $a$ or $b$ is in $\{d\} + \bS^{\circ k}$ to confirm that $\{d\} + \bS^{\circ k} = \bA$. If $k$ is odd, then we have $(a,b) \in \bS^{\circ k}$, so
\[
b = f(b,b) \in f(\{a\} + \bS^{\circ k}, \{c\} + \bS^{\circ k}) \subseteq \{f(a,c)\} + \bS^{\circ k} = \{d\} + \bS^{\circ k},
\]
and similarly if $k$ is even we have $(a,a) \in \bS^{\circ k}$, so
\[
a = f(a,a) \in f(\{a\} + \bS^{\circ k}, \{c\} + \bS^{\circ k}) \subseteq \{f(a,c)\} + \bS^{\circ k} = \{d\} + \bS^{\circ k}.
\]
Either way, we see that $\{d\} + \bS^{\circ k} = \bA$, so
\[
d = f(a,c) \in \bC.
\]
The argument showing that $f(b,\bC) \subseteq \bC$ is the same, but with $a,b$ swapped throughout.
%Let $U$ be the set of $u \in \bA$ with $(d,u) \in \bS^{\circ k}$ - we just need to show that $U = \bA$. We will show that $U = \bA$ by showing that at least one of $a,b$ is in $U$ and that $u \in U$ and $u \rightarrow v$ imply $v \in U$.

Since we've already verified that $\bC$ is a union of right orbits of $\bA$ and that $f(a,\bC), f(b,\bC) \subseteq \bC$, we see that $\bC$ absorbs $a$ and $b$ with respect to $f$. Since $f$ is nontrivial by Proposition \ref{semi-nice-terms}, we have $s \in \Clo(f)$, so we can apply Proposition \ref{nontriv-absorb} to see that $\bC$ absorbs $a$ and $b$ with respect to $s$ as well. Then by \eqref{eq-semi-meet-S} we can apply Proposition \ref{absorb-gen} to $\bC$ to see that $\bA$ has a quotient which is a two-element semilattice, which is a contradiction to the assumption that $\bA$ is simple and that $a,b$ are in the same strongly connected component of $\bA$. This contradiction completes the proof.
\end{proof}

\begin{cor} Suppose that $\bA = (A,s)$ is a clone-minimal partial semilattice, and that $\bA$ is generated by two elements $a,b$. If $|\bA| \ge 3$, then the digraph $D_\bA$ of semilattice subalgebras of $\bA$ has at least three strongly connected components.
\end{cor}

\section{Inert dispersive algebras with small right orbits}\label{a-dispersive}

In this appendix we will prove Theorem \ref{inert-small}. Recall that we say that a dispersive algebra is \emph{inert} if every right orbit forms a projection subalgebra.

\begin{repthm}{inert-small} If $\bA$ is a clone-minimal inert dispersive algebra which is generated by two elements $a,b$, and if $|O(b)| \le 2$, then we also have $|O(a)| \le 2$ and $\bA$ is isomorphic to a subalgebra of $\cF_\cD(x,y)$.
\end{repthm}
\begin{proof} We will prove this by induction on $|\bA|$. By Proposition \ref{dispersive-nice-terms}, we may choose a nontrivial $f \in \Clo_2^{\pi_1}(\bA)$ such that
\begin{equation}
f(f(x,y),y) \approx f(x,y)\label{eq-inert-f}
\end{equation}
and $f(x,y)$ is contained in a maximal strongly connected component of the labeled digraph attached to $\cF_{\cV(\bA)}(x,y)$.

%In the case where $O(a)$ is strongly connected, we claim that we can choose this $f$ so that we additionally have $f(a,b) = a$. We will verify this claim by slightly modifying the construction of $g_1$ from Proposition \ref{nice-terms}.
{\bf Claim 0.} In the case where $O(a)$ is strongly connected, we can choose $f$ as above so that we additionally have $f(a,b) = a$.

{\bf Proof of claim 0.} Assume that $O(a)$ is strongly connected. We will verify this claim by slightly modifying the construction of $g_1$ from Proposition \ref{nice-terms}. As in Proposition \ref{nice-terms}, we start by picking some nontrivial $g_0 \in \Clo_2^{\pi_1}(\bA)$ such that $g_0(x,y)$ is maximal in the labeled digraph attached to the free algebra $\cF_{\cV(\bA)}(x,y)$. Since $g_0(a,b) \in O(a)$ by Proposition \ref{prop-right-pi1} and since $O(a)$ is strongly connected, we have $a \in O(g_0(a,b))$, so there are some $c_1, ..., c_n \in \bA$ such that
\[
a = g_0(\cdots g_0(g_0(a,b),c_1)\cdots,c_n).
\]
Choosing terms $t_i \in \Clo_2(\bA)$ such that $t_i(b,a) = c_i$, define the term $g_0'$ by
\[
g_0'(x,y) = g_0(\cdots g_0(g_0(x,y),t_1(y,x))\cdots,t_n(y,x)).
\]
Then we have $g_0(x,y) \preceq g_0'(x,y)$, so $g_0'(x,y)$ is still maximal in $\cF_{\cV(\bA)}(x,y)$, and we have
\[
g_0'(a,b) = a
\]
by construction. Now define $g_1' = g_0'^{\infty_1}$ as in Proposition \ref{nice-terms}, and note that we have $g_1'(a,b) = a$ and $g_1'(g_1'(x,y),y) \approx g_1'(x,y)$. That $g_1'$ is nontrivial is verified as in Proposition \ref{dispersive-nice-terms}, so taking $f = g_1'$ proves Claim 0.

We assume from here on that $\bA = (A,f)$, with $f$ is chosen as in claim 0 in the case where $O(a)$ is strongly connected. We now divide into many cases.

{\bf Case 1.} Suppose that $|O(b)| = 1$, i.e. that $O(b) = \{b\}$. Since $O(a)$ is a projection algebra, the directed edges of $O(a)$ must all be labeled by elements of $O(b) = \{b\}$, and the identity \eqref{eq-inert-f} implies that no pair of consecutive edges can both be labeled by $b$. Therefore there is at most one edge in $D_\bA$, and if $\bA$ is nontrivial we see that $O(a) = \{a,f(a,b)\}$ and that $\bA$ is isomorphic to the three-element subalgebra $\{x,xy,yx\}$ of $\cF_\cD(x,y)$.

{\bf Case 2.} Suppose that $|O(b)| = 2$, with $O(b) = \{b,d\}$. Then every edge of $D_\bA$ between elements of $O(a)$ is either labeled with $b$ or $d$ or both, and no pair of consecutive edges can both have the same label. Thus every vertex of $O(a)$ other than $a$ has outdegree at most one, and $a$ can only have outdegree two if $O(a)$ is not strongly connected. In particular, we see that $O(a)$ is strongly connected if and only if it is a directed cycle with edge labels alternating between $b$ and $d$. After proving that $O(a) \cap O(b) = \emptyset$ below, we will divide into subcases based on whether or not $O(a)$ and $O(b)$ are strongly connected.

{\bf Claim 1.} We have $O(a) \cap O(b) = \emptyset$.

{\bf Proof of claim 1.} Suppose for contradiction that $O(a) \cap O(b) \ne \emptyset$. We can't have $a$ or $b$ in $O(a) \cap O(b)$, since in that case $\{a,b\}$ would be a two-element projection algebra by the assumption that $\bA$ is inert. Then we must have $d \in O(a)$, so since $O(a)$ is a projection subalgebra we see that no edge between elements of $O(a)$ can be labeled with $d$. In this case we see that $O(a)$ contains at most one edge, necessarily going from $a$ to $d$ and labeled by $b$.

Similarly, since $|O(a)| = 2$ and no edge of $O(b)$ can be labeled by $d$, we see that $O(b)$ contains just one edge, necessarily from $b$ to $d$ and labeled by $a$. This completely describes $\bA = (\{a,b,d\},f)$, and we find that $f$ satisfies the identity
\[
f(x,f(y,x)) \approx x,
\]
so in fact $\bA$ is a three-element meld, contradicting our assumption that $\bA$ is dispersive. This contradiction proves Claim 1.

{\bf Subcase 2.1.} Suppose that $O(a)$ is not a cycle. Then for every $c \in O(a)\setminus\{a\}$, we have $a \not\in \Sg_\bA\{b,c,d\}$, so by the inductive hypothesis we see that $\Sg_\bA\{b,c\}$ and $\Sg_\bA\{c,d\}$ are isomorphic to subalgebras of $\cF_\cD(x,y)$. In particular, if one of $\Sg_\bA\{b,c\}$ or $\Sg_\bA\{c,d\}$ has size $4$, then $c$ must have an outgoing edge labeled by both $b$ and $d$, which contradicts the fact that $c$ must have an incoming edge labeled by either $b$ or $d$ (since $c \in O(a)\setminus\{a\}$).

{\bf Claim 2.} If $O(a)$ is not a cycle and there is some $c \in O(a)\setminus \{a\}$ such that any edge of $O(b)$ is labeled by $c$, then $c$ must have no outgoing edges.

{\bf Proof of Claim 2.} If there is a directed edge from $b$ to $d$ labeled by $c$, then $\Sg_\bA\{b,c\}$ contains $b,c,d$ and $f(c,b), f(c,d)$, so in order to have $|\Sg_\bA\{b,c\}| < 4$ we must have $f(c,b) = f(c,d) = c$. A similar argument applies to $\Sg_\bA\{c,d\}$ if there is a directed edge from $d$ to $b$ labeled by $c$.

Thus at most three edge labels can occur in $O(b)$: there are at most two elements of $O(a)$ with no outgoing edges (one per directed edge leaving $a$), and together with $a$ itself these are the possible labels of edges of $O(b)$. This information reduces the case where $O(a)$ is not a cycle to just a few subsubcases. In many of these subsubcases, $\Clo(\bA)$ will fail to be minimal by the following claim.

{\bf Claim 3.} If $|O(b)| = 2$ and there is some $c \in O(a)\setminus \{a\}$ such that $(c,b) \in \Sg_{\bA^2}\{(a,b),(b,a)\}$ with $f(c,b) = c$ and $f(b,c) = b$, then we have a contradiction to the assumption that $\Clo(\bA)$ is minimal.

{\bf Proof of Claim 3.} Pick some $g \in \Clo_2^{\pi_1}(\bA)$ satisfying $g(a,b) = c$ and $g(b,a) = b$. Then $\{a,b,c\}$ is closed under $g$: $\{b,c\}$ is closed under $g$ since it is closed under $f$, and $\{a,c\} \subseteq O(a)$ is closed under $g$ since it is a projection subalgebra. Since $g(a,b) \ne a,b$, we see that $g$ is nontrivial, so if $\Clo(\bA)$ is minimal then $f \in \Clo(g)$ and $\bA = \{a,b,c\}$. But then $O(b) = \{b\}$, contradicting our assumption that $|O(b)| = 2$.

{\bf Subsubcase 2.1.0.} If $|O(a)| = 1$, then by swapping $a$ and $b$ we are in Case 1.

{\bf Subsubcase 2.1.1.} Suppose there is a directed edge leaving $a$ labeled by both $b$ and $d$. In this case, we have $O(a) = \{a,c\}$ where $c = f(a,b) = f(a,d)$. If $f(b,c) = b$, then we have
\[
\begin{bmatrix} c\\ b\end{bmatrix} = f\Big(\begin{bmatrix} a\\ b\end{bmatrix}, f\Big(\begin{bmatrix} b\\ a\end{bmatrix},\begin{bmatrix} a\\ b\end{bmatrix}\Big)\Big) \in \Sg_{\bA^2}\Big\{\begin{bmatrix} a\\ b\end{bmatrix},\begin{bmatrix} b\\ a\end{bmatrix}\Big\},
\]
which gives us a contradiction by Claim 3. Thus we must have $f(b,c) = d$, and in this case the identity \eqref{eq-inert-f} implies that $f(d,c) = d$ as well. Then if we define the operation
\[
g(x,y) = f(x,f(y,x)),
\]
we see from $f(O(a),O(b)) = \{c\}$ and $f(O(b),c) = \{d\}$ that
\[
g(O(a),O(b)) \subseteq f(O(a), O(b)) \subseteq \{c\}
\]
and
\[
g(O(b),O(a)) \subseteq f(O(b),f(O(a),O(b))) \subseteq f(O(b), \{c\}) = \{d\},
\]
so $(\{a,b,c,d\},g)$ is isomorphic to $\cF_\cD(x,y)$. Since $g$ is nontrivial, we have $f \in \Clo_2^{\pi_1}(g)$, so in fact $f = g$.

{\bf Subsubcase 2.1.2.} Suppose that there is only one directed edge leaving $a$, labeled only by $d$, so that $f(a,b) = a$. Let $c = f(a,d)$. Since $\{a,b\}$ is not closed under $f$, we must have $f(b,a) = d$. If $f(b,c) = d$, then there are no outgoing edges from $c$ by Claim 2, so $|O(a)| = 2$, and since the edge from $b$ to $d$ is labeled by both $a$ and $c$ we can swap $a$ with $b$ and $c$ with $d$ to find ourselves in Subsubcase 2.1.1, which we have already handled.

Thus we may assume that $f(b,c) = b$. Then if we define $g$ by
\[
g(x,y) = f(x,f(y,f(x,f(y,x)))),
\]
we find that
\[
g(a,b) = f(a,f(b,f(a,f(b,a)))) = f(a,f(b,f(a,d))) = f(a,f(b,c)) = f(a,b) = a
\]
and
\[
g(b,a) = f(b,f(a,f(b,f(a,b)))) = f(b,f(a,f(b,a))) = f(b,f(a,d)) = f(b,c) = b,
\]
while $g(x,y) \approx x$ does not follow from $\eqref{D2}$ since $x(y(x(yx))) = xy$ in $\cF_\cD(x,y)$, and by Proposition \ref{dispersive-hom} this contradicts the assumption that $\bA$ dispersive.
{\bf Subsubcase 2.1.3.} Suppose that $f(a,b) = c \ne a$ and $f(b,c) = b$. By \eqref{eq-inert-f} we have $f(c,b) = c$, so in this case $f$ acts as first projection on $\{b,c\}$. If we had $f(b,a) = b$, then we would have
\[
\begin{bmatrix} c\\ b\end{bmatrix} = f\Big(\begin{bmatrix} a\\ b\end{bmatrix}, \begin{bmatrix} b\\ a\end{bmatrix}\Big) \in \Sg_{\bA^2}\Big\{\begin{bmatrix} a\\ b\end{bmatrix},\begin{bmatrix} b\\ a\end{bmatrix}\Big\},
\]
which would give us a contradiction by Claim 3. Thus we must have $f(b,a) = d$, and $f(d,a) = d$ by \eqref{eq-inert-f}.

Set $e = f(a,d)$. Defining $g(x,y) = f(x,f(y,x))$, we have
\[
g(b,a) = f(b,f(a,b)) = f(b,c) = b
\]
and
\[
g(a,b) = f(a,f(b,a)) = f(a,d) = e.
\]
Then by Proposition \ref{dispersive-hom} we see that if $\bA$ is dispersive then we must have $g(a,b) \ne a$, so $e \ne a$. Additionally, we have
\[
g(a,d) = f(a,f(d,a)) = f(a,d) = e,
\]
so if $e$ had no outgoing edges then $\{a,b,d,e\}$ would be closed under $g$, contradicting the minimality of $\Clo(f)$.

Thus by Claim 2 we must have $f(b,e) = b$ and $f(d,e) = d$. Setting $i = f(e,b)$, we have $f(i,b) = i$ by \eqref{eq-inert-f} and
\[
\begin{bmatrix} i\\ b\end{bmatrix} = f\Big(g\Big(\begin{bmatrix} a\\ b\end{bmatrix}, \begin{bmatrix} b\\ a\end{bmatrix}\Big), g\Big(\begin{bmatrix} b\\ a\end{bmatrix}, \begin{bmatrix} a\\ b\end{bmatrix}\Big)\Big) \in \Sg_{\bA^2}\Big\{\begin{bmatrix} a\\ b\end{bmatrix},\begin{bmatrix} b\\ a\end{bmatrix}\Big\},
\]
so if we had $f(b,i) = b$ then we would get a contradiction by Claim 3. Thus we must have $f(b,i) = d$ and $i \ne e$. Now we have
\[
\{b,d,e,i\} \in \Sg_\bA\{b,e\} \subseteq \bA\setminus \{a\},
\]
so by the inductive hypothesis $\Sg_\bA\{b,e\}$ must be isomorphic to $\cF_\cD(x,y)$, which gives us a contradiction since $e$ does not have an outgoing edge with two labels.
\begin{comment}
 (and $f(d,i) = d$ by \eqref{eq-inert-f}), and by Claim 2 we see that $i$ has no outgoing edges. Defining $h(x,y)$ by
\[
h(x,y) = f(g(x,y),g(y,x)),
\]
we have
\begin{align*}
h(a,b) &= f(g(a,b),g(b,a)) = f(e,b) = i,\\
h(a,d) &= f(g(a,d),g(d,a)) = f(e,)
\end{align*}
\end{comment}

\begin{comment}
acceei
b?ddbb

ace
bdd
\end{comment}

{\bf Subsubcase 2.1.4.} Since we are not in one of the previous subsubcases, we must have $f(a,b) = c \ne a$ and $f(b,c) = d$. By Claim 2, $c = f(a,b)$ must not have any outgoing edges. If $f(a,d) = c$, then we are in Subsubcase 2.1.1.

Suppose now that $f(a,d) = a$, then $O(a) = \{a,c\}$ and $\bA = \{a,b,c,d\}$. Then since $f(a,d) = a$ we see that $c \not\in \Sg_\bA\{a,d\}$ by the assumption that $f(x,y)$ is maximal in $\cF_{\cV(\bA)}(x,y)$, so we must have $f(d,a) = d$, and $d$ has no outgoing edges. Swapping $a$ with $b$ and $c$ with $d$, we find ourselves in Subsubcase 2.1.2 or 2.1.3 depending on whether or not $f(b,a) = b$.

Thus we may assume that $e = f(a,d) \not\in \{a,c\}$. Regardless of whether we have $f(b,a) = b$ or $f(b,a) = d$, either way we have
\[
\begin{bmatrix} c\\ d\end{bmatrix} = f\Big(f\Big(\begin{bmatrix} a\\ b\end{bmatrix}, \begin{bmatrix} b\\ a\end{bmatrix}\Big), f\Big(\begin{bmatrix} b\\ a\end{bmatrix}, \begin{bmatrix} a\\ b\end{bmatrix}\Big)\Big) \in \Sg_{\bA^2}\Big\{\begin{bmatrix} a\\ b\end{bmatrix},\begin{bmatrix} b\\ a\end{bmatrix}\Big\},
\]
so
\[
\begin{bmatrix} e\\ d\end{bmatrix} = f\Big(\begin{bmatrix} a\\ b\end{bmatrix}, \begin{bmatrix} d\\ c\end{bmatrix}\Big) \in \Sg_{\bA^2}\Big\{\begin{bmatrix} a\\ b\end{bmatrix},\begin{bmatrix} b\\ a\end{bmatrix}\Big\}.
\]

Suppose first that we have $f(b,e) = b$, and let $f(e,b) = i$ (in which case we also have $f(i,b) = i$ by \eqref{eq-inert-f}). Then we have
\[
\begin{bmatrix} i\\ b\end{bmatrix} = f\Big(f\Big(\begin{bmatrix} a\\ b\end{bmatrix}, \begin{bmatrix} d\\ e\end{bmatrix}\Big), f\Big(\begin{bmatrix} b\\ a\end{bmatrix}, \begin{bmatrix} e\\ d\end{bmatrix}\Big)\Big) \in \Sg_{\bA^2}\Big\{\begin{bmatrix} a\\ b\end{bmatrix},\begin{bmatrix} b\\ a\end{bmatrix}\Big\},
\]
so if $f(b,i) = b$ then we would get a contradiction to Claim 3. Thus we must have $f(b,i) = d$ and $i \ne e$, so we have
\[
\{b,d,e,i\} \in \Sg_\bA\{b,e\} \subseteq \bA\setminus \{a\},
\]
so by the inductive hypothesis $\Sg_\bA\{b,e\}$ must be isomorphic to $\cF_\cD(x,y)$, which gives us a contradiction since $e$ does not have an outgoing edge with two labels.

Thus we must have $f(b,e) = d$ and by Claim 2 we see that $e$ has no outgoing edges. Then we see that $\bA$ has a congruence $\theta$ with congruence classes $\{a\},\{b\},\{c,e\},\{d\}$, so by the induction hypothesis applied to $\bA/\theta$ we see that $\bA/\theta$ is isomorphic to $\cF_\cD(x,y)$, so $f(b,a) = f(d,a) = d$. Then if we let $g(x,y) = f(x,f(y,x))$, we have
\[
g(a,b) = f(a,f(b,a)) = f(a,d) = e
\]
and
\[
g(a,d) = f(a,f(d,a)) = f(a,d) = e,
\]
so $\{a,b,d,e\}$ is closed under the nontrivial operation $g$, contradicting the minimality of $\Clo(\bA)$.

{\bf Subcase 2.2.} Now suppose that $O(a)$ is a cycle and $O(b)$ is not. In this case, for every $c \in O(a)\setminus\{a\}$, we have
\[
\Sg_\bA\{c,d\} \subseteq \bA \setminus \{b\},
\]
so $\Sg_\bA\{c,d\}$ is isomorphic to a subalgebra of $\cF_\cD(x,y)$ size at most $3$ by the inductive hypothesis. From the extra assumption $f(a,b) = a$ in this case, we see there must be some $c \ne a$ such that $f(c,b) = a$. Then we have $\Sg_\bA\{b,c\} = \bA$ and $\{c,d\}$ is a subalgebra of $\bA$ since the incoming edge to $c$ must be labeled by $d$.

Letting $e \in O(a)$ be such that $f(e,d) = c$, from $\Sg_\bA\{b,c\} = \bA$ we see that there are some terms $t,u \in \Clo_2^{\pi_1}(\bA)$ such that $t(c,b) = e$ and $u(b,c) = d$. Writing
\[
g(x,y) = f(t(x,y),u(y,x)),
\]
we see that $g(x,y) = xy$ in $\cF_\cD(x,y)$, so by Proposition \ref{dispersive-hom} $g$ must be nontrivial, so $f \in \Clo(g)$. Since we have
\[
g(c,b) = c,
\]
we see that $\{b,c,d\}$ is closed under $g$ and hence also closed under $f$, contradicting $f(c,b) = a$.

{\bf Subcase 2.3.} The only remaining possibility is that $O(a)$ and $O(b)$ are both cycles.

{\bf Claim 4.} If $O(a)$ and $O(b)$ are both cycles, then for any $c \in O(a)$, we can't have both $\Sg_\bA\{b,c\} = \bA$ and $\Sg_\bA\{c,d\} = \bA$.

{\bf Proof of Claim 4.} Note that if $\Sg_\bA\{b,c\} = \bA$, then since $O(a)$ is a cycle, there must be a nontrivial binary term $g$ such that $g(c,b) = c$, and from $\Sg_\bA\{b,c\} \ne \{b,c\}$ we must then have $g(b,c) = d$. A similar argument shows that there is a nontrivial binary term $h$ such that $h(c,d) = c$ and $h(d,c) = b$, but then the term $h(g(x,y), g(y,x))$ is also nontrivial and $\{b,c\}$ is closed under it, a contradiction.

The strategy now is to show that for strictly more than half of the elements $c \in O(a)$, we have $\Sg_\bA\{b,c\} = \bA$, and similarly for $d$, which will give us a contradiction to Claim 4.

Recall that since $O(a)$ is a cycle, we can use Claim 0 to make the further assumption that $f(a,b) = a$. Then we must have $f(b,a) = d$, and then by \eqref{eq-inert-f} we also have $f(d,a) = d$. If we label the elements of $O(a)$ in cyclic order as $c_0 = a, c_1, ..., c_{n-1}$, then we see that for each $i$ we have an edge from $c_i$ to $c_{i+1}$ in $D_\bA$, labeled by $d$ if $i$ is even and labeled by $b$ if $i$ is odd. Then we can see that
\[
\begin{bmatrix} a\\ d\end{bmatrix} = f\Big(\begin{bmatrix} a\\ b\end{bmatrix}, \begin{bmatrix} b\\ a\end{bmatrix}\Big) \in \Sg_{\bA^2}\Big\{\begin{bmatrix} a\\ b\end{bmatrix},\begin{bmatrix} b\\ a\end{bmatrix}\Big\}.
%(a,d) \in \Sg_{\bA^2}\{(a,b),(b,a)\},
\]
We prove by induction on $i$ that we have
\[
\begin{bmatrix} c_i\\ d\end{bmatrix} \in \Sg_{\bA^2}\Big\{\begin{bmatrix} a\\ b\end{bmatrix},\begin{bmatrix} b\\ a\end{bmatrix}\Big\}.
%(c_i,d) \in \Sg_{\bA^2}\{(a,b),(b,a)\}
\]
For even $i$, we have
\[
\begin{bmatrix} c_{i+1}\\ d\end{bmatrix} = f\Big(\begin{bmatrix} c_i\\ d\end{bmatrix}, \begin{bmatrix} d\\ a\end{bmatrix}\Big) \in \Sg_{\bA^2}\Big\{\begin{bmatrix} a\\ b\end{bmatrix},\begin{bmatrix} b\\ a\end{bmatrix}\Big\},
%f((x_i,d),(d,a)) = (x_{i+1},d),
\]
while for odd $i$ we have
\[
\begin{bmatrix} c_{i+1}\\ d\end{bmatrix} = f\Big(\begin{bmatrix} c_i\\ d\end{bmatrix}, \begin{bmatrix} b\\ a\end{bmatrix}\Big) \in \Sg_{\bA^2}\Big\{\begin{bmatrix} a\\ b\end{bmatrix},\begin{bmatrix} b\\ a\end{bmatrix}\Big\}.
%f((x_i,d),(b,a)) = (x_{i+1},d).
\]
Also, since $O(b)$ is strongly connected, there must be some $j$ such that $f(d,c_j) = b$. But then from
\[
\begin{bmatrix} c_j\\ d\end{bmatrix} \in \Sg_{\bA^2}\Big\{\begin{bmatrix} a\\ b\end{bmatrix},\begin{bmatrix} b\\ a\end{bmatrix}\Big\},
%(x_j,d) \in \Sg_{\bA^2}\{(a,b),(b,a)\},
\]
for any odd $i$ we have
\[
\begin{bmatrix} c_i\\ b\end{bmatrix} = f\Big(\begin{bmatrix} c_i\\ d\end{bmatrix}, \begin{bmatrix} d\\ c_j\end{bmatrix}\Big) \in \Sg_{\bA^2}\Big\{\begin{bmatrix} a\\ b\end{bmatrix},\begin{bmatrix} b\\ a\end{bmatrix}\Big\},
%(x_i,b) = f((x_i,d), (d,x_j)) \in \Sg_{\bA^2}\{(a,b),(b,a)\}.
\]

Suppose for a contradiction that there is any odd $k$ such that $f(b,c_k) = b$. Then we have
\[
\begin{bmatrix} c_{i+1}\\ b\end{bmatrix} = f\Big(\begin{bmatrix} c_i\\ b\end{bmatrix}, \begin{bmatrix} b\\ c_k\end{bmatrix}\Big) \in \Sg_{\bA^2}\Big\{\begin{bmatrix} a\\ b\end{bmatrix},\begin{bmatrix} b\\ a\end{bmatrix}\Big\},
%f((x_i,b),(b,x_k)) = (x_{i+1},b)
\]
for every odd $i$, so in fact we have $O(a) \times O(b) \subseteq \Sg_{\bA^2}\{(a,b),(b,a)\}$, so there is a nontrivial term $g$ with $g(a,b) = a$ and $g(b,a) = b$, a contradiction.

Thus for every odd $i$ we must have $f(b,c_i) = d$, so $\Sg_\bA\{c_i,b\} = \bA$ for every odd $i$ as well as for $i = 0$. This shows that strictly more than half of the elements $c \in O(a)$ have $\Sg_\bA\{b,c\} = \bA$.

To see that a least half of the $c$s in $O(a)$ satisfy $\Sg_\bA\{c,d\} = \bA$, first note that by $f(d,c_j) = b$ we have $\Sg_\bA\{c_j,d\} = \bA$. Then we can swap $c_j$ for $a$ and $d$ for $b$, and apply Claim 0 to find a nontrivial $f' \in \Clo_2^{\pi_1}$ such that $f'(f'(x,y),y) \approx f'(x,y)$ and $f'(c_j,d) = c_j$, and then we use the previous argument to see that more than half of the $c$s in $O_{f'}(c_j)$ satisfy $\Sg_{(A,f')}\{c,d\} = (A,f')$. Since $\Clo(f) = \Clo(f')$, this completes the proof.
\end{proof}

\end{document}